\documentclass[11pt, twoside]{article}
\usepackage{amssymb}
\usepackage{amsfonts}
\usepackage{amsmath}
\usepackage{amsthm}
\usepackage{color}
\usepackage{mathrsfs}
\usepackage{txfonts}
\usepackage{enumerate}
\usepackage{indentfirst}
\usepackage{graphicx}

\usepackage{citeref}

\usepackage{anysize}

\allowdisplaybreaks

\newtheorem{theorem}{Theorem}[section]
\newtheorem{lemma}[theorem]{Lemma}

\newtheorem{proposition}[theorem]{Proposition}

\theoremstyle{definition}

\newtheorem{remark}[theorem]{Remark}
\newtheorem{definition}[theorem]{Definition}

\newcounter{assum}

\renewcommand{\appendix}{\par
\setcounter{section}{0}%
\setcounter{subsection}{0}%
\setcounter{subsubsection}{0}%
\gdef\thesection{\@Alph\c@section}%
\gdef\thesubsection{\@Alph\c@section.\@arabic\c@subsection}%
\gdef\theHsection{\@Alph\c@section.}%
\gdef\theHsubsection{\@Alph\c@section.\@arabic\c@subsection}%
\csname appendixmore\endcsname
}

\numberwithin{equation}{section}

\begin{document}

\arraycolsep=1pt

\title{\bf\Large
A Sharp Localized Weighted  Inequality Related to Gagliardo and
Sobolev Seminorms and Its Applications
\footnotetext{\hspace{-0.35cm} 2020
\emph{Mathematics Subject Classification}. Primary 26D10; Secondary
46E35, 42B25, 42B35.
\endgraf \emph{Key words and phrases}.
Muckenhoupt weight, Gagliardo Seminorm,
Sobolev Seminorm, Gagliardo--Nirenberg interpolation,
Bourgain--Brezis--Mironescu formula, ball Banach function
space.
\endgraf
This project is partially supported by the National
Key Research and Development Program of China
(Grant No.\ 2020YFA0712900),
the National Natural Science Foundation of
China (Grant Nos.\ 12431006, 12371093, and 124B2004),
the Fundamental Research Funds
for the Central Universities (Grant Nos.\ 2253200028 and 2233300008).
}}
\author{Pingxu Hu, Yinqin Li, Dachun Yang\footnote{Corresponding
author, E-mail: \texttt{dcyang@mail.bnu.edu.cn} /{\color{red}
\today}/Final version.}
\ and Wen Yuan}
\date{}
\maketitle

\vspace{-0.8cm}

\begin{center}
\begin{minipage}{13cm}
{\small {\bf Abstract}\quad
In this article, we establish a nearly sharp
localized weighted inequality related to Gagliardo
and Sobolev seminorms, respectively,
with the sharp $A_1$-weight constant or with
the specific $A_p$-weight constant when $p\in (1,\infty)$.
As applications, we further obtain a new characterization of
Muckenhoupt weights and, in the framework of
ball Banach function spaces,
an inequality related to Gagliardo and Sobolev
seminorms on cubes, a Gagliardo--Nirenberg interpolation inequality,
and a Bourgain--Brezis--Mironescu formula.
All these obtained results have wide generality and
are proved to be (nearly) sharp.

\emph{The original version of this article was published in
[Adv. Math. 481 (2025), Paper No. 110537]. 
In this revised version, we correct an error appeared 
in Theorem \ref{thm-upi} in the case where $p=1$,
which was pointed out to us by Emiel Lorist.}}
\end{minipage}
\end{center}

%
\tableofcontents
%

\section{Introduction}

Let $\Omega\subset \mathbb{R}^n$ be an open set
and $L^1_{\rm loc}(\Omega)$ denote the set of all
locally integrable functions on $\Omega$.
Recall that, for any $p\in [1,\infty)$, the \emph{homogeneous Sobolev space}
$\dot{W}^{1,p}(\Omega)$ and the \emph{homogeneous fractional
Sobolev space} $\dot{W}^{s,p}(\Omega)$ for any $s\in (0,1)$ are defined,
respectively, by setting
\begin{align*}
\dot{W}^{1,p}(\Omega):&
= \left\{f\in  L^1_{\rm loc}(\Omega):
|\nabla f|\in L^p(\Omega)\right\}
\end{align*}
and
\begin{align*}
\dot{W}^{s,p}(\Omega):
&= \left\{f\in L^1_{\rm loc}(\Omega): \|f\|_{\dot{W}^{s,p}(\Omega)}
:=\left[\int_{\Omega}
\int_{\Omega}\frac{|f(x)-f(y)|^p}{|x-y|^{n+sp}}\,dy\,dx
\right]^{\frac{1}{p}}<\infty\right\},
\end{align*}
where $\nabla f$ denotes the weak gradient of $f$.
It is well known that, if
$\Omega$ is a smooth bounded domain,
then, for any $p\in[1,\infty)$ and $s\in(0,1)$,
$\dot{W}^{1,p}(\Omega)\subset\dot{W}^{s,p}(\Omega)$;
that is, for any $f\in\dot{W}^{1,p}(\Omega)$,
$\|f\|_{\dot{W}^{s,p}(\Omega)}\lesssim\|\,|\nabla f|\,\|_{L^p(\Omega)}$
with the implicit positive constant depending only on
$n$, $s$, $p$, and $\Omega$ (see, for instance,
\cite[Proposition 2.2]{DPV12} or \cite[Theorem 6.21]{Leo23}).
An improved quantitative variant of this embedding was established by
Bourgain, Brezis, and Mironescu \cite[Theorem 1]{BBM01}, whose
special case is that, for any given
cube $Q\subset\mathbb{R}^n$ and any given $p\in[1,\infty)$ and
for any $s\in(0,1)$
and $f\in \dot{W}^{1,p}(Q)$,
\begin{align}\label{eq-sss03}
(1-s)^{\frac{1}{p}}
\left\{
\int_{Q}\int_{Q}
\frac{|f(x)-f(y)|^p}{|x-y|^{n+sp}}\,dy
\,dx\right\}^{\frac{1}{p}}
\lesssim
[\ell (Q)]^{1-s}
\left[\int_{Q}\left|\nabla f(x)\right|^p\,dx\right]^{\frac{1}{p}}
\end{align}
with the implicit positive constant depending only on $n$ and $p$,
here and thereafter,
any \emph{cube} $Q$ has finite \emph{edge length}
and all its edges parallel to the coordinate
axes and $\ell (Q)$ denotes its edge length. The factor $(1-s)^{\frac1p}$
in \eqref{eq-sss03} is \emph{optimal}, as evidenced by the limiting behavior
(see \cite[Theorems 2 and 3]{BBM01}):
For any given $p\in[1,\infty)$ and for any $f\in \dot{W}^{1,p}(Q)$,
$$\lim_{s\to 1^-}(1-s)^{\frac{1}{p}}\|f\|_{\dot{W}^{s,p}(Q)}
=C_{(n,p)}\|\,|\nabla f|\,\|_{L^p(Q)}$$
with $C_{(n,p)}\in(0,\infty)$ depending only on $n$ and $p$, where
$s\to 1^-$ means $s\in(0,1)$ and $s\to 1$.
This seminal result,
now called the Bourgain--Brezis--Mironescu \emph{formula}
(for short, BBM \emph{formula}), combined with
the above quantitative inequality \eqref{eq-sss03},
has inspired extensive developments and applications;
we refer to \cite{BIK13,BBM02,Dav02,DLTYY24,DLTYY25,DM23,KMX05,Lud14,Pon04jems}
for the studies in Euclidean spaces,
to \cite{DD22,Moh24} for the studies in domains,
and to \cite{Han24} for the studies in metric spaces.

However, inequality \eqref{eq-sss03}
fails if we replace $Q$ by $\mathbb{R}^n$.
Indeed, by a standard scaling argument, we
can find that, for any $s\in (0,1)$ and $p\in[1,\infty)$,
$\dot{W}^{1,p}\nsubseteq\dot{W}^{s,p}$
(see, for instance, \cite[(1.7)]{Tre15}).
Here, and thereafter, \emph{unless otherwise specified, we take $\mathbb{R}^n$
as the default underlying space.}
Therefore, the local properties
of cubes play an important role in inequality \eqref{eq-sss03}.
Recently, to obtain the BBM formula in arbitrary
domains, Mohanta \cite{Moh24}
introduced a new \emph{localized} Gagliardo seminorm
including an additional exponent $q$
(see \cite[(5)]{Moh24}) (which we call the \emph{Triebel--Lizorkin
exponent}, for the sake of presentation) and established
a variant of \eqref{eq-sss03}
in arbitrary domains via this
seminorm.
In particular, when applied to $\mathbb{R}^n$,
his results can be stated as follows (see \cite[Lemmas 11 and 12]{Moh24}):
If $p=q=1$ or $p\in(1,\infty)$ and $q\in [1,\infty)$ satisfying
$n(\frac1p-\frac1q)<1$, then, for any
$s\in (0,1)$, $r\in (0,\infty)$,
and $f\in \dot{W}^{1,p}$,
\begin{align}\label{eq-s04}
(1-s)^{\frac{1}{q}}
\left\{
\int_{\mathbb{R}^n}
\left[\int_{B({\bf 0},r)}
\frac{|f(x+h)-f(x)|^q}{|h|^{n+sq}}\,dh\right]^{\frac pq}
\,dx\right\}^{\frac{1}{p}}
\lesssim
r^{1-s}
\left\|\,|\nabla f|\,\right\|_{L^p}.
\end{align}
with the implicit positive constant depending only on $n$, $p$,
and $q$. In this article, we give a \emph{nearly sharp} weighted
generalization of this inequality in higher-order Sobolev spaces
with the sharp $A_1$-weight constant or with
the specific $A_p$-weight constant when $p\in (1,\infty)$ (see Theorem \ref{thm-upi}
and Remark \ref{re12}). Using it, we further obtain a new characterization
of Muckenhoupt weights (see Theorem \ref{thm-ap}) and,
in a general function space framework
(that is, the ball Banach function space; see Definition \ref{def-X}
for its definition), an inequality related to Gagliardo and Sobolev seminorms
on cubes (see Theorems \ref{thm-Q} and \ref{thm-QX}),
a fractional Gagliardo--Nirenberg interpolation inequality (see Theorems \ref{thm-ex}
and \ref{thm-in}), and a BBM formula (see Theorems
\ref{thm-wls}, \ref{thm-main}, and \ref{thm-cwkx}).
All these obtained results have wide generality (see Section \ref{sec-app})
and, through the flexible combination of
bump functions and power weights,
we constructively prove that these results are also the (nearly) sharp
generalizations of their corresponding known counterparts
(see Section \ref{sec-q}).

Next, we state our main results.
For any $k\in \mathbb{N}$ and $h\in \mathbb{R}^n$,
the {\emph{$k$-{\rm th} order
difference}} $\Delta^k_h f$ of
any measurable function $f$ on $\mathbb{R}^n$
is defined by setting
\begin{align}\label{eq-sho-hd}
\Delta_{h}^k
f(\cdot):=
\sum_{j=0}^{k}(-1)^{k-j}\binom{k}{j}f(\cdot+jh),
\end{align}
where $\binom{k}{j}:=\frac{k!}{j!(k-j)!}$.
Let $p\in [1,\infty)$
and $\omega\in A_p$ (the Muckenhoupt $A_p$-weight;
see Definition \ref{def-Ap} for its definition). We let
\begin{align}\label{eq-rw}
p_{\omega}:= \inf\{r\in [1,\infty):\omega\in A_r\}
\end{align}
and use $\dot{W}^{k,p}_\omega$ to denote the
homogeneous weighted Sobolev space
[see Definition \ref{def-Ap}(iv) for its definition].
For any $p,q\in [1,\infty)$, let
\begin{align}\label{Gqp}
	\gamma_{p,q}:=
	\begin{cases}
		1&\text{if}\ p=1,\\
		\frac{1}{q}&\text{if}\ p\in(1,\infty).
	\end{cases}
\end{align}
Then we obtain the following nearly sharp weighted variant
of \eqref{eq-s04}.

\begin{theorem}\label{thm-upi}
Let $k\in \mathbb{N}$ and $p\in[1,\infty)$.
\begin{enumerate}[{\rm (i)}]
\item
 If $ \omega \in A_1$ and $q\in [1,\infty)$
satisfies $n(\frac{1}{p}-\frac{1}{q})<k$,
then
there exists a positive constant
$C$, depending only on
$n$, $k$, $p$, and $q$,
such that, for any
$s\in (0,1)$, $r\in (0,\infty)$,
and $f\in \dot{W}^{k,p}_{\omega}$,
\begin{align}\label{eq-sss00}
	(1-s)^{\gamma_{p,q}}
	\left\{
	\int_{\mathbb{R}^n}
	\left[\int_{B({\bf 0},r)}
	\frac{|\Delta^k_h f(x)|^q}{|h|^{n+skq}}\,dh
	\right]^\frac{p}{q}\omega(x)\,dx\right\}^{\frac{1}{p}}
	\le
	C
	[\omega]_{A_{1}}
	^{\frac{1}{p}}
	r^{(1-s)k}
	\left\|\,\left|\nabla^k f\right|\,\right\|_{L^p_{\omega}},
\end{align}
where
$\nabla^k f:=\{\partial^{\alpha}f:\alpha \in \mathbb{Z}_+^n ,\
|\alpha|=k\}$ (the higher-order weak derivatives of $f$)
and
\begin{align*}
\left|\nabla^k f\right|
:= \left(\sum_{\alpha\in \mathbb{Z}_+^n,\,
|\alpha|=k}\left|\partial^{\alpha}
f\right|^2\right)^{\frac{1}{2}};
\end{align*}
moreover, the exponent
$\frac{1}{p}$ on the $A_1$-weight constant
$[\omega]_{A_1}$ is sharp.

\item  If $p\in (1,\infty)$,
$\omega\in A_p$ with $p_{\omega}$ as in \eqref{eq-rw},
and $q\in [1,\infty)$
satisfies $n(\frac{p_\omega}{p}-\frac{1}{q})<k$,
then there exists a positive constant
$C$ such that, for any $s\in (0,1)$, $r\in (0,\infty)$,
and $f\in \dot{W}^{k,p}_{\omega}$,
\begin{align}\label{eq-sss}
(1-s)^{\frac{1}{q}}
\left\{
\int_{\mathbb{R}^n}
\left[\int_{B({\bf 0},r)}
\frac{|\Delta^k_h f(x)|^q}{|h|^{n+skq}}\,dh
\right]^\frac{p}{q}\omega(x)\,dx\right\}^{\frac{1}{p}}
\le
C r^{(1-s)k}
\left\|\,\left|\nabla^k f\right|\,\right\|_{L^p_{\omega}},
\end{align}
where $C$ is of the form
$C_{(n,k,p,q,\beta)}[\omega]_{A_\beta}^{\frac{\beta}{p(\beta-1)}}$
with some $\beta \in (p_\omega,p)$ and
some positive constant $C_{(n,k,p,q,\beta)}$ depending only on
$n$, $k$, $p$, $q$, and $\beta$.
\end{enumerate}
\end{theorem}

\begin{remark}\label{re12}
\begin{enumerate}[{\rm (i)}]
\item By the open property of $A_p$-weights
with $p\in(1,\infty)$ (see, for instance,
\cite[Corollary 7.6(i)]{Duo01}), we conclude
the existence of $\beta$ in Theorem \ref{thm-upi}(ii).

\item When $p=q=1$ or $p\in(1,\infty)$,
Theorem \ref{thm-upi} with $\omega\equiv1$
and $k=1$ reduces to \eqref{eq-s04} and, in other cases,
is new.

\item In Theorem \ref{thm-upi},
the exponents $\gamma_{p,q}$ and $\frac1q$
on $1-s$ are both \emph{sharp}; see Proposition \ref{gammaqp}.

\item The ranges of the Triebel--Lizorkin exponent $q$ in
Theorem \ref{thm-upi} are \emph{nearly sharp};
see Proposition \ref{pro-q}.

\item It is still unknown whether the
dependence of the constant in \eqref{eq-sss} on the $A_p$-weight constant
is sharp.
\end{enumerate}
\end{remark}

Using Theorem \ref{thm-upi}, we further obtain the following four applications.

\textbf{Application (I): Characterization of Muckenhoupt Weights.}
As in \cite{LYYZZ-arXiv,ZLYYZ24}, we find that the condition
$\omega\in A_p$ in Theorem \ref{thm-upi}
is not only sufficient but also necessary, which
provides a new characterization of Muckenhoupt weights.
In what follows, for any $p\in [1,\infty)$ and any nonnegative locally
integrable function $\omega$ on
$\mathbb{R}^n$,
$\dot{Y}^{1,p}_{\omega}$
denotes the homogeneous
weighted Sobolev space [see \eqref{eq-Yd}
for its definition].

\begin{theorem}\label{thm-ap}
If $\omega$ is a nonnegative locally
integrable function on $\mathbb{R}^n$ and if
$p\in [1,\infty)$,
then the following statements are equivalent.
\begin{enumerate}[{\rm (i)}]
\item $\omega\in A_p$.

\item There exist $s\in(0,1)$, $q\in[1,\infty)$, and
a positive constant $C$ such that,
for any $r\in(0,\infty)$ and
$f\in \dot{Y}^{1,p}_{\omega}$,
\begin{align}\label{eq-chap1}
\left\{
\int_{\mathbb{R}^n}
\left[\int_{B({\bf 0},r)}
\frac{| f(x+h)-f(x)|^q}{|h|^{n+sq}}\,dh
\right]^\frac{p}{q}\omega(x)\,dx\right\}^{\frac{1}{p}}\le
C r^{(1-s)}
\left\|f\right\|_{\dot{ Y}^{1,p}_{\omega}}.
\end{align}
\end{enumerate}
\end{theorem}

\textbf{Application (II):
An Inequality Related to Gagliardo and Sobolev Seminorms.}
Note that, recently, a nice weighted variant of \eqref{eq-sss03}
with specific Muckenhoupt $A_1$-weights was obtained in
\cite[Corollary 2.2]{HMPV23}.
By combining Theorem \ref{thm-upi},
the extension theorem of integer-order Sobolev spaces,
and the extrapolation,
we obtain a nearly sharp variant of this result
in the framework of higher-order
Sobolev spaces and ball Banach function spaces (see Theorem \ref{thm-QX}).
Recall that the concept of ball
Banach function spaces
(see Definition \ref{def-X} for its definition) was
introduced by Sawano et al.\ \cite{SHYY17}
to unify the study of several different important function spaces,
such as Lebesgue spaces, weighted Lebesgue spaces,
Morrey spaces,
and all the function spaces in Section \ref{sec-app}.
Due to its wide generality, ball
Banach function spaces have recently attracted a lot of
attention
and yielded many applications; we refer to
\cite{LYH22,YHYY22ams,
YHYY22mn,YYY20,ZHYY22} for the real-variable
theory of function spaces and to
\cite{HCY21,IS17,TYYZ21,WYY20,WYYZ21,ZYYW21}
for the boundedness of operators on various function spaces.

In what follows,
for any open set $\Omega\subset \mathbb{R}^n$,
any $x\in \Omega$, and any $k\in \mathbb{N}$, let
\begin{align*}
\Omega(x,k):= \left\{h\in \mathbb{R}^n:
\text{for any}\ i\in \{1,\dots,k\},\ x+ih\in \Omega\right\}
\end{align*}
and, for any $p\in[1,\infty)$
and any nonnegative locally integrable function $\omega$
on $\mathbb{R}^n$, denote by $\dot{W}^{k,p}_\omega(\Omega)$
[resp.\,$W^{k,p}_\omega(\Omega)$]
the homogeneous [resp.\,inhomogeneous]
weighted Sobolev spaces on $\Omega$ (see Definition \ref{def-Ap}
for their definitions).
As an accessible example of Theorem \ref{thm-QX}, we
now state its special case where this conclusion is applied to
weighted Lebesgue spaces.

\begin{theorem}\label{thm-Q}
Let $k\in \mathbb{N}$, $p\in [1,\infty)$,
$\omega\in A_p$ with $p_{\omega}$ as in \eqref{eq-rw}, and
$q\in [1,\infty)$ satisfies $n(\frac{p_\omega}{p}-\frac{1}{q})<k$. Then
there exists a positive constant
$C$
such that, for any $s\in (0,1)$,
any cube $Q$ in $\mathbb{R}^n$, and any
$f\in \dot{W}^{k,p}_{\omega}(Q)$,
\begin{align}\label{eq-sss3}
&(1-s)^{\gamma_{p,q}}
\left\{
\int_{Q}
\left[\int_{Q(x,k)}
\frac{|\Delta^k_h f(x)|^q}{|h|^{n+skq}}\,dh
\right]^\frac{p}{q}\omega(x)\,dx\right\}^{\frac{1}{p}}\notag\\
&\quad \le
C[\ell (Q)]^{(1-s)k}
\left[\int_{Q}\left|\nabla^k f(x)\right|^p\omega(x)\,dx\right]^{\frac{1}{p}},
\end{align}
where $\ell (Q)$ denotes the edge length of
$Q$ and $\gamma_{p,q}$ is the same as in \eqref{Gqp}.
\end{theorem}

\begin{remark}
Although we do not obtain the
explicit dependence on weight
constants in \eqref{eq-sss3},
Theorem \ref{thm-Q} still
generalizes \cite[Corollary 2.2]{HMPV23}
in the sense that it covers more general
weight classes and higher-order derivatives and it includes
a Triebel--Lizorkin exponent $q$.
Indeed, Theorem \ref{thm-Q} with $\omega\in A_1$, $k=1$,
and $q=p$  reduces to
\cite[Corollary 2.2]{HMPV23} but without the explicit dependence on weight
constants as in \cite[Corollary 2.2]{HMPV23}.
Furthermore, in Theorem \ref{thm-Q},
the exponent $\gamma_{p,q}$ on $1-s$ is sharp
and the range of $q$ is nearly sharp;
see Remark \ref{rm312} and Proposition \ref{pro-q2}, respectively.
\end{remark}

\textbf{Application (III):
A Gagliardo--Nirenberg Interpolation Inequality.}
Gagliardo--Nirenberg interpolation inequalities
(see \cite[(1)]{BM18})
play an important role in the study of
Sobolev spaces and partial differential equations;
see, for example,
\cite{BM19,DD23,FFRS19,HSS20,KLV21,LRS25,MNSS12,MNSS13,SW13}.
In particular, the following
interpolation inequality
with optimal asymptotic factors holds:
For any given $p\in[1,\infty)$ and for any $s\in(0,1)$
and $f\in W^{1,p}$,
\begin{align}\label{GNe1}
s^{\frac1p}(1-s)^{\frac1p}
\|f\|_{\dot{W}^{s,p}}\lesssim
\|f\|_{L^p}^{1-s}\|\,|\nabla f|\,\|_{L^p}^s,
\end{align}
where the implicit positive constant
depends only on $n$ and $p$
(see, for instance, \cite[(17)]{YZ20}
or \cite{JM91,Mil05});
the optimality of $s^{\frac1p}$ is due to
the well-known Maz'ya--Shaposhnikova formula (see \cite{MS02})
and, similar to \eqref{eq-sss03},
the optimality of $(1-s)^{\frac1p}$ follows from
the BBM formula on $\mathbb{R}^n$
(see \cite{Bre02}).

As an application of Theorem \ref{thm-upi},
we obtain a nearly sharp
variant of \eqref{GNe1}
in the framework of ball Banach function spaces and higher-order
Sobolev spaces (see Theorem \ref{thm-in}).
As an accessible example, we
present its special case where this conclusion is applied to
weighted Lebesgue spaces.

\begin{theorem}\label{thm-ex}
Let $k\in \mathbb{N}$, $p\in [1,\infty)$,
$\omega\in A_p$ with $p_{\omega}$ as in \eqref{eq-rw}, and
$q\in [1,\infty)$ satisfy $n(\frac{p_\omega}{p}-\frac{1}{q})<k$.

\begin{enumerate}[{\rm(i)}]
\item If $q\le \frac{p}{p_\omega}$,
then there exists a positive constant
$C$ such that, for any $s\in(0,1)$ and $f\in W^{k,p}_\omega$,
\begin{align}\label{eq-exss}
s^{\frac1q}(1-s)^{\frac1q}\left\{
\int_{\mathbb{R}^n}
\left[\int_{\mathbb{R}^n}
\frac{|\Delta^k_h f(x)|^q}{|h|^{n+skq}}\,dh
\right]^\frac{p}{q}\omega(x)\,dx\right\}^{\frac{1}{p}}
\le C\left\|f\right\|^{1-s}_{L^p_{\omega}}
\left\|\,\left|\nabla^k f\right|\,\right\|^s_{L^p_{\omega}}.
\end{align}
\item If $q>\frac{p}{p_\omega}$ and
$\theta \in
(0, 1-\frac{n}{k}[\frac{p_\omega}{p}-\frac{1}{q}])$,
then there exists a positive constant
$C$ such that,
for any $s\in(\max\{1-\frac{\theta}{2},1-\frac{n}{kq}\},1)$
and $f\in W^{k,p}_\omega$.
\begin{align}\label{eq-exss00}
(1-s)^{\gamma_{p,q}}\left\{
\int_{\mathbb{R}^n}
\left[\int_{\mathbb{R}^n}
\frac{|\Delta^k_h f(x)|^q}{|h|^{n+skq}}\,dh
\right]^\frac{p}{q}\omega(x)\,dx\right\}^{\frac{1}{p}}
\le C\left\|f\right\|^{1-s}_{L^p_{\omega}}
\left\|\,\left|\nabla^k f\right|\,\right\|^s_{L^p_{\omega}},
\end{align}
where $\gamma_{p,q}$ is the same as in \eqref{Gqp}.
\end{enumerate}
\end{theorem}

\begin{remark}
Let all the symbols be as in Theorem \ref{thm-ex}.
Then the range $n(\frac{p_\omega}{p}-\frac{1}{q})<k$
of the Triebel--Lizorkin exponent $q$ in Theorem \ref{thm-ex}
is sharp;
see Proposition \ref{pro-sp3}. Moreover,
noting that, as mentioned earlier,
$s^{\frac{1}{q}}$ is due to the Maz'ya--Shaposhnikova formula,
since $q\le\frac{p}{p_\omega}$
is sharp in the weighted Maz'ya--Shaposhnikova formula
(see \cite[Remark 2.13(iii) and Theorem 4.1]{PYYZ24})
and Theorem \ref{thm-ex}(ii) requires $q>\frac{p}{p_\omega}$,
we cannot obtain the asymptotic factor $s^{\frac1q}$
in \eqref{eq-exss00} as $s\in(0,1)$ and $s\to 0$. Therefore,
the range $q>\frac{p}{p_\omega}$ in Theorem \ref{thm-ex}(ii)
is also sharp.
\end{remark}

Different from the localized inequality \eqref{eq-sss00},
the inner integration region of
the seminorm on the left-hand side of \eqref{eq-exss}
or \eqref{eq-exss00} is $\mathbb{R}^n$.
Thus, the proof of Theorem \ref{thm-ex} needs
careful handling of the non-local component of Gagliardo seminorms.
For the subcritical
case $q\le\frac{p}{p_\omega}$ in Theorem
\ref{thm-ex}, we bound the non-local component by
using Lebesgue norms
via employing the boundedness of ball average operators on weighted Lebesgue
spaces (see Lemma \ref{eq-s1}).
However, this approach is not applicable
to the supercritical case $q>\frac{p}{p_\omega}$,
where we exploit the higher-order Poincar\'e
inequality (see Lemma \ref{Poin}) and some
exquisite geometrical properties of shifted dyadic grids
(see Lemma \ref{lem-dy-cub})
to obtain a novel domination
of the corresponding
non-local component by interpolation
terms involving derivatives
(see Proposition \ref{lem-upw}).

\textbf{Application (IV): A Bourgain--Brezis--Mironescu Formula.}
The final application of Theorem \ref{thm-upi}
is to establish a sharp higher-order
BBM formula and a related characterization
of ball Banach Sobolev spaces on extension domains
(see Theorems \ref{thm-main} and \ref{thm-cwkx}).
These results improve the known corresponding ones
(see Remarks \ref{rem-main} and \ref{thm-cwkx} for the details).
As an accessible example, we
state the special case where these conclusions are applied to
weighted Lebesgue spaces. In what follows,
for any multi-index
$\alpha=(\alpha_1, \dots, \alpha_n)\in \mathbb{Z}_+^n$ and
any $x=(x_1,\dots,x_n)\in {\mathbb{R}^n}$,
we denote by $x^{\alpha}:={x_1}^{\alpha_1}\cdots{x_n}^{\alpha_n}$
the monomial of degree $|\alpha|:=\sum_{i=1}^{n}\alpha_i$.

\begin{theorem}\label{thm-wls}
Let $k\in {\mathbb{N}}$, $\omega\in
A_{p}$ with $p_{\omega}$
as in \eqref{eq-rw}, $q\in [1,\infty)$ satisfy that $q=p=1$ or
$n(\frac{p_\omega}{p}-\frac{1}{q})<k$ when $p\in (1,\infty)$,
and $\Omega$ be a $W^{k,p}_{\omega}$-extension domain (see Definition
\ref{df-EXD}). Then the following statements hold.
\begin{enumerate}[{\rm (i)}]
\item For any $f\in {W}^{k,p}_{\omega}(\Omega)$,
\begin{align*}
&\lim_{s\to 1^{-}}
(1-s)^{\frac{1}{q}}
\left\{
\int_{\Omega}
\left[\int_{\Omega(x,k)}
\frac{|\Delta^k_h f(x)|^q}{|h|^{n+skq}}\,dh
\right]^\frac{p}{q}\omega(x)\,dx\right\}^{\frac{1}{p}}\notag\\
&\quad=
\frac{1}{(kq)^\frac{1}{q}}
\left\{
\int_{\Omega}
\left[ \int_{\mathbb{S}^{n-1}}
\left|\sum_{\alpha\in\mathbb{Z}_+^n,\,|\alpha|=k}
\partial^{\alpha}f(x)\xi^{\alpha}\right|^q\,d\mathcal{H}^{n-1}(\xi)
\right]^{\frac{1}{q}}
\omega(x)\,dx\right\}^{\frac{1}{p}},
\end{align*}
here and thereafter,
$\mathcal{H}^{n-1}$ denotes the $(n-1)$-dimensional Hausdorff
measure.

\item If $p\in (1,\infty)$, then $f\in W^{k,p}_{\omega}(\Omega)$
if and only if
$f\in W^{k-1,p}_{\omega} (\Omega)$ [where $W^{0,p}_\omega
(\Omega):=L^p_\omega(\Omega)$] and
\begin{align*}
I(f):=\liminf_{s\to 1^{-}}
(1-s)^{\frac{1}{q}}
\left\{
\int_{\Omega}
\left[\int_{\Omega(x,k)}
\frac{|\Delta^k_h f(x)|^q}{|h|^{n+skq}}\,dh
\right]^\frac{p}{q}\omega(x)\,dx\right\}^{\frac{1}{p}}<\infty;
\end{align*}
moreover, for any $f\in W^{k,p}_{\omega}(\Omega)$,
$\|f\|_{W^{k,p}_{\omega}(\Omega)}
\sim \|f\|_{W^{k-1,p}_{\omega}(\Omega)}+I(f)$
with the positive equivalence constants independent of $f$.
\end{enumerate}
\end{theorem}

\begin{remark}
\begin{enumerate}[{\rm (i)}]
\item In the special case where $k=1$, $\omega\equiv 1$,
and $\Omega$ is a smooth bounded domain in $\mathbb{R}^n$,
Theorem \ref{thm-wls} with
$q=p$ coincides with the classical BBM formula in \cite{BBM01}.
In addition, when $k=1$, $\omega\equiv 1$,
and $\Omega=\mathbb{R}^n$,
Theorem \ref{thm-wls}(i)
is a special case of \cite[Theorem 1]{Moh24}.

\item The range of the Triebel--Lizorkin exponent $q$ in
Theorem \ref{thm-wls} is \emph{sharp};
see Proposition \ref{pro-sp3}.
\end{enumerate}
\end{remark}

The organization of the remainder of this article is as follows.

In Section \ref{sec-cap}, we show
Theorems \ref{thm-upi} and \ref{thm-ap}.
Section \ref{sec-bbf} is devoted to proving an inequality
related to Gagliardo and Sobolev seminorms
in the framework of  ball Banach function spaces.
The main target of Section \ref{sec-pf1} is
to obtain a fractional
Gagliardo--Nirenberg interpolation
inequality in the framework of ball Banach function spaces.
In Section \ref{sec-pf2}, we establish
the BBM formula on ball Banach function spaces and
the related characterization of ball Banach Sobolev spaces.
In Section \ref{sec-q}, we show the sharpness of the Triebel--Lizorkin exponent $q$
in the main theorems of this article.
Finally, in Section \ref{sec-app}, we apply these main results
(Theorems \ref{thm-QX},
\ref{thm-in}, \ref{thm-main},
and \ref{thm-cwkx})
to various specific examples of ball Banach function spaces.

We end this introduction by making some conventions on symbols.
We always let
$\mathbb{Z}_+ :=\mathbb{N}\cup\{0\}$.
If $E$ is a subset of ${\mathbb{R}^n}$, we denote by
${\bf 1}_E$ its \emph{characteristic function} and
by $E^{\complement}$ the set ${\mathbb{R}^n} \setminus E$.
For any sets $E,F\subset \mathbb{R}^n$, let
$E-F:=\{x-y:x\in E\ \text{ and}\ y\in F\}.$
Moreover, we use ${\bf 0}$ to denote the origin of ${\mathbb{R}^n}$
and $\mathbb{S}^{n-1}$ the unit sphere of ${\mathbb{R}^n}$.
For any $x\in {\mathbb{R}^n}$ and $r\in (0,\infty)$, let
$B(x,r):=\{y\in {\mathbb{R}^n}:|x-y|<r\}$.
For any $\lambda\in (0,\infty)$
and any cube $Q$ in $\mathbb{R}^n$, $\lambda Q$ means a cube
with the same center as $Q$ and $\lambda$ times the edge length of
$Q$.
Let $\Omega\subset \mathbb{R}^n$ be an open set.
For any measurable function $f$ on $\Omega$, its support
$\mathrm{supp\,} (f)$
is defined by setting
$\mathrm{supp\,} (f):=
\{x\in\Omega:f(x)\neq 0\}$.
We use $C_{\rm c}^{\infty}(\Omega)$ to denote
the space of all infinitely
differentiable functions on $\Omega$ with compact support.
For any $k\in\mathbb{N}$ and any Banach space $X$ of measurable functions
on $\mathbb{R}^n$, if $f\in L^1_{\rm loc}$ satisfies that
$\nabla^k f$ exists,
then define $\|\nabla^k f\|_X:=\|\,|\nabla^k f|\,\|_X$.
For any $f\in L^{1}_{\rm loc} $ and any bounded measurable $E\subset
{\mathbb{R}^n}$,
let
$
f(E):=\int_{E}f(x)\, dx
$
and $\fint_Ef(x)\,dx:=\frac{1}{|E|}\int_{E}f(x)\,dx$.
In addition,
we denote by $C$ a \emph{positive constant} which is independent of
the main parameters involved, but may vary from line to line and
use $C_{(\alpha,\dots)}$
to denote a positive constant depending on the indicated parameters
$\alpha,\dots$.
The symbol $f\lesssim g$ means $f\leq C g$ and, if
$f\lesssim g\lesssim f$, then we write $f\sim g$.
If $f\leq C g$ and $g=h$ or $g\leq h$, we
then write $f\lesssim g=h$ or $f\lesssim g \leq h$.
The symbol $s\to 0^{+}$ (resp. $s \to 1^-$)
means that there exists a constant $c_0\in (0,1)$
such that $s\in (0,c_0)$ and $s\to 0$
[resp. $s\in (1-c_0 ,1)$ and $s\to 1$].
Finally, in all proofs, we consistently retain the symbols
introduced in the original theorem (or related statement).

\begin{remark}\label{e}
We point out that, in Theorem 1.1(i) of
[Adv. Math. 481 (2025), Paper No. 110537], the exponent on $1-s$
was stated to be $\frac1q$ rather than $\gamma_{p,q}$,
due to an error in its proof.
In this revised version, we clarify that $\frac1q$
should be $\gamma_{p,q}$, which is indeed sharp;
see Theorem \ref{thm-upi} and Remark \ref{re12}(iii) of the present version.
As a result, we also present corrected versions
of Theorems 1.4, 1.6, and 1.8 of [Adv. Math. 481 (2025), Paper No. 110537], respectively, 
in Theorems \ref{thm-Q}, \ref{thm-ex}, and \ref{thm-wls}
of the present version.
\end{remark}

\section{Proofs of Theorems \ref{thm-upi} and \ref{thm-ap}}\label{sec-cap}
The target of this section is to prove
Theorems \ref{thm-upi} and \ref{thm-ap}.

First, we recall the concepts of
$A_p$-weights and
weighted function spaces as follows
(see, for instance, \cite[Section 7.1]{Gra14}).

\begin{definition}\label{def-Ap}
Let $k\in \mathbb{N}$, $p\in[1,\infty)$,
$\omega$ be a nonnegative locally integrable function on $\mathbb{R}^n$,
and $\Omega\subset\mathbb{R}^n$ an open set.
\begin{enumerate}[{\rm (i)}]
\item
The function $\omega$ is called an \emph{$A_p$-weight} if,
when $p\in(1,\infty)$,
\begin{align*}
[\omega]_{A_p}:=
\sup_{Q\subset {\mathbb{R}^n}}
\left[\frac{1}{|Q|}\int_{Q}\omega(x)\,dx\right]
\left\{\frac{1}{|Q|}\int_{Q}[\omega(x)]^{-\frac{1}{p-1}}\,dx\right\}^{p-1}
<
\infty
\end{align*}
and, when $p=1$,
\begin{align*}
[\omega]_{A_1}:=\sup_{Q\subset {\mathbb{R}^n}}
\left[\frac{1}{|Q|}
\int_{Q}\omega(x)\,dx\right]
\left\|\omega^{-1}\right\|_{L^\infty (Q)}
<
\infty,
\end{align*}
where
the suprema are
taken over all cubes $Q\subset {\mathbb{R}^n}$.

\item
The \emph{weighted Lebesgue space}
$L^{p}_{\omega}(\Omega)$
is defined to be the set of all
measurable functions $f$ on $\Omega$
such that
$\|f\|_{L^p_{\omega}(\Omega)}:=
[\int_{\Omega}|f(x)|^p \omega(x)\,
dx]^{\frac{1}{p}}
<
\infty.$

\item
The \emph{inhomogeneous weighted Sobolev space}
$W_{\omega}^{k,p}(\Omega)$ is defined to be the set
of all $f\in L^p_\omega (\Omega)$ such that,
for any multi-index
$\alpha\in \mathbb{Z}_{+}^n$ with
$|\alpha|\le k$,
the $\alpha$-th weak partial derivative
$\partial^{\alpha} f$ of $f$ exists and
$\partial^{\alpha} f \in L^p_\omega(\Omega) $.
Moreover, the \emph{norm} $\|\cdot\|_{{W}^{k,p}_{\omega}(\Omega)}$
of $W^{k,p}_\omega (\Omega)$
is defined by setting,
for any $f\in W^{k,p}_{\omega}(\Omega)$,
$\|f\|_{{W}^{k,p}_{\omega}(\Omega)}
:=
\sum_{\alpha\in \mathbb{Z}_{+}^n,\, |\alpha|\le k}
\left\|\partial^{\alpha} f\right\|_{L^p_\omega (\Omega)}.$

\item
The \emph{homogeneous weighted Sobolev space}
$\dot{W}_{\omega}^{k,p}(\Omega)$ is defined to be the set
of all $f\in L^1_{\rm loc}(\Omega) $ such that,
for any multi-index
$\alpha\in \mathbb{Z}_{+}^n$ with
$|\alpha|= k$,
the $\alpha$-th weak partial derivative
$\partial^{\alpha} f$ of $f$ exists and
$\partial^{\alpha} f \in L^p_\omega(\Omega) $.
Moreover, the \emph{seminorm} $\|\cdot\|_{\dot{W}^{k,p}_{\omega}(\Omega)}$
of $\dot{W}^{k,p}_\omega(\Omega) $
is defined by setting,
for any $f\in \dot{W}^{k,p}_{\omega}(\Omega)$,
$\|f\|_{\dot{W}^{k,p}_{\omega}(\Omega)}
:=
\sum_{\alpha\in \mathbb{Z}_{+}^n,\, |\alpha|= k}
\left\|\partial^{\alpha} f\right\|_{L^p_\omega (\Omega)}.$
\end{enumerate}
\end{definition}
In what follows, for any $f\in L^{1}_{\rm loc} $, the
\emph{Hardy--Littlewood maximal function}
$\mathcal{M}(f)$ of $f$ is defined by setting,
for any $x\in {\mathbb{R}^n}$,
$
\mathcal{M}(f)(x):=\sup_{B\ni x}\fint_{B}|f(y)|\,dy,
$
where the supremum is taken over all balls
$B\subset \mathbb{R}^n$ containing $x$. We denote by
$\|{\mathcal M}\|_{X \to Y}$
the operator norm of $\mathcal{M}$ from a
Banach space $X$ to another Banach space
$Y$.
For Muckenhoupt $A_p$-weights, we have the
following basic properties,
which are frequently used in this article; see, for instance,
\cite[(7.3)
and (7.5)]{Duo01},
\cite[Proposition 7.1.5 and Theorem 7.1.9]{Gra14},
\cite[Lemma 3.1]{Cla25},
and \cite[Theorem
2.7.4]{DHHR11}.

\begin{lemma}\label{lem-apwight}
If $p\in [1,\infty)$ and $\omega\in A_p$, then the
following
statements hold.
\begin{enumerate}[{\rm (i)}]
\item For any cubes $Q,S\subset {\mathbb{R}^n}$ with $Q\subset S$,
$
\omega(S)\leq [\omega]_{A_p}
({|S|}/{|Q|})^p
\omega(Q).
$

\item $\omega\in A_{q}$ for any $q\in [p,\infty)$
and, moreover,
$[\omega]_{A_{q}}\leq [\omega]_{
A_p}.$

\item If $p\in (1,\infty)$, then $\mathcal{M}$
is bounded on $L_{\omega}^p  $ and, moreover, there
exists a positive constant $C$,
depending only on $n$ and $p$, such that
\begin{align}\label{lem-apwight-e1}
\|{\mathcal M}\|_{L_{\omega}^p   \to
L_{\omega}^p  }
\leq C[\omega]^{\frac{1}{p-1}}_{A_p};
\end{align}
furthermore, if $\omega\in A_1$, then
$[\omega]_{A_p}^{\frac{1}{p-1}}$ in \eqref{lem-apwight-e1} can be replaced by
$[\omega]_{A_1}^{\frac1p}$.

\item
If $\upsilon$ is a nonnegative locally integrable
function on $\mathbb{R}^n$, then $\upsilon\in A_p$
if and only if
\begin{align*}
[\upsilon]^*_{A_p}:
=\sup_{Q\subset {\mathbb{R}^n} }
\sup_{\|f{\bf 1}_{Q}\|_{L^p_{\upsilon}}
\in (0,\infty)}
\frac{[\frac{1}{|Q|}\int_{Q}|f(x)|\,dx]^p}{\frac{1}{\upsilon(Q)}
\int_{Q}|f(x)|^p\upsilon(x)\,dx}<\infty,
\end{align*}
where the first supremum is taken over all cubes $Q\subset
{\mathbb{R}^n}$
and the second supremum is taken
over all $f\in L^{1}_{\rm loc}$ such that
$\|f{\bf 1}_{Q}\|_{L^p_{\upsilon}}
\in (0,\infty).$ Moreover, for any $\upsilon\in A_p$,
$[\upsilon]^*_{A_p}=[\upsilon]_{A_p}$.
\end{enumerate}
\end{lemma}

We first show the sharpness
of the exponent $\frac{1}{p}$ on the $A_1$-weight constant
in Theorem \ref{thm-upi}.

\begin{proposition}\label{pro-sharp}
Let $k\in \mathbb{N}$, $s\in (0,1)$, and $p,q\in [1,\infty)$.
If there exists a positive constant
$C$ such that, for any
$r\in (0,\infty)$, $ \omega \in A_1$,
and $f\in \dot{W}^{k,p}_{\omega}$,
the inequality
\begin{align}\label{eq-s0}
\left\{
\int_{\mathbb{R}^n}
\left[\int_{B({\bf 0},r)}
\frac{|\Delta^k_h f(x)|^q}{|h|^{n+skq}}\,dh
\right]^\frac{p}{q}\omega(x)\,dx\right\}^{\frac{1}{p}}
\le
C[\omega]_{A_{1}}
^{\beta}
r^{(1-s)k}
\left\|\nabla^k f\right\|_{L^p_{\omega}}
\end{align}
holds with some $\beta\in \mathbb{R}$,
then $\beta\ge \frac{1}{p}$.
\end{proposition}

\begin{proof}
For any given $\delta\in(0,n]$,
consider the weight $\omega(x):=|x|^{\delta-n}$ for any $x\in\mathbb{R}^n$.
Then, by \cite[Example 7.1.7]{Gra14}, we find that
$\omega\in A_1$ and that $[\omega]_{A_1}\sim\delta^{-1} $ and $\omega(B({\bf 0},1))\sim \delta^{-1}$
with the positive equivalence constants independent of $\delta$
(see also \cite[p.\,6118]{pr19}).
Choose a function $f\in C_{\rm c}^\infty$ satisfying
${\bf 1}_{B({\bf 0},4k+4)\setminus B({\bf 0},4k)}\le f\le {\bf 1}_{B({\bf 0},4k-1)^{\complement}}$
and a ball $B_0 \subset B({\bf 0},4k+3)\setminus B({\bf 0},4k)$ with radius $1$.
Notice that, for any $i\in \{0,\ldots,k-1\}$, $x\in B({\bf 0},1)$,
and $y\in B_0$,
$\frac{|(k-i)x+iy|}{k}\le \frac{|x|+(k-1)|y|}{k}<4k-1$, which further implies that
\begin{align}\label{eq-s2}
\Delta^k_{\frac{y-x}{k}}f(x)
=\sum_{i=0}^{k}(-1)^{k-i}\binom{k}{i}f\left(\frac{(k-i)x+iy}{k}\right)
=f(y)=1.
\end{align}
From this, the observation that,
for any $x\in B({\bf 0},1)$ and $y\in B_0$, $|x-y|<4k+4$,
and a change of variables, it follows that
\begin{align*}
&\left\{
\int_{\mathbb{R}^n}
\left[\int_{B({\bf 0},4k+4)}
\frac{|\Delta^k_h f(x)|^q}{|h|^{n+skq}}\,dh
\right]^\frac{p}{q}\omega(x)\,dx\right\}^{\frac{1}{p}}\\
&\quad \gtrsim
\left\{
\int_{B({\bf 0},1)}
\left[\int_{B_0}
\frac{|\Delta^k_{\frac{y-x}{k}} f(x)|^q}{|y-x|^{n+skq}}\,dy
\right]^\frac{p}{q}\omega(x)\,dx\right\}^{\frac{1}{p}}\gtrsim
\omega(B({\bf 0},1))^{\frac{1}{p}}\sim \delta^{-\frac{1}{p}}.
\end{align*}
By this, \eqref{eq-s0} with $r:=4k+4$,
and $[\omega]_{A_1}\sim\delta^{-1}$,
we obtain
$\delta^{-\frac{1}{p}}\lesssim \delta^{-\beta}$ with the implicit
positive constant independent of $\delta$.
This, together with the arbitrariness of $\delta\in(0,n]$, further implies that
$\beta \ge \frac{1}{p}$, which completes the proof of Proposition
\ref{pro-sharp}.
\end{proof}

To prove Theorem \ref{thm-upi},
we also need
the following well-known $(q, p)$-Poincar\'e inequality
(see, for instance,
\cite[Corollary 9.14]{Bre11} or \cite[Lemma 4]{Moh24}).

\begin{lemma}\label{lem-qpPoin}
If $k\in \mathbb{N}$ and $p,q\in [1,\infty)$
satisfies $n(\frac{1}{p}-\frac{1}{q})< k$,
then there exists a positive constant $C_{(n,k,p,q)}$, depending
only on $n,k,p$, and $q$,
such that, for any ball $B\subset \mathbb{R}^n$
with radius $r\in (0,\infty)$
and for any $f\in C^{\infty}$,
\begin{align*}
\left[\fint_{B}|f(y)|^q \,dy\right]^\frac{1}{q}
\le
C_{(n,k,p,q)}
\left\{r^k \left[\fint_{B}
\left|\nabla^k f(y)\right|^p \,dy\right]^\frac{1}{p}
+\left[\fint_{B}|f(y)|^p \,dy\right]^\frac{1}{p}\right\}.
\end{align*}
\end{lemma}

Using this inequality, we first establish the following
estimate.

\begin{lemma}\label{Lem2.5}
If $k\in \mathbb{N}$ and $p,q\in [1,\infty)$
satisfy $n(\frac{1}{p}-\frac{1}{q})< k$,
then there exists a positive constant $C_{(n,k,p,q)}$, depending
only on $n,k,p$, and $q$, such that,
for any $f\in C^{\infty}$, $r\in (0,\infty)$, and $x\in \mathbb{R}^n$,
\begin{align}\label{eq-pp}
\left[\fint_{ B({\bf 0},r)}
\left|\Delta^k_h f(x)\right|^q\,dh\right]^{\frac{1}{q}}
\le C_{(n,k,p,q)}  r^k \left[\int_{0}^{k+1}
\fint_{B(x,tr)}
\left|\nabla^k f(y)\right|^p\,dy\,dt\right]^{\frac{1}{p}}.
\end{align}
\end{lemma}

\begin{proof}
Let $f\in C^\infty$, $r\in (0,\infty)$, and $x\in \mathbb{R}^n$. For any
$h\in \mathbb{R}^n$, define $F_x(h):=\Delta_h^k f(x)$.
By \eqref{eq-sho-hd}, we find that,
for any $h\in \mathbb{R}^n$,
\begin{align*}
\left|\nabla^k F_x(h)\right|
&=\left[\sum_{\alpha\in \mathbb{Z}_{+}^n,\,|\alpha|=k}
\left|\sum_{j=1}^{k}(-1)^{k-j}
j^{k}\binom{k}{j}\partial^{\alpha}
f(x+jh)\right|^2\right]^\frac{1}{2}\lesssim
\sum_{j=1}^{k}
\left|\nabla^k
f(x+jh)\right|.
\end{align*}
From this, the assumption $n(\frac{1}{p}-\frac{1}{q})< k$,
and Lemma \ref{lem-qpPoin} with $f:=F_x$, it follows that
\begin{align}\label{eqPkey2}
\left[\fint_{ B({\bf 0},r)}
\left|\Delta^k_h f(x)\right|^q\,dh\right]^{\frac{1}{q}}
&\lesssim
r^{k}
\left[\fint_{ B({\bf 0},r)}
\left|\nabla^k F_x(h)\right|^{p}\,dh\right]^{\frac{1}{p}}
+
\left[\fint_{ B({\bf 0},r)}
\left|\Delta^k_h f(x)\right|^{p}\,dh\right]^{\frac{1}{p}}\notag\\
&\lesssim r^k
\sum_{j=1}^{k}
\left[
\fint_{ B({\bf 0},r)}
\left|\nabla^k
f(x+jh)\right|^{p}\,dh\right]^{\frac{1}{p}}
+
\left[\fint_{ B({\bf 0},r)}
\left|\Delta^k_h f(x)\right|^{p}\,dh\right]^{\frac{1}{p}}.
\end{align}
Using a change of variables, we obtain
\begin{align}\label{1kp}
\sum_{j=1}^{k}
\left[
\fint_{ B({\bf 0},r)}
\left|\nabla^k
f(x+jh)\right|^{p}\,dh\right]^{\frac{1}{p}}
&=
\sum_{j=1}^{k}
\left[
\fint_{ B(x,jr)}
\left|\nabla^k
f(y)\right|^{p}\,dy\right]^{\frac{1}{p}}
\lesssim
\left[
\fint_{ B(x,kr)}
\left|\nabla^k
f(y)\right|^{p}\,dy\right]^{\frac{1}{p}}\notag\\
&\lesssim \left[\int_{0}^{k+1}
\fint_{B(x,tr)}
\left|\nabla^k f(y)\right|^p\,dy\,dt\right]^{\frac{1}{p}}.
\end{align}
Furthermore,
applying the assumption
$f\in C^{\infty}$ and \cite[Proposition 1.4.5]{Gra14},
we conclude that,
for any $h\in \mathbb{R}^n$,
\begin{align*}
\Delta^{k}_{h}f(x)
=
\int_{[0,1]^k}
\sum_{\alpha\in\mathbb{Z}_{+}^n,\,|\alpha|=k}
\partial^{\alpha}f(x+[s_1+\cdots+s_k]h)h^{\alpha}
\,ds_1\cdots \, ds_k,
\end{align*}
which, together with H\"{o}lder's inequality, Tonelli's theorem, and a change of variables, further
implies that
\begin{align*}
\left[\fint_{ B({\bf 0},r)}
\left|\Delta^k_h f(x)\right|^{p}\,dh
\right]^{\frac{1}{p}}
&\lesssim
r^k
\left[\int_{[0,1]^k}\fint_{ B({\bf 0},r)}
\left|\nabla^k f(x+[s_1+\cdots+s_k]h)
\right|^{p}
\,dh\,ds_1\cdots \, ds_k\right]^{\frac{1}{p}}\\
&\le
r^k
\left[\int_{[0,1]^{k-1}}\int_{0}^{k}\fint_{ B({\bf 0},r)}
\left|\nabla^k f(x+th)
\right|^{p}
\,dh\, dt \,ds_1\cdots \, ds_{k-1}\right]^{\frac{1}{p}}\\
&=
r^k
\left[\int_{[0,1]^{k-1}}\int_{0}^{k}\fint_{ B(x,tr)}
\left|\nabla^k f(y)
\right|^{p}
\,dy\, dt \,ds_1\cdots \, ds_{k-1}\right]^{\frac{1}{p}}\\
&\le
r^k \left[\int_{0}^{k+1}
\fint_{B(x,tr)}
\left|\nabla^k f(y)\right|^p\,dy\,dt\right]^{\frac{1}{p}}.
\end{align*}
Combining this, \eqref{eqPkey2}, and \eqref{1kp}, we obtain
\eqref{eq-pp} and hence complete the proof of Lemma \ref{Lem2.5}.
\end{proof}

We now turn to show Theorem \ref{thm-upi}.

\begin{proof}[Proof of Theorem \ref{thm-upi}]
We first prove (i). Let $\omega\in A_1$, $s\in(0,1)$,
and $r\in(0,\infty)$. We now show that \eqref{eq-sss00}
holds for any $f\in C^{\infty}\cap \dot{W}_\omega^{k,p}$.
For this purpose, fix $f\in C^{\infty}\cap \dot{W}_\omega^{k,p}$.
Notice that, for any $h\in \mathbb{R}^n$
with $0<|h|<r$,
\begin{align*}
\int_{|h|}^{r}\frac{dt}{t^{n+skq+1}}
=\frac{1}{n+skq}
\left(\frac{1}{|h|^{n+skq}}-\frac{1}{r^{n+skq}}\right),
\end{align*}
which further implies that
\begin{align*}
\frac{1}{|h|^{n+skq}}
=(n+skq) \int_{|h|}^{1}\frac{dt}{t^{n+skq+1}}
+\frac{1}{r^{n+skq}}.
\end{align*}
From this, the assumption $s\in (0,1)$, and
Tonelli's theorem,
 we infer that, for any $x\in \mathbb{R}^n$,
\begin{align}\label{eq-div12-1}
&\left[\int_{B({\bf 0},r)}
\frac{|\Delta^k_h f(x)|^q}{|h|^{n+skq}}\,dh
\right]^\frac{1}{q}\notag\\
&\quad\lesssim
\left[\int_{B({\bf 0},r)}
\int_{|h|}^{r}
\frac{|\Delta^k_h f(x)|^q}{t^{n+skq+1}}\,dt\,dh
\right]^\frac{1}{q}
+r^{-sk}
\left[\fint_{B({\bf 0},r)}
{\left|\Delta^k_h f(x)\right|^q}\,dh
\right]^\frac{1}{q}\notag\\
&\quad=\left[\int_{0}^{r}
t^{-skq-1}
\fint_{B({\bf 0},t)}
\left|\Delta^k_h f(x)\right|^q\,dh\,dt
\right]^\frac{1}{q}+
r^{-sk}\left[\fint_{B({\bf 0},r)}
{\left|\Delta^k_h f(x)\right|^q}\,dh
\right]^\frac{1}{q};
\end{align}
notice that
\begin{align*}
\left[\int_{0}^{r}
t^{-skq-1}
\fint_{B({\bf 0},t)}
\left|\Delta^k_h f(x)\right|^q\,dh\,dt
\right]^\frac{1}{q}
\lesssim r^{-sk}
\left[\sum_{\nu =0}^{\infty}
2^{\nu skq}
\fint_{B({\bf 0}, 2^{-\nu }r)}
\left|\Delta^k_h f(x)\right|^q\,dh
\right]^\frac{1}{q}
\end{align*}
Hence, using \eqref{eq-div12-1},
we obtain, for any $x\in\mathbb{R}^n$,
\begin{align}\label{eq-div12}
\left[\int_{B({\bf 0},r)}
\frac{|\Delta^k_h f(x)|^q}{|h|^{n+skq}}\,dh
\right]^\frac{1}{q}
\lesssim r^{-sk}
\left[\sum_{\nu =0}^{\infty}
2^{\nu skq}
\fint_{B({\bf 0}, 2^{-\nu }r)}
\left|\Delta^k_h f(x)\right|^q\,dh
\right]^\frac{1}{q}.
\end{align}
We next consider the following two cases for $p$.

\emph{Case 1: $p=1$}. In this case, $\gamma_{p,q}=1$.
By \eqref{eq-div12},
the well-known inequality that,
for any sequence
$\{a_{\ell}\}_{\ell\in \mathbb{N}}\subset \mathbb{C}$
and any $\mu\in (0,1]$,
\begin{align}\label{eq-rrr}
	\left(\sum_{\ell\in \mathbb{N}}|a_{\ell}|\right)^\mu
	\le
	\sum_{\ell\in \mathbb{N}}|a_{\ell}|^\mu,
\end{align}
Lemma \ref{Lem2.5} with $p:=1$, Tonelli's theorem, and Definition
\ref{def-Ap}(i),
we find that
\begin{align*}
	&\int_{\mathbb{R}^n}\left[\int_{B({\bf 0},r)}
	\frac{|\Delta^k_h f(x)|^q}{|h|^{n+skq}}\,dh
	\right]^\frac{1}{q}\omega(x)\,dx\\
	&\quad\lesssim
	r^{-sk}
	\int_{\mathbb{R}^n}
\sum_{\nu =0}^{\infty}
	2^{\nu sk}
		\left[\fint_{B({\bf 0}, 2^{-\nu }r)}
	\left|\Delta^k_h f(x)\right|^q\,dh
	\right]^\frac{1}{q}\omega(x)\,dx\\
		&\quad\lesssim
		r^{(1-s)k}
	\int_{\mathbb{R}^n}
		\sum_{\nu =0}^{\infty}
		2^{\nu (s-1)k}\int_{0}^{k+1}\fint_{B(x,2^{-\nu}tr)}
		\left|\nabla^k f(y)\right|\,dy\,dt
	\omega(x)\,dx\\
	&\quad=
	r^{(1-s)k}
	\sum_{\nu =0}^{\infty}
	2^{\nu (s-1)k}\int_{0}^{k+1}	\int_{\mathbb{R}^n}\fint_{B(y,2^{-\nu}tr)}\omega(x)\,dx
	\left|\nabla^k f(y)\right|\,dy	\,dt\\
	&\quad\le
	r^{(1-s)k}[\omega]_{A_1}\left\|\nabla^k f\right\|_{L^1_\omega}
\frac{k+1}{1-2^{(s-1)k}}.
\end{align*}
This, combined with the fact that $1-2^{(s-1)k}\sim 1-s$
for any $s\in (0,1)$, implies that
\eqref{eq-sss00}
holds in this case.

\emph{Case 2: $p\in(1,\infty)$.} In this case, $\gamma_{p,q}=\frac{1}{q}$.
Choose $\beta  \in (1,p)$
satisfying $n(\frac{\beta }{p}-\frac{1}{q})<k$.
From \eqref{eq-div12}, Lemma \ref{Lem2.5} with $p$ replaced by $\frac{p}{\beta}$,
the definition of the Hardy--Littlewood
maximal function, and $1-2^{(s-1)k}\sim 1-s$
for any $s\in (0,1)$ again,
we infer that, for any $x\in\mathbb{R}^n$,
\begin{align*}
\left[\int_{B({\bf 0},r)}
\frac{|\Delta^k_h f(x)|^q}{|h|^{n+skq}}\,dh
\right]^\frac{1}{q}
&\lesssim r^{(1-s)k}
\left\{\sum_{\nu=0}^{\infty}2^{\nu skq}
\left[\int_{0}^{k+1}
\fint_{B(x,tr)}\left|\nabla^k f(y)\right|^{\frac{p}{\beta}}
\,dy\,dt\right]^{\frac{\beta q}{p}}\right\}^{\frac1q}\\
&\le r^{(1-s)k}(k+1)^{\frac\beta p}
\left[\mathcal{M}\left(\left|\nabla^k f\right|^{\frac p\beta}\right)
(x)\right]^{\frac{\beta}{p}}
\left[\sum_{\nu=0}^{\infty}2^{\nu(s-1)kq}\right]^{\frac1q}\\
&\sim\frac{r^{(1-s)k}}{(1-s)^{\frac1q}}
\left[\mathcal{M}\left(\left|\nabla^k f\right|^{\frac p\beta}\right)
(x)\right]^{\frac{\beta}{p}}.
\end{align*}
Taking the $L^p_\omega$
norm on the leftmost and rightmost sides
of the above inequality and applying Lemma \ref{lem-apwight}(iii)
and $\beta\in(1,p)$, we further obtain \eqref{eq-sss00}.
This then finishes the proof that
\eqref{eq-sss00} holds for any $f\in C^\infty\cap \dot{W}^{k,p}_{\omega}$.

Next, we show that
\eqref{eq-sss00} also holds for any
$f\in \dot{W}^{k,p}_{\omega}$. To
do this, let $f\in \dot{W}^{k,p}_{\omega}$ and
choose
$\eta\in C_{\rm c}^{\infty}$
satisfying that ${\rm supp}\,(\eta)\subset B({\bf 0},1)$
and $\int_{\mathbb{R}^n}\eta(x)\,dx =1$.
For any $j\in \mathbb{N}$, let
$\eta_j(\cdot):= j^n\eta(j\cdot)$
and $f_j:=  \eta_j \ast f$.
Then, from \cite[Theorem 2.1.4]{Tur00} and
\cite[Theorem 4.1(iv)]{EG15}, we deduce
that
$
\lim_{j\to \infty}\|f_j -f\|_{\dot{W}_{\omega}^{k,p}}=0
$
and, for almost every $x\in \mathbb{R}^n$,
$\lim_{j\to \infty}	f_j (x)= f(x)$.
Applying these, Fatou's lemma,
and \eqref{eq-sss00} for
$\{f_j\}_{j\in \mathbb{N}}$,
we conclude that, for any $s\in (0,1)$ and $r\in (0,\infty)$,
\begin{align*}
	&	(1-s)^{\gamma_{p,q}}
	\left\{
	\int_{\mathbb{R}^n}
	\left[\int_{B({\bf 0},r)}
	\frac{|\Delta^k_h f(x)|^q}{|h|^{n+skq}}\,dh
	\right]^\frac{p}{q}\omega(x)\,dx\right\}^{\frac{1}{p}}\notag\\
	&\quad	\le
	\liminf_{j\to \infty}(1-s)^{\gamma_{p,q}}
	\left\{\int_{\mathbb{R}^n}
	\left[\int_{B({\bf 0},r)}
	\frac{|\Delta^k_h f_j(x)|^q}{|h|^{n+skq}}\,dh
	\right]^\frac{p}{q}\omega(x)\,dx\right\}^{\frac{1}{p}}\notag\\
	&\quad \lesssim
	\liminf_{j\to \infty} 	[\omega]_{A_{1}}^{\frac{1}{p}}r^{(1-s)k}
	\left\|\nabla^k f_j\right\|_{L^p_{\omega}}
	=[\omega]_{A_{1}}^{\frac{1}{p}}r^{(1-s)k}
	\left\|\nabla^k f\right\|_{L^p_{\omega}},
\end{align*}
which implies that \eqref{eq-sss00} also holds for any
$f\in \dot{W}^{k,p}_{\omega}$.
Thus, we obtain \eqref{eq-sss00} and hence complete
the proof (i).

Finally, let $p\in(1,\infty)$,
$\omega\in A_p $, and $\beta  \in (p_\omega,p)$
satisfy $n(\frac{\beta }{p}-\frac{1}{q})<k$.
Then $\omega\in A_\beta$.
Therefore, by the proof of Case 2 and Lemma \ref{lem-apwight}(iii),
we conclude \eqref{eq-sss}.
This then finishes the proof of (ii)
and hence Theorem \ref{thm-upi}.
\end{proof}

Moreover, the following proposition implies that, in Theorem
\ref{thm-upi}, the factors on
$(1-s)$ are sharp.

\begin{proposition}\label{gammaqp}
Let $p,q \in [1, \infty)$ satisfy $n(\frac{1}{p}-\frac{1}{q})<1$ and $\gamma\in (0,\infty)$.
If there exists a positive constant $C$ such that,
for any $s \in (0,1)$ and $f \in \dot{W}^{1,p}$,
\begin{align}\label{1sGa2}
	(1-s)^\gamma\left\{
	\int_{\mathbb{R}^n}\left[\int_{|h|\le 1}
\frac{|f(x+h)-f(x)|^q}{|h|^{n+sq}}\,dh\right]^{\frac{p}{q}}\,dx
\right\}^{\frac{1}{p}} \le C {\|\nabla f\|_{L^p}},
\end{align}
then $\gamma \ge \gamma_{p,q}$.
\end{proposition}

\begin{proof}
	We consider the following three cases for $p$ and $n$.
	
	\emph{Case 1: $p=n=1$}. In this case,
let $m \in \mathbb{N}$ be a large integer to be specified later.
For any $t\in \mathbb{R}$, define
\begin{align*}
\eta_m(t) := \begin{cases}
0 & \text{if}\ t \le 0, \\
t & \text{if}\ 0 < t < 2^{-m}, \\
2^{-m} & \text{if}\ 2^{-m} \le t \le 2- 2^{-m}, \\
-t+2 & \text{if}\ 2- 2^{-m} < t < 2, \\
0 & \text{if}\ t \ge 2.
\end{cases}
\end{align*}
Then $\|\eta_m'\|_{L^1(\mathbb{R})}\sim 2^{-m}$
and $\|\eta_m\|_{L^1(\mathbb{R})}\lesssim 2^{-m}$.
 Let $f_m:=\eta_m$.
Notice that, for any fixed $j\in \{0,\ldots,m\}$ and
for any $x\in (2^{-j},2^{-j+1})$ and $h\in (2^{-j+2},2^{-j+3})$, $x-h<0$ and $2^{-m}<x\le 1$,
which further imply that $f_m(x-h)=0$ and $f_m (x)=2^{-m}$. From these,
and a change of variables, we deduce that,
for any $s\in (0,1)$,
\begin{align}\label{1s1-n1}
&\int_{\mathbb{R}}\left[\int_{-1}^1
\frac{|f_m(x+h)-f_m(x)|^q}{|h|^{1+sq}}\,dh\right]^{\frac{1}{q}}\,dx\notag\\
&\quad\ge
\sum_{j=3}^{m}
\int_{2^{-j}}^{2^{-j+1}}
\left[\int_{2^{-j+2}}^{2^{-j+3}}
\frac{|f_m(x)-f_m (x-h)|^q}{h^{1+sq}}\,dh\right]^{\frac{1}{q}}\,dx\notag\\
&\quad = 2^{-m}\sum_{j=3}^{m}
\int_{2^{-j}}^{2^{-j+1}}
\left(\int_{2^{-j+2}}^{2^{-j+3}}
\frac{1}{h^{1+sq}}\,dh\right)^{\frac{1}{q}}\,dx\notag\\
&\quad \sim 2^{-m}\sum_{j=3}^{m}2^{j(s-1)} = 2^{-m}
\frac{2^{3(s-1)}[1-2^{-(m-2)(1-s)}]}{1-2^{s-1}}.
\end{align}
For any $s\in (0,1)$, choose $m$ satisfying $(m-2)(1-s)\ge 1$. Then we have
\begin{align*}
\frac{2^{3(s-1)}[1-2^{-(m-2)(1-s)}]}{1-2^{s-1}}
\gtrsim \frac{1}{1-s}.
\end{align*}
Combining this, \eqref{1s1-n1}, \eqref{1sGa2}, and $\|f'_m\|_{L^1}\sim 2^{-m}$,
we find that, for any $s\in (0,1)$, $(1-s)^{\gamma-1}\lesssim 1$,
which further implies that $\gamma \ge 1=\gamma_{p,q}$.

\emph{Case 2: $p=1$ and $n\ge 2$}. In this case,
let $\phi \in C_c^\infty(\mathbb{R}^{n-1})$ be a nonnegative smooth function such that
$$
\text{supp}(\phi) \subset \{y \in \mathbb{R}^{n-1} : |y| \le 1\} \quad
\text{and}\quad
\inf_{y\in\mathbb{R}^{n-1}:|y|\le\frac12}\phi(y)>0.
$$
For any given $s\in(0,1)$, we choose the same $m\in\mathbb{N}$ as in Case 1,
that is, $(m-2)(1-s)\ge 1$.
For any $x := (x_1, x') \in \mathbb{R} \times \mathbb{R}^{n-1}$, define
$
f_m(x) := \eta_m(x_1)\phi(x').
$
Thus,
\begin{align}\label{eq-Nfm}
\|\nabla f_m\|_{L^1(\mathbb{R}^n)}
&\le \|\partial_{x_1} f_m\|_{L^1(\mathbb{R}^n)} + \|\nabla_{x'} f_m\|_{L^1(\mathbb{R}^n)} \notag\\
&= \|\eta_m'\|_{L^1(\mathbb{R})} \|\phi\|_{L^1(\mathbb{R}^{n-1})} + \|\eta_m\|_{L^1(\mathbb{R})} \|\nabla_{x'} \phi\|_{L^1(\mathbb{R}^{n-1})}\lesssim 2^{-m}.
\end{align}
Similar to Case 1, for any fixed $j \in \{3, \ldots, m\}$
and for any $x\in E_j:=\{ x=(x_1,x') \in \mathbb{R}\times\mathbb{R}^{n-1} : x_1 \in (2^{-j}, 2^{-j+1}), \ |x'|<\frac{1}{2} \}$ and $h:=(h_1,h')
\in \mathbb{R} \times \mathbb{R}^{n-1}$
with $h_1 \in (2^{-j+2},2^{-j+3})$ and $|h'|\le h_1$,
it holds that $\eta_m(x_1)=2^{-m}$ and $\eta_m(x_1-h_1)=0$, which imply
$$
|f(x)-f(x-h)| = |\eta_m(x_1)\phi(x')-\eta_m(x_1-h_1)\phi(x'-h')|=
\eta_m(x_1)\phi(x')\gtrsim 2^{-m}.
$$
By this and the choice of $m$, we further obtain
\begin{align*}
&\int_{\mathbb{R}^n}\left[\int_{|h|\le 1} \frac{|f(x)-f(x-h)|^q}{|h|^{n+sq}}\,dh\right]^{\frac{1}{q}}\,dx\\
&\quad \ge \sum_{j=3}^m \int_{E_j} \left[ \int_{2^{-j+2}}^{2^{-j+3}} \int_{|h'|\le h_1} \frac{|f(x)-f(x-h)|^q}{|h|^{n+sq}} \, dh' \, dh_1 \right]^{\frac{1}{q}} dx\\
&\quad \gtrsim 2^{-m}\sum_{j=3}^m \int_{E_j} \left( \int_{2^{-j+2}}^{2^{-j+3}} \int_{|h'|\le h_1} \frac{1}{|h|^{n+sq}} \, dh' \, dh_1 \right)^{\frac{1}{q}} dx
\sim2^{-m} \sum_{j=3}^m \int_{2^{-j}}^{2^{-j+1}}
\left( \int_{2^{-j+2}}^{2^{-j+3}} \frac{1}{h_1^{1+sq}} dh_1 \right)^{\frac{1}{q}} dx_1\\
&\quad \sim2^{-m}\frac{2^{3(s-1)}[1-2^{-(m-2)(1-s)}]}{1-2^{s-1}}
\gtrsim \frac{2^{-m}}{1-s}.
\end{align*}
Hence, from this, \eqref{eq-Nfm}, and \eqref{1sGa2},
we deduce that, for any $s\in (0,1)$, $(1-s)^{\gamma-1}\lesssim 1$,
which further implies that $\gamma \ge 1=\gamma_{p,q}$.

\emph{Case 3: $p\in (1,\infty)$}. In this case, $\gamma_{p,q}=\frac{1}{q}$ and
$C_{\rm c}^{\infty}\subset\dot{W}^{1,p}$ (see, for instance, \cite[Theorem 4]{HK95}).
Choose $f\in C_{\rm c}^\infty$ such that $\nabla f\ne 0$.
By \cite[Theorem 14]{Moh24}, we find that
there exists a positive constant $C$ such that
\begin{align*}
	\lim_{s\to 1^-}	{(1-s)^\frac{1}{q}}\left[
	\int_{\mathbb{R}^n}\left(\int_{|h|\le 1} \frac{|f(x+h)-f(x)|^q}{|h|^{n+sq}}\,dh\right)^{\frac{p}{q}}\,dx\right]^{\frac{1}{p}}= C \|\nabla f\|_{L^p}.
\end{align*}
Thus, there exists $s_f\in(0,1)$ such that,
for any $s\in(s_f,1)$,
\begin{align*}
	{(1-s)^\frac{1}{q}}\left[
	\int_{\mathbb{R}^n}\left(\int_{|h|\le 1} \frac{|f(x+h)-f(x)|^q}{|h|^{n+sq}}\,dh\right)^{\frac{p}{q}}\,dx\right]^{\frac{1}{p}}\ge   \frac{C}{2} \|\nabla f\|_{L^p}.
\end{align*}
Combining this and \eqref{1sGa2},
we conclude that, for any $s\in(s_f,1)$,
$(1-s)^{\gamma-\frac{1}{q}}\lesssim 1$, which further implies
$\gamma\ge \frac{1}{q}=\gamma_{p,q}$.
This then finishes the proof of
Proposition \ref{gammaqp}.
\end{proof}

Now, we are devoted to
showing Theorem \ref{thm-ap}.
Let us first recall some related concepts.
Let $\mathcal{S}$ denote the set of
all Schwartz functions on $\mathbb{R}^n$ and
$\mathcal{S}'
/\mathcal{P}$ denote the set of
all tempered distributions modulo polynomials.
Recall that, for any $f\in\mathcal{S}$,
its \emph{Fourier transform} $\mathcal{F} f$ is defined by
setting, for any $x\in\mathbb{R}^n$,
$
\mathcal{F}f(x):=\int_{\mathbb{R}^n}
f(\xi)e^{-2\pi ix\cdot \xi}\,d\xi
$
and its \emph{inverse Fourier transform} $\mathcal{F}^{-1}f$
is defined by setting
$\mathcal{F}^{-1}f(x):=\mathcal{F}f(-x)$ for any
$x\in\mathbb{R}^n$.
These definitions are naturally extended to
$\mathcal{S}'/\mathcal{P}$.
Then, for any $\beta\in\mathbb{R}$ and
$f\in\mathcal{S}'/\mathcal{P}$,
the \emph{Riesz potential} $I_\beta f$ of $f$
is defined by setting
\begin{align}\label{df-ibeta}
I_{\beta} f:=\mathcal{F}^{-1}
\left[(2\pi|\cdot|)^{-\beta}\mathcal{F}f\right];
\end{align}
see
\cite[Chapter 1]{Gra14m} for more details.
Let $p\in [1,\infty)$ and $\omega$ be a
nonnegative locally integrable function on $\mathbb{R}^n$.
The \emph{homogeneous weighted
Sobolev space} $\dot{H}^{1,p}_\omega$
is defined to be the set of all $f\in\mathcal{S}'
/\mathcal{P}$ such that
$\|f\|_{\dot{H}^{1,p}_\omega}
:=\left\|I_{-1}f\right\|_{L^p_\omega}$ is finite.
Moreover, let
\begin{align}\label{eq-Yd}
\dot{Y}^{1,p}_{\omega} :=
\begin{cases} \dot{W}^{1,p}_\omega
& \text{if }n=1\text{ or }p=1,\\
\dot{H}^{1,p}_\omega
& \text{if } n\in \mathbb{N} \cap [2,\infty)
\text{ and }p\in (1,\infty).
\end{cases}
\end{align}

To prove Theorem \ref{thm-ap},
we also need the following conclusion, which
is a direct consequence of \cite[Theorem 2.8 and Remark 2.9]{b82};
we omit the details.

\begin{lemma}\label{lem-hp}
If $p\in (1,\infty)$ and
$\omega\in A_p$, then
$\dot{W}^{1,p}_{\omega} =\dot{H}^{1,p}_{\omega} $ with equivalent seminorms.
\end{lemma}

\begin{proof}[Proof of Theorem \ref{thm-ap}]
From Theorem \ref{thm-upi} with $k=1$ and from
Lemma \ref{lem-hp}, it follows that (i) $\Longrightarrow$ (ii)
holds.
Next,
we show (ii) $\Longrightarrow$ (i)
by considering the following three
cases for both $n$ and $p$.

\emph{Case 1: $n=1$ and $p\in (1,\infty)$.}
In this case,
$\dot{Y}^{1,p}_{\omega}(\mathbb{R})=\dot{W}^{1,p}_{\omega}(\mathbb{R})$.
Let $\eta \in C_{\rm c}^{\infty}(\mathbb{R})$
satisfy ${\bf 1}_{[0,1]}\le \eta \le {\bf 1}_{[-2,2]}$,
$g\in C^{\infty}(\mathbb{R})$
be nonnegative, and, for any $x\in \mathbb{R}$,
$f(x):=\int_{-\infty}^x g(t)\eta (t)\,dt$.
Then $f'=g\eta \in C_{\rm c}^{\infty}(\mathbb{R})
\subset \dot{W}^{1,p}_{\omega}(\mathbb{R})$ and
\begin{align}\label{eq-f'}
\int_{\mathbb{R}}|f'(x)|^p \omega(x)\,dx
\le \int_{-2}^2 [g(x)]^p \omega(x)\,dx.
\end{align}
Notice that, for any $x\in [-1,0]$ and $h\in [2,3]$,
\begin{align*}
|\Delta_h f(x)|
=\int_{x}^{x+h}g(t)\eta (t)\,dt
\ge \int_{0}^1 g(t)\,dt
\end{align*}
and, for any $x\in [1,2]$ and $h\in [-3,-2]$,
\begin{align*}
|\Delta_h f(x)|
=\int_{x+h}^{x}g(t)\eta (t)\,dt
\ge \int_{0}^1 g(t)\,dt.
\end{align*}
From these, we infer that
\begin{align*}
&\int_{\mathbb{R}}
\left[\int_{|h|<3}
\frac{|\Delta_h f(x)|^q}{|h|^{1+sq}}\,dh
\right]^\frac{p}{q}\omega(x)\,dx\\
&\quad \ge
\int_{-1}^0
\left[\int_{2}^3
\left\{\int_{0}^1 g(t)\,dt\right\}^q
{|h|^{-1-sq}}\,dh
\right]^\frac{p}{q}\omega(x)\,dx\\
&\qquad +\int_{1}^2
\left[\int_{-3}^{-2}
\left\{\int_{0}^1 g(t)\,dt\right\}^q
{|h|^{-1-sq}}\,dh
\right]^\frac{p}{q}\omega(x)\,dx\\
&\quad \sim
\omega\left([-1,0]\cup [1,2]\right)	\left[\int_{0}^1 g(t)\,dt\right]^p.
\end{align*}
Using this, \eqref{eq-chap1} with $r:=3$,
and \eqref{eq-f'}, we find that
\begin{align*}
\omega\left([-1,0]\cup [1,2]\right)	\left[\int_{0}^1 g(t)\,dt\right]^p
\lesssim
\int_{-2}^2 [g(x)]^p \omega(x)\,dx.
\end{align*}
Repeating the proof of \cite[(3.43)]{DLYYZ23},
we conclude that, for any nonnegative
$g\in L^{1}_{\rm loc}(\mathbb{R})$,
\begin{align}\label{eq-g2}
\omega\left([-1,0]\cup [1,2]\right)	\left[\int_{0}^1 g(t)\,dt\right]^p
\lesssim
\int_{0}^1 [g(x)]^p \omega(x)\,dx.
\end{align}
From this with $g:=1 $, we deduce that
$
\omega ([-1,0])\le
\omega\left([-1,0]\cup [1,2]\right)
\lesssim
\omega ([0,1]).
$
Since \eqref{eq-chap1} is translation invariant
for the weight $\omega$ [that is, for any
$x_0\in \mathbb{R}$,
the weight $\omega (\cdot-x_0)$ satisfies \eqref{eq-chap1}
with the same positive constant], it follows that
$\omega ([0,1])\lesssim \omega ([1,2])
\le
\omega([-1,0]\cup [1,2]).$
This, together with \eqref{eq-g2},
further implies that, for any nonnegative function
$g\in L^{1}_{\rm loc}(\mathbb{R})$,
\begin{align*}
\left[\int_{0}^1 g(x)\,dx\right]^p
\lesssim
\frac{1}{\omega([0,1])}
\int_{0}^1 [g(x)]^p\omega(x)\,dx.
\end{align*}
By this and the facts
that \eqref{eq-chap1} is translation invariant
and dilation invariant [that is, for any $\delta\in (0,\infty)$,
the weight
$\omega (\delta\cdot)$ satisfies \eqref{eq-chap1}
with the same positive constant],
we conclude that, for any interval $I$ and
any nonnegative $g\in L^{1}_{\rm loc}(\mathbb{R})$,
$$
\left[\fint_{I}g(x)\,dx\right]^p \lesssim \frac{1}{\omega(I)}\int_{I}
[g(x)]^p\omega(x)\,dx,
$$
which, together with Lemma \ref{lem-apwight}(iv), further
implies
$\omega\in A_p $ in this case.

\emph{Case 2: $n\in \mathbb{N}$ and $p=1$.}
In this case, $\dot{Y}^{1,p}_{\omega}
=\dot{W}^{1,p}_{\omega} $.
Let $D_0:= B({\bf 0},6)\setminus B({\bf 0},3)$,
$D_1:= B({\bf 0},2)\setminus B({\bf 0},1)$,
$D_2:= B({\bf 0},18)\setminus B({\bf 0},9)$,
and
$D_3:= B({\bf 0},48)\setminus B({\bf 0},24)$.
Let
$g\in C^{\infty} $ be nonnegative and radial
and choose $\eta \in C^{\infty} $
such that $\eta $ is radial and
${\bf 1}_{D_0}\le \eta \le {\bf 1}_{B({\bf 0},8)\setminus B({\bf 0},2)}$.
For any $x\in \mathbb{R}^n$, let
$
f(x):= \int_{B({\bf 0},|x|)}g(y)\eta (y)\,dy.
$
Then, for any $x\in \mathbb{R}^n$, $\nabla f(x)=\mathcal{H}^{n-1}(\mathbb{S}^{n-1})g(x)\eta (x)|x|^{n-2}x$.
Thus,
\begin{align}\label{eq-c2}
\int_{\mathbb{R}^n}|\nabla f(x)|
\omega(x)\,dx \lesssim
\int_{B({\bf 0},8)\setminus B({\bf 0},2)}
|g(x)|\omega(x)\,dx .
\end{align}
Observe that, for any $x\in D_1$ and $h\in D_2 -\{x\}$,
it holds that
$7<|h|<20$ and
\begin{align*}
|\Delta_h f(x)|=
\int_{|x|\le |y|<|x+h|}g(y)\eta (y)\,dy
\ge \int_{D_0}g(y)\,dy;
\end{align*}
on the other hand, for any $x\in D_2$ and
$h\in D_1-\{x\}$, it also holds that
$7<|h|<20$ and
\begin{align*}
|\Delta_h f(x)|=
\int_{|x+h|\le |y|<|x|}g(y)\eta (y)\,dy
\ge \int_{D_0}g(y)\,dy.
\end{align*}
From these, we further infer that
\begin{align*}
&\int_{\mathbb{R}^n}
\left[\int_{|h|<20}
\frac{|\Delta_h f(x)|^q}{|h|^{1+sq}}\,dh
\right]^\frac{1}{q}\omega(x)\,dx\\
&\quad \ge
\int_{D_1}
\left[\int_{D_2-\{x\}}
\left\{\int_{D_0}g(y)\,dy\right\}^q {|h|^{-1-sq}}\,dh
\right]^\frac{1}{q}\omega(x)\,dx\\
&\qquad +
\int_{D_2}
\left[\int_{D_1-\{x\}}
\left\{\int_{D_0}g(y)\,dy\right\}^q {|h|^{-1-sq}}\,dh
\right]^\frac{1}{q}\omega(x)\,dx\\
&\quad \sim
\omega(D_1 \cup D_2)
\int_{D_0}g(y)\,dy.
\end{align*}
By this,  \eqref{eq-chap1} with $r:=20$,
and \eqref{eq-c2}, we find that
\begin{align}\label{eq-aaa}
\omega(D_1 \cup D_2)
\int_{D_0}g(y)\,dy
\lesssim
\int_{B({\bf 0},8)\setminus B({\bf 0},2)}
g(x)\omega(x)\,dx.
\end{align}
Now, repeating the proof of Case 2 of \cite[Proposition 5.1]{LYYZZ-arXiv} with
(5.12) therein replaced by \eqref{eq-aaa} here,
we obtain $\omega\in A_1$.

\emph{Case 3: $n\in \mathbb{N} \cap [2,\infty)$
and $p\in (1,\infty)$.}  In this case,
$\dot{Y}^{1,p}_{\omega} =\dot{H}^{1,p}_{\omega} $.
Let $g\in C^{\infty} $ be
nonnegative, $\eta \in C^{\infty} $
satisfy ${\bf 1}_{A_0}\le \eta
\le {\bf 1}_{B({\bf 0},8)\setminus B({\bf 0},2)}$,
and $f:= I_1 (g\eta )$, where $I_1$ is the same as in \eqref{df-ibeta}
with $\beta$ replaced by $1$.
Then, by \cite[pp.\,10--11]{Gra14m}
and the assumption $n\in \mathbb{N} \cap [2,\infty)$,
we find that $f\in L^{1}_{\rm loc} $
and, for any $x\in \mathbb{R}^n$,
\begin{align}\label{eq-Iff}
f(x)=\frac{\Gamma(\frac{n-1}{2})}{2 \pi^\frac{n}{2} \Gamma (\frac{1}{2})}
\int_{\mathbb{R}^n}
\frac{g(y)\eta (y)}{|x-y|^{n-1}}\,dy,
\end{align}
where $\Gamma (\cdot)$ denotes the Gamma function.
In addition, applying the semigroup property of Riesz
potentials (see, for instance, \cite[p.\,9]{Gra14m}) and (2.31),
we conclude that
\begin{align}\label{eq-cap2}
\left\| f\right\|_{\dot{H}^{1,p}_{\omega} }^p
&= \left\|I_{-1}f\right\|^p_{L^p_\omega }
=
\left\|I_{-1}I_1 (g\eta )\right\|^p_{L^p_\omega }\notag\\
&= \int_{\mathbb{R}^n}|g(y)\eta (y)|^p \omega (y)\,dy
\le
\int_{B({\bf 0},8)\setminus B({\bf 0},2)}|g(y)|^p \omega (y)\,dy.
\end{align}
Notice that, for any $x\in A_1$, $h \in A_3 -\{x\}$, and
$y\in B({\bf 0},8)\setminus B({\bf 0},2)$,
$
|x-y|\le |x|+|y|<10
$
and
$
|x+h-y|\ge |x+h|-|y|>16,
$
which, combined with \eqref{eq-Iff}, further imply that
\begin{align*}
|\Delta_h f(x)|
&\sim
\int_{B({\bf 0},8)\setminus B({\bf 0},2)}g(y)\eta (y)
\left(\frac{1}{|x-y|^{n-1}}-
\frac{1}{|x+h-y|^{n-1}}\right)\,dy\\
&\gtrsim
\int_{B({\bf 0},8)\setminus B({\bf 0},2)}g(y)\eta (y)
\,dy
\ge
\int_{A_0}g(y)
\,dy;
\end{align*}
similarly, for any $x \in A_3$ and $h\in A_1-\{x\}$,
$
|\Delta_h f(x)|
\gtrsim
\int_{A_0}g(y)
\,dy.
$
On the other hand,
observe that, for any $x\in A_1$ and $h \in A_3 -\{x\}$
or for any $x \in A_3$ and $h\in A_1-\{x\}$,
it holds that $22<|h|<50$.
From this, it follows that
\begin{align*}
&\int_{\mathbb{R}^n}
\left[\int_{|h|<50}
\frac{|\Delta_h f(x)|^q}{|h|^{1+sq}}\,dh
\right]^\frac{p}{q}\omega(x)\,dx\\
&\quad \gtrsim
\int_{A_1}
\left[\int_{A_3-\{x\}}
\left\{\int_{A_0}g(y)\,dy\right\}^q {|h|^{-1-sq}}\,dh
\right]^\frac{p}{q}\omega(x)\,dx\\
&\qquad +
\int_{A_3}
\left[\int_{A_1-\{x\}}
\left\{\int_{A_0}g(y)\,dy\right\}^q {|h|^{-1-sq}}\,dh
\right]^\frac{p}{q}\omega(x)\,dx\\
&\quad \sim
\omega(A_1 \cup A_3)
\left[\int_{A_0}g(y)\,dy\right]^p.
\end{align*}
Using this, \eqref{eq-chap1} with $r:=50$,
and \eqref{eq-cap2}, we find that
\begin{align}\label{eq-a13}
\omega(A_1 \cup A_3)
\left[\int_{A_0}g(y)\,dy\right]^p
\lesssim
\int_{B({\bf 0},8)\setminus B({\bf 0},2)}
[g(x)]^p\omega(x)\,dx.
\end{align}
Repeating the proof of Case 3 of \cite[Proposition 5.1]{LYYZZ-arXiv} with
(5.16) therein replaced by \eqref{eq-a13} here,
we obtain $\omega\in A_p$ in this case.
This finishes the proof that (ii) $\Longrightarrow$ (i)
and hence Theorem \ref{thm-ap}.
\end{proof}

\section{An Inequality Related to Gagliardo and Sobolev Seminorms \\
in Framework of Ball Banach Function Spaces $X(Q)$
over Cubes}\label{sec-bbf}

The main target of this section
is to establish an inequality related to
Gagliardo and Sobolev seminorms in the framework of
ball Banach function spaces
(see Theorem \ref{thm-QX}).
To do this, we first give some preliminary
concepts about ball Banach function spaces in Subsection \ref{s3.1},
which are frequently used in subsequent sections.
Moreover, in Subsection \ref{s3.1}, we also establish an extrapolation lemma
on ball Banach function spaces (see Lemma \ref{lem-po}), which
is one of the most important tools in subsequent sections.
Next, we prove an inequality related to
Gagliardo and Sobolev seminorms in Subsection \ref{s3.2}.

\subsection{Ball Banach Function Spaces and Extrapolation}\label{s3.1}

In this subsection, we give some preliminary
concepts about ball Banach function spaces and
establish a key extrapolation lemma (see Lemma \ref{lem-po}).
For any $\Omega\subset\mathbb{R}^n$,
let $\mathscr{M}(\Omega)$ denote the set of all
measurable functions on $\Omega$.
The definition
of ball Banach function spaces is as follows,
which was introduced in \cite[Definition 2.1]{SHYY17}.

\begin{definition}\label{def-X}
A Banach space $X \subset \mathscr{M} $,
equipped with a norm
$\|\cdot\|_{X }$ which makes sense for all functions in
$\mathscr{M} $,
is called a \emph{ball Banach function space}
(for short, \emph{{\rm BBF} space})
if $X $ satisfies that
\begin{enumerate}[{\rm (i)}]
\item for any $f\in \mathscr{M} $,
if $\|f\|_{X } =0$,
then $f=0$ almost
everywhere;
\item if $f,g\in\mathscr{M} $ satisfy that $|g|\leq
|f|$
almost everywhere, then
$\|g\|_{X } \leq \|f\|_{X }$;
\item if a sequence
$\{f_m\}_{m\in{\mathbb{N}}}$ in $\mathscr{M} $
satisfies that
$0\leq f_m \uparrow f$ almost everywhere as $m\to\infty$,
then $\|f_m\|_{X } \uparrow \|f\|_{X }$
as $m\to\infty$;
\item for any ball $B$ in $\mathbb{R}^n$,
${\bf 1}_{B}\in {X }$;

\item[${\rm (v)}$]
for any ball $B$ in $\mathbb{R}^n$,
there exists a positive constant $C_{(B)}$,
depending on $B$,
such that, for any $f\in {X }$,
$
\int_{B}|f(x)|\,dx \leq C_{(B)}\|f\|_{{X }}.
$
\end{enumerate}
\end{definition}

\begin{remark}
\begin{enumerate}[{\rm (i)}]
\item As mentioned in \cite[Remark 2.5(ii)]{YHYY22ams},
we obtain an equivalent formulation of
Definition \ref{def-X} via replacing any ball
$B$ in $\mathbb{R}^n$ by any bounded measurable set
$S$ in $\mathbb{R}^n$.

\item In Definition \ref{def-X}, if we replace any ball $B$
by any measurable $E$ with $|E|<\infty$, then
we obtain the definition of \emph{Banach function spaces},
which was originally
introduced by Bennett and Sharpley in
\cite[Chapter 1, Definitions 1.1 and 1.3]{BS88}.
Using their definitions,
we easily find that
a Banach function space is always a ball Banach function space.
However,
the converse is not necessarily true
(see, for instance, \cite[p.\,9]{SHYY17}).

\item In Definition \ref{def-X}, if we replace (iv)
by the \emph{saturation property}
that,
for any measurable set
$E$ in $\mathbb{R}^n$
with $|E|\in (0,\infty)$, there exists a measurable
set $F\subset E$ with $|F|\in (0,\infty)$ satisfying that
${\bf 1}_{F}\in X $,
then we obtain the definition of Banach function spaces in the terminology
of Lorist and Nieraeth \cite[p.\,251]{LN24im}. Moreover,
by \cite[Proposition 2.5]{ZYY23ccm}
(see also \cite[Proposition 4.21]{Nie23}),
we conclude that, if the normed vector space $X $ under
consideration satisfies the additional assumption that
the Hardy--Littlewood maximal
operator $\mathcal{M}$ is weakly bounded on one of its convexification,
then the definition of Banach function spaces in \cite{LN24im}
coincides with the definition of ball
Banach function spaces. Thus, under this additional assumption, working with
ball Banach function spaces in the sense of Definition \ref{def-X} or Banach
function spaces in the sense of \cite{LN24im} would
yield exactly the same results.

\item  From \cite[Proposition 1.2.36]{LYH22}, we deduce that
both (ii) and (iii) of Definition \ref{def-X}
imply that any ball Banach function space is complete.
\end{enumerate}
\end{remark}

For any $g\in \mathscr{M} $,
we denote by $g|_{\Omega}$ the restriction of $g$ to $\Omega$.
The following concept of restricted spaces
was introduced in \cite[Definition 2.6]{ZYY24jga}.
\begin{definition}
Let $\Omega\subset \mathbb{R}^n$ be
an open set and $X $ a {\rm BBF} space.
The \emph{restricted space} $X(\Omega)$ of
$X $
is defined by setting
$X(\Omega):=\{f\in \mathscr{M}(\Omega):f=g|_{\Omega} \
\text{for some}\ g\in X \}$.	
Moreover, for any $f\in X(\Omega)$, define
$
\|f\|_{X(\Omega)}:= \inf\left\{\|g\|_{X }:
f=g|_{\Omega},\  g\in X \right\}.
$
\end{definition}

\begin{remark}\label{rem-0}
Let $\Omega\subset \mathbb{R}^n$ be
an open set, $X$ a {\rm BBF} space,
and $X(\Omega)$ its
restricted space. From \cite[Proposition 2.7]{ZYY24jga},
it follows that,
for any $f\in X(\Omega)$,
$\|f\|_{X(\Omega)}=\|\widetilde{f}\|_{X }$,
where
\begin{align*}
\widetilde{f}(x):=
\begin{cases}
f(x) &\text{if}\  x\in \Omega,\\
0 &\text{if}\  x\in \Omega^{\complement}.
\end{cases}
\end{align*}
\end{remark}

Moreover, we need the following definition of the $p$-convexification of
a {\rm BBF} space given in \cite[Definition 2.6]{SHYY17}.

\begin{definition}\label{def-convex}
Let $X $ be a {\rm BBF} space and $p\in (0,\infty)$. The
\emph{$p$-convexification}
$X^p $ of $X $ is defined by setting
$X^p :=\{f\in \mathscr{M} :
|f|^p \in X \}$
and is equipped with the quasi-norm
$\left\|f\right\|_{X^p }
:=\left\|\,|f|^p\right\|_{X }^{\frac{1}{p}}$
for any $f\in X^p $.
\end{definition}

The following concept of the associate space of a
{\rm BBF} space can be found in
\cite[p.\,9]{SHYY17}; see \cite[Chapter 1, Section 2]{BS88}
for more details.

\begin{definition}\label{def-X'}
Let $\Omega\subset \mathbb{R}^n$ be
an open set, $X $ a {\rm BBF} space,
and $X(\Omega)$ its restricted space.
The \emph{associate space} (also called the
\emph{K\"{o}the dual}) $[X(\Omega)]'$ is defined by setting
\begin{align*}
[X(\Omega)]':=\left\{f\in \mathscr{M}(\Omega):\|f\|_{[X(\Omega)]'}:=
\sup_{\genfrac{}{}{0pt}{}{g\in X(\Omega)}{\|g\|_{X(\Omega)} =1}}\|fg\|_{L^1
(\Omega)}<\infty\right\},
\end{align*}
where $\|\cdot\|_{[X(\Omega)]'}$ is called the \emph{associate norm}
of $\|\cdot\|_{X(\Omega)}$.
\end{definition}

\begin{remark}\label{rem-x-a}
Let $\Omega\subset \mathbb{R}^n$ be an open set,
$X $ a {\rm BBF} space,
and $X(\Omega)$ its
restricted space.
\begin{enumerate}[{\rm (i)}]
\item From \cite[Proposition 2.3]{SHYY17}, we infer that the
associate space $X'$
is also a {\rm BBF} space.

\item By \cite[Proposition 2.12]{ZYY24jga},
we find that $X'(\Omega)=[X(\Omega)]'$
with the same norms, where $X'(\Omega)$
denotes the restricted space of $X'$ on $\Omega$.

\item From \cite[Lemma 2.6]{ZYYW21}, we deduce that
$X$ coincides with $X''$ (namely its second associate space)
with the same norms.
\end{enumerate}
\end{remark}

Now, we present the following definition
of ball Banach Sobolev
spaces on any open set of $\mathbb{R}^n$, which was introduced in
\cite[Definition 2.14]{ZYY24jga} (see also \cite[Definition 2.6]{DGPYYZ24}
for the corresponding space defined on $\mathbb{R}^n$).

\begin{definition}
Let $\Omega\subset \mathbb{R}^n$ be an open set,
$X $ a {\rm BBF} space, $X(\Omega)$ its
restricted space, and $k\in {\mathbb{N}}$.
\begin{enumerate}[{\rm (i)}]
\item The \emph{inhomogeneous ball Banach Sobolev space}
$W^{k,X}(\Omega)$ is defined to be the set
of all $f\in X(\Omega)$ such that,
for any multi-index
$\alpha\in \mathbb{Z}_{+}^n$ with
$|\alpha|\le k$,
the $\alpha$-th weak partial derivative
$\partial^{\alpha} f$ of $f$ exists and
$\partial^{\alpha} f \in X(\Omega)$.
Moreover, the \emph{norm} $\|\cdot\|_{{W}^{k,X}(\Omega)}$ of $W^{k,X}(\Omega)$
is defined by setting, for any $f\in W^{k,X}(\Omega)$,
$\|f\|_{{W}^{k,X}(\Omega)}
:=\sum_{\alpha\in \mathbb{Z}_{+}^n,\, |\alpha|\le k}
\|\partial^{\alpha} f\|_{X(\Omega)}.$

\item 	The \emph{homogeneous ball Banach Sobolev space}
$\dot{W}^{k,X}(\Omega)$ is defined to be the set
of all $f\in L^{1}_{\rm loc}(\Omega) $ such that,
for any multi-index
$\alpha\in \mathbb{Z}_{+}^n$ with
$|\alpha|= k$,
the $\alpha$-th weak partial derivative
$\partial^{\alpha} f$ of $f$ exists and
$\partial^{\alpha} f \in X(\Omega)$.
Moreover, the \emph{seminorm} $\|\cdot\|_{\dot{W}^{k,X}(\Omega)}$
of $\dot{W}^{k,X}(\Omega)$
is defined by setting, for any $f\in \dot{W}^{k,X}(\Omega)$,
$\|f\|_{\dot{W}^{k,X}(\Omega)}
:=\sum_{\alpha\in \mathbb{Z}_{+}^n,\, |\alpha|= k}
\|\partial^{\alpha} f\|_{X(\Omega)}.$
\end{enumerate}
\end{definition}

Finally, we establish the following
extrapolation lemma on BBF spaces,
which is one of the most important tools in subsequent sections;
for other variants of extrapolation lemmas on BBF spaces,
we refer to \cite{LN24jfaa,Nie23,TYYZ21}.

\begin{lemma}\label{lem-po}
Let $X $ be a {\rm BBF} space, $\Omega$ an open set,
and $X(\Omega)$ the
restricted space of $X$ on $\Omega$.
Assume that
there exists some $p\in [1,\infty)$
such that $X^{\frac{1}{p}} $ is
also a {\rm BBF} space
and the Hardy--Littlewood maximal operator
$\mathcal{M}$ is bounded on $Y:= (X^{\frac{1}{p}})'$.
If a nonnegative measurable function pair $(F,G)$ on $\Omega$
satisfies that $G\in X(\Omega)$ and that there exists an increasing function
$\phi:[0,\infty)\to [0,\infty)$ such that,
for any $\omega \in A_{1}$,
$\|F\|_{L^p_\omega (\Omega)}\le \phi([\omega]_{A_{1}})\|G\|_{L^p_\omega (\Omega)}$,
then
$$\|F\|_{X(\Omega)} \le 2^{\frac{1}{p}}\phi(2\|\mathcal{M}\|_{Y\to Y}) \|G\|_{X(\Omega)}.$$
\end{lemma}

To show this extrapolation lemma, we need the following
Rubio de Francia's iteration algorithm
and related equivalent norms of
{\rm BBF} spaces, which is precisely
\cite[Lemma 4.6 and Remark 4.7]{DLYYZ23}.

\begin{lemma}\label{lem-rxp}
Let $X $ be a {\rm BBF} space
satisfying the same assumptions as in Lemma \ref{lem-po}
with some $p\in[1,\infty)$.
For any $g\in Y:=(X^{\frac1p})'$, let
$
R_{Y}g:=\sum_{m=0}^{\infty}\frac{\mathcal{M}^m g}{2^m
\|\mathcal{M}\|_{Y\to
Y}^m},
$
where, for any $m\in {\mathbb{N}}$, $\mathcal{M}^m$ is the
$m$-fold iteration
of $\mathcal{M}$ and
$\mathcal{M}^0 g:=|g|$.
Then it holds that, for any $g\in Y$,
\begin{enumerate}[{\rm (i)}]
\item for any $x\in {\mathbb{R}^n}$, $|g(x)|\leq R_{Y}g(x),$
\item $R_{Y}g\in A_1  $ and $[R_{Y}g]_{A_1}\leq
2\|\mathcal{M}\|_{Y\to Y},$
\item $\|R_{X}g\|_Y \leq 2\|g\|_Y$.
\end{enumerate}
Moreover, if $f\in X $, then
\begin{align*}
\|f\|_X
\leq
\sup_{\|g\|_{Y}\leq
1}\left[\int_{{\mathbb{R}^n}}|f(x)|^p R_{Y}
g(x)\, dx\right]^{\frac{1}{p}}
\leq
2^{\frac{1}{p}}\|f\|_X.
\end{align*}
\end{lemma}

\begin{proof}[Proof of Lemma \ref{lem-po}]
Let $F\in \mathscr{M}(\Omega)$ and $G\in X(\Omega)$.
By Remark \ref{rem-0}, Definition \ref{def-convex}, Remark \ref{rem-x-a}(iii),
Definition \ref{def-X'}, and Lemma \ref{lem-rxp}, we find that
\begin{align*}
\|F\|_{X(\Omega)}
&=\left\|\widetilde{F}\right\|_{X}=
\left\|\left(\widetilde{F}\right)^p\right\|^{\frac{1}{p}}_{X^{\frac{1}{p}}}
=\left\|\left(\widetilde{F}\right)^p\right\|^{\frac{1}{p}}_{(X^{\frac{1}{p}})''}
=\sup_{\|g\|_{Y}\leq
1}\left[\int_{{\mathbb{R}^n}}\left|\widetilde{F}(x)\right|^p
g(x)\, dx\right]^{\frac{1}{p}}\\
&\le
\sup_{\|g\|_{Y}\leq
1}\left[\int_{\Omega}\left|F(x)\right|^p
R_Y g(x)\, dx\right]^{\frac{1}{p}}\le
\sup_{\|g\|_{Y}\leq 1}	\phi([R_Y g]_{A_{1}})\left[\int_{\Omega}\left|G(x)\right|^p
R_Y g(x)\, dx\right]^{\frac{1}{p}}\\
&\le
\phi(2\|\mathcal{M}\|_{Y \to Y })
\sup_{\|g\|_{Y}\leq 1}\left[\int_{\Omega}\left|G(x)\right|^p
R_Y g(x)\, dx\right]^{\frac{1}{p}}\\
&=
\phi(2\|\mathcal{M}\|_{Y \to Y })\sup_{\|g\|_{Y}\leq 1}\left[\int_{\mathbb{R}^n}\left|\widetilde{G}(x)\right|^p
R_Y g(x)\, dx\right]^{\frac{1}{p}}\le
2^{\frac{1}{p}}\phi(2\|\mathcal{M}\|_{Y \to Y }) \|G\|_{X(\Omega)}.
\end{align*}
This finishes the proof of Lemma \ref{lem-po}.
\end{proof}

\subsection{An Inequality Related to Gagliardo and Sobolev Seminorms\\
in Framework of $X(Q)$}\label{s3.2}

In this subsection, we mainly prove the following theorem.

\begin{theorem}\label{thm-QX}
Let $k\in \mathbb{N}$ and $X $ be a {\rm BBF} space.
Assume that $X$ satisfies the same assumptions as in Lemma \ref{lem-po}
with some $p\in [1,\infty)$.
If
$q\in [1,\infty)$
satisfies $n(\frac{1}{p}-\frac{1}{q})<k$, then
there exists a positive constant
$C$
such that, for any $s\in (0,1)$,
any cube $Q$ in $\mathbb{R}^n$, and any
$f\in \dot{W}^{k,X}(Q)$,
\begin{align}\label{eq-sssX}
(1-s)^{\gamma_{p,q}}
\left\|\left[\int_{Q(\cdot,k)}
\frac{|\Delta^k_h f(\cdot)|^q}{|h|^{n+skq}}\,dh
\right]^\frac{1}{q}\right\|_{X(Q)}
\le
C[\ell (Q)]^{(1-s)k}
\left\|\nabla^k f\right\|_{X(Q)},
\end{align}
where
$\ell (Q)$ denotes the edge length of
$Q$ and $\gamma_{p,q}$ is the same as in \eqref{Gqp}.
\end{theorem}

\begin{proof}
Let $Q$ be a cube in $\mathbb{R}^n$.
We first claim that
there exists an increasing function $\phi:[0,\infty)\to [0,\infty)$
such that, for any $s\in (0,1)$, $\omega\in A_1$, and $f\in \dot{W}^{k,p}_{\omega}(Q)$,
\begin{align}\label{eq-claim1}
(1-s)^{\gamma_{p,q}}
\left\{
\int_{Q}
\left[\int_{Q(x,k)}
\frac{|\Delta^k_h f(x)|^q}{|h|^{n+skq}}\,dh
\right]^\frac{p}{q}\omega(x)\,dx\right\}^{\frac{1}{p}}
\lesssim\phi([\omega]_{A_1})
{\ell (Q)}^{(1-s)k}
\left\|\nabla^k  f\right\|_{L^p_{\omega}(Q)},
\end{align}
where the implicit positive constant depends only on $n$, $k$,
$p$, and $q$.
To this end, fix $\omega\in A_1$.
By carefully checking the proof of
\cite[Theorem 1.2]{Chu92}, we find
that there exist an extension operator
$\Lambda:\dot{W}_{\omega}^{k,p}(Q)\to\dot{W}^{k,p}_\omega$ and an increasing
function $\psi:[0,\infty)\to[0,\infty)$ such that, for any
$f\in\dot{W}^{k,p}_\omega(Q)$, $\Lambda f=f$ almost everywhere in $Q$
and
\begin{align}\label{eq-FQ}
\left\|\nabla^k \Lambda f\right\|_{L^p_{\omega}}
\lesssim\psi([\omega]_{A_1})
\left\|\nabla^k f\right\|_{L^p_{\omega} (Q)}
\end{align}
with the implicit positive constant depending only
on $n$, $k$, and $p$.

On the other hand, assume that $c_Q$ denotes the center of $Q$.
Then, for any $x\in Q$ and $h\in Q(x,k)$,
we have $|h|\le |h+x-c_Q|+|x-c_Q|<2 \sqrt{n}\ell(Q)$.
Combining this, Theorem \ref{thm-upi}(i),
and \eqref{eq-FQ}, we find that,
for any $s\in (0,1)$ and $f\in\dot{W}^{k,p}_\omega(Q)$,
\begin{align*}
&(1-s)^{\gamma_{p,q}}
\left\{
\int_{Q}
\left[\int_{Q(x,k)}
\frac{|\Delta^k_h f(x)|^q}{|h|^{n+skq}}\,dh
\right]^\frac{p}{q}\omega(x)\,dx\right\}^{\frac{1}{p}}\\
&\quad\le
(1-s)^{\gamma_{p,q}}
\left\{
\int_{\mathbb{R}^n}
\left[\int_{B({\bf 0},2\sqrt{n}\ell(Q))}
\frac{|\Delta^k_h \Lambda f(x)|^q}{|h|^{n+skq}}\,dh
\right]^\frac{p}{q}\omega(x)\,dx\right\}^{\frac{1}{p}}\\
&	\quad\lesssim
[\omega]_{A_1}^{\frac{1}{p}}{\ell (Q)}^{(1-s)k}
\left\|\nabla^k \Lambda f\right\|_{L^p_{\omega}}
\lesssim
[\omega]_{A_1}^{\frac{1}{p}}
\psi([\omega]_{A_1})
{\ell (Q)}^{(1-s)k}
\left\|\nabla^k  f\right\|_{L^p_{\omega}(Q)}.
\end{align*}
Thus, the above claim \eqref{eq-claim1} holds with
$\phi(t):=t^{\frac1p}\psi(t)$ for any $t\in[0,\infty)$.
From \eqref{eq-claim1} and Lemma \ref{lem-po}, it follows
\eqref{eq-sssX}, which then completes the proof of Theorem \ref{thm-QX}.
\end{proof}

\begin{remark}\label{rm312}
Similar to the proof of Proposition \ref{gammaqp},
we can find that, in Theorem \ref{thm-QX},
the exponent $\gamma_{p,q}$
on $1-s$ is sharp.
\end{remark}

Next, with the aid of Theorem \ref{thm-QX},
we can show Theorem \ref{thm-Q}.

\begin{proof}[Proof of Theorem \ref{thm-Q}]
Applying Theorem \ref{thm-QX}, we find that,
to show the present theorem, we only need to
verify that $X:=L^p_\omega$ with $\omega\in A_p$
satisfies all the assumptions of Theorem \ref{thm-QX}.
To this end,
we consider the following two cases for $p$ and $\omega$.

\emph{Case 1: $p\in (1,\infty)$ and $\omega\in A_p \setminus A_1$}.
In this case, $p_\omega>1$.
By the assumption that $n(\frac{p_\omega}{p}-\frac{1}{q})<k$,
we find that there exists
$\epsilon \in (0, p- p_\omega)$ such that
$n(\frac{p_\omega +\epsilon}{p}-\frac{1}{q})<k$.
Using the definition of $p_{\omega}$, we obtain
$\omega\in A_{p_\omega +\epsilon}$.
From Definition \ref{def-convex}, we infer that
$[L^p_{\omega} ]^{\frac{p_\omega +\epsilon}{p}}=L^{p_\omega +\epsilon}_{\omega} $.
By this, $\omega\in A_{p_\omega +\epsilon}$,
and  \cite[p.\,86]{SHYY17}, we find that
both $L^p_{\omega} $ and $L^{p_\omega +\epsilon}_{\omega} $
are {\rm BBF} spaces.
Moreover,
from the proof of \cite[Theorem 5.10]{DGPYYZ24},
we deduce that the Hardy--Littlewood maximal operator $\mathcal{M}$ is
bounded on $[L^{p_\omega +\epsilon}_{\omega} ]'$.
Therefore, applying Theorem \ref{thm-QX}
with $X $ and $p$ therein replaced, respectively, by
$L^p_{\omega} $ and $\frac{p}{p_\omega +\epsilon}$ here,
we obtain \eqref{eq-sss3}.

\emph{Case 2: $p\in [1,\infty)$ and $\omega\in A_1$}.
In this case, $p_\omega=1$.
Similar to the proof of Case 1, we conclude that
$L^p_\omega$ with $\omega\in A_1$ satisfies all the assumptions of Theorem \ref{thm-QX}.
Thus, \eqref{eq-sss3} also holds in this case.
This finishes the proof of Theorem \ref{thm-Q}.
\end{proof}

\section{A Fractional Gagliardo--Nirenberg Inequality \\
in Framework of Ball Banach Function Spaces}\label{sec-pf1}

In this section, we establish the following
general fractional
Gagliardo--Nirenberg interpolation
inequalities on {\rm BBF} spaces.

\begin{theorem}\label{thm-in}
Let $k\in \mathbb{N}$ and
$X $ be a {\rm BBF} space.
Assume that $X$ satisfies the same assumptions as in Lemma \ref{lem-po}
with some $p\in [1,\infty)$.
\begin{enumerate}[{\rm (i)}]
\item If $q\in [1,p]$, then there exists
a positive constant $C$
such that,
for any $s\in (0,1)$ and $f\in W^{k,X}$,
\begin{align*}
s^{\frac{1}{q}}(1-s)^{\frac{1}{q}}\left\|
\left[\int_{\mathbb{R}^n}
\frac{|\Delta^k_h f(\cdot)|^q}{|h|^{n+skq}}\,dh
\right]^\frac{1}{q}\right\|_{X}
\le C
\left\|f\right\|^{1-s}_{X}
\left\|\nabla^k f\right\|^s_{X}.
\end{align*}

\item If $q\in (p,\infty)$ satisfies
$n(\frac{1}{p}-\frac{1}{q})<k$
and
$\theta \in
(0, 1-\frac{n}{k}[\frac{1}{p}-\frac{1}{q}])$, then there exists
a positive constant $C$
such that,
for any $s\in (\max\{1-\frac{\theta}{2},1-\frac{n}{kq}\},1)$
and $f\in W^{k,X}$,
\begin{align*}
(1-s)^{\gamma_{p,q}}\left\|
\left[\int_{\mathbb{R}^n}
\frac{|\Delta^k_h f(\cdot)|^q}{|h|^{n+skq}}\,dh
\right]^\frac{1}{q}\right\|_{X}
\le C
\left\|f\right\|^{1-s}_{X}
\left\|\nabla^k f\right\|^s_{X},
\end{align*}
where $\gamma_{p,q}$ is the same as in \eqref{Gqp}.
\end{enumerate}
\end{theorem}

\begin{remark}
Theorem \ref{thm-in} gives the fractional
Gagliardo--Nirenberg interpolation
inequality between $X$ and $\dot{W}^{k,X}$.
It would be interesting to
establish a corresponding interpolation
inequality between $X_1$ and $\dot{W}^{k,X_2}$
for different function spaces $X_1$ and $X_2$;
see also \cite[Remark 4.14(iii)]{DLYYZ23} for a related
question.
\end{remark}

To prove Theorem \ref{thm-in},
we first establish two weighted estimates.
Indeed, applying \cite[Lemma 3.15]{DGPYYZ24}
and repeating an argument similar to that used in the proof of \cite[Lemma 2.25]{PYYZ24}
with the first-order difference therein replaced by
the higher-order difference, we obtain
the following lemma, which plays a key role in the proof of
Theorem \ref{thm-in}(i);
we omit the details.

\begin{lemma}\label{eq-s1}
Let $\omega\in A_1$,
$k\in \mathbb{N}$,
and $1\le q\le p <\infty$.
Then there exists a positive constant $C$,
depending only
on $n,k,p$, and $q$, such that,
for any $r\in (0,\infty)$ and $f\in L^{p}_{\omega}$,
\begin{align*}
s^{\frac{1}{q}}
\left\{\int_{\mathbb{R}^n}\left[\int_{B({\bf 0},r)^{\complement}}
\frac{|\Delta^k_h f(x)|^q}{|h|^{n+skq}}\,dh
\right]^\frac{p}{q}\omega(x)\,dx\right\}^{\frac{1}{p}}
\le C
[\omega]_{A_{1}}^{\frac{2}{p}}
r^{-sk}
\left\| f\right\|_{L^{p}_{\omega}}.
\end{align*}
\end{lemma}

On the other hand, the following weighted
estimate involving derivatives is essential in the proof of
Theorem \ref{thm-in}(ii) and
is also of independent interest.

\begin{proposition}\label{lem-upw}
Let $\omega\in A_1$,
$k\in \mathbb{N}$,
$1\le p <q <\infty$
satisfy $n(\frac{1}{p}-\frac{1}{q})<k$,
and $\theta \in
[0, 1-\frac{n}{k}[\frac{1}{p}-\frac{1}{q}])$.
Assume $\rho$ is a nonnegative
and decreasing function
on $(c_0,\infty)$
with some $c_0\in(0,\infty)$.
If
\begin{align}\label{eq-conRho}
C_{(\rho)}:=\int_{c_0}^{\infty}
\left[\rho(r)r^{n-kq\theta}\right]
^{\frac{p}{q}}\,\frac{dr}{r}
<\infty,
\end{align}
then there exist a positive constant
$C_{(n)}$, depending only on $n$,
and a positive constant
$C_{(n,k,p,q)}$, depending only
on $n,k,p$, and $q$, such that,
for any $f\in {W}^{k,p}_{\omega} $,
\begin{align*}
&\int_{\mathbb{R}^n}\left[\int_{B({\bf0},C_{(n)}c_0)^\complement}
\frac{|\Delta^k_h f(x)|^q}{|h|^{kq}}\rho(|h|)\,dh
\right]^\frac{p}{q}\omega(x)\,dx\\
&\quad\le C_{(n,k,p,q)} C_{(\rho)}
[\omega]_{A_{1}}\left([\omega]_{A_1}+1\right)
\left\| f\right\|^{p\theta}_{L^{p}_{\omega}}
\left\|\nabla^k f\right\|^{p(1-\theta)}_{L^{p}_{\omega}}.
\end{align*}
\end{proposition}

To show this proposition, we need
some concepts and lemmas.
Let $\mathcal{P}_{s} $ denote the set of
all polynomials of degree not greater than $s\in \mathbb{Z}_+$ on
$\mathbb{R}^n$.
For any
$s\in \mathbb{Z}_+$, any ball $B\subset \mathbb{R}^n$,
and any $f\in L^{1}_{\rm loc} $,
let
$P^{(s)}_{B}(f)$ denote the \emph{minimizing polynomial} in
$\mathcal{P}_{s} $ such that, for any
$\alpha\in\mathbb{Z}_+^n$ with $|\alpha|\leq s$,
\begin{align*}
\int_{B}\left[f(x)-P^{(s)}_{B}(f)(x)\right]x^{\alpha}\,dx=0;
\end{align*}
for any cube $Q\subset \mathbb{R}^n$, $P^{(s)}_{Q}(f)$ is defined
in a similar way.
The following technical lemma about minimizing polynomials can
be easily deduced from the definition;
see also \cite[p.\,83]{TG80} and \cite[Lemma 4.1]{Lu95}.

\begin{lemma}\label{lem-unip}
Let $s\in\mathbb{Z}_+$,
$p\in [1,\infty)$, and $\Omega$
be a ball or a cube in $\mathbb{R}^n$.
Then, for any $f\in L^{1}_{\rm loc} $,
\begin{align*}
\left\|f-P^{(s)}_{\Omega}(f)\right\|_{L^p(\Omega)}
\sim
\inf_{P\in \mathcal{P}_s  }
\left\|f-P\right\|_{L^p(\Omega)},
\end{align*}
where the positive equivalence constants depend only on $n$ and $s$.
\end{lemma}

Moreover, we need the shifted
dyadic grids and their exquisite geometrical properties.
Recall that,
for any $\alpha\in\{0,\frac{1}{3},\frac{2}{3}\}^n$, the \emph{shifted
dyadic grid} $\mathcal{D}^\alpha$ is defined by setting
\begin{align*}
\mathcal{D}^\alpha:=\left\{2^j\left[m+[0,1)^n+(-1)^j\alpha\right]:
j\in\mathbb{Z},\ m\in\mathbb{Z}^n\right\}.
\end{align*}
The following properties (see, for instance, \cite[p.\,479]{MTT02})
of shifted dyadic grids on the Euclidean space play
key roles in the proof of Proposition \ref{lem-upw}.

\begin{lemma}\label{lem-dy-cub}
Let $\alpha\in\{0,\frac{1}{3},\frac{2}{3}\}^n$. On
$\mathcal{D}^\alpha$, the following statements hold.
\begin{enumerate}[{\rm(i)}]
\item
For any $Q,P\in\mathcal{D}^\alpha$ with
$\alpha\in\{0,\frac{1}{3},\frac{2}{3}\}^n$,
$Q\cap P\in\{\emptyset,Q,P\}$.
\item
There exists a positive constant $C_{(n)}$,
depending only on $n$, such that,
for any ball $B\subset\mathbb{R}^n$, there exist
$\alpha\in\{0,\frac{1}{3},\frac{2}{3}\}^n$ and
$Q\in\mathcal{D}^\alpha$
such that $B\subset Q\subset C_{(n)}B$.
\end{enumerate}
\end{lemma}

The following variant of the higher-order
Poincar\'e inequality is exactly \cite[Lemma 3.10]{HLYYZ-bsvy},
which is also needed in the proof of Proposition \ref{lem-upw}
(see also \cite[Lemma 4.13]{LYYZZ-arXiv}).

\begin{lemma}\label{Poin}
Let $k\in \mathbb{N}$ and $f\in L^1_{\rm loc} $.
Then there exists a positive constant $C_{(n,k)}$, depending only on
$n$ and $k$, such that,
for almost every $x\in\mathbb{R}^n$ and for any $r\in(0,\infty)$
and any ball $B_1\subset B:=B(x,r)\subset 3B_1$,
\begin{align*}
\left|f(x)-P^{(k-1)}_{B_1}(f)(x)\right|
\le C_{(n,k)}\sum_{j\in\mathbb{Z}_+}\fint_{B(x,2^{-j}r)}
\left|f(y)-P^{(k-1)}_{B(x,2^{-j}r)}(f)(y)\right|\,dy.
\end{align*}
\end{lemma}

\begin{proof}[Proof of Proposition \ref{lem-upw}]
Let $f\in {W}^{k,p}_{\omega} $.
Notice that, if $P\in \mathcal{P}_{k-1} $,
then $\Delta^k_h P(x)=0$
for any $x,h\in{\mathbb{R}^n}$.
Using this and \eqref{eq-sho-hd}, we find that,
for any $x,h\in {\mathbb{R}^n}$,
\begin{align*}
\left|\Delta_{h}^k f(x)\right|=
\left|\Delta_{h}^k \left(f-P^{(k-1)}_{B_{x,h,k}}(f)\right)(x)\right|
\leq\sum_{i=0}^{k}2^k
\left|\left(f-P^{(k-1)}_{B_{x,h,k}}(f)\right)(x+ih)\right|,
\end{align*}
where $B_{x,h,k}:=B(x+\frac{kh}{2},k|h|)$. From
this, it follows that, to show the present proposition,
it suffices to prove that, for any $i\in \{0,\dots,k\}$,
\begin{align}\label{eq-wup2}
{\rm I}_i:&=\int_{\mathbb{R}^n}\left[\int_{\mathbb{R}^n}
\frac{|f(x+ih)-P^{(k-1)}_{B_{x,h,k}}(f)(x+ih)|^q}{|h|^{kq}}\rho(|h|)\,dh
\right]^\frac{p}{q}\omega(x)\,dx\notag\\
&\lesssim
C_{(\rho)} [\omega]_{A_1}
\left([\omega]_{A_{1}}+1\right)
\left\| f\right\|^{p\theta}_{L^{p}_{\omega}
}
\left\|\nabla^k f\right\|^{p(1-\theta)}_{L^{p}_{\omega}
}
\end{align}
with the implicit positive constant depending only on
$n,k,p$, and $q$.

Fix $i\in \{0,\dots,k\}$ and $x\in\mathbb{R}^n$
and let $C_{(n)}$ be the same as in Lemma \ref{lem-dy-cub}
and $h\in B({\bf0},C_{(n)}c_0)^\complement$.
Using Lemma \ref{Poin} with $B_1:=B_{x,h,k}$
and $B:=B(x+ih,2k|h|)$, we find that
\begin{align}\label{eq-pigeonhole1}
&\left|f(x+ih)-P^{(k-1)}_{B_{x,h,k}}(f)(x+ih)\right|\notag\\
&\quad\le C_{(n,k)}\sum_{j\in\mathbb{Z}_+}
\fint_{B(x+ih,2^{-j+1}k|h|)}
\left|f(z)-P^{(k-1)}_{B(x+ih,2^{-j+1}k|h|)}(f)(z)\right|\,dz,
\end{align}
where $C_{(n,k)}$ is the same positive constant as in Lemma \ref{Poin}.
By the assumptions $n(\frac{1}{p}-\frac{1}{q})<k$
and $\theta \in
[0, 1-\frac{n}{k}[\frac{1}{p}-\frac{1}{q}])$, we can choose
\begin{align}\label{eq-ep}
\epsilon\in\left(0,
k[1-\theta]-n\left[\frac{1}{p}-\frac{1}{q}\right]\right),
\end{align}
and let $c_0:=\frac{1-2^{-\epsilon}}{ C_{(n,k)}}$.
Then, from \eqref{eq-pigeonhole1}, we infer that
\begin{align*}
&c_0\sum_{j\in\mathbb{Z}_+}2^{-j\epsilon}
\left|f(x+ih)-P^{(k-1)}_{B_{x,h,k}}(f)(x+ih)\right|\\
&\quad =
\frac{1}{C_{(n,k)}}
\left|f(x+ih)-P^{(k-1)}_{B_{x,h,k}}(f)(x+ih)\right|\\
&\quad\le\sum_{j\in\mathbb{Z}_+}\fint_{B(x+ih,2^{-j+1}k|h|)}
\left|f(z)-P^{(k-1)}_{B(x+ih,2^{-j+1}k|h|)}(f)(z)\right|\,dz,
\end{align*}
which further implies that
there exists $j_{x,h}\in\mathbb{Z}_+$ such that
\begin{align}\label{eq-claim-02}
&c_02^{-j_{x,h}\epsilon}
\left|f(x+ih)-P^{(k-1)}_{B_{x,h,k}}(f)(x+ih)\right|\notag\\
&\quad\le
\fint_{B(x+ih,2^{-j_{x,h}+1}k|h|)}
\left|f(z)-P^{(k-1)}_{B(x+ih,2^{-j_{x,h}+1}k|h|)}(f)(z)\right|\,dz.
\end{align}
In addition, applying Lemma \ref{lem-dy-cub},
we find that there exist
$\alpha_{x,h}\in \{0,\frac13,\frac23\}^n$ and
$Q_{x,h}\in \mathcal{D}^{\alpha_{x,h}}$
satisfying that
\begin{align}\label{eq-BQB}
B\left(x+ih,2^{-j_{x,h}+1}k|h|\right)\subset Q_{x,h}\subset
B\left(x+ih,2^{-j_{x,h}+1}C_{(n)}k|h|\right).
\end{align}
By this, Lemma \ref{lem-unip}, and the higher-order
Poincar\'{e} inequality (see, for instance, \cite[Theorem 13.27]{Leo17}),
we conclude that
\begin{align}\label{eq-pig2}
&\fint_{B(x+ih,2^{-j_{x,h}+1}k|h|)}
\left|f(z)-P^{(k-1)}_{B(x+ih,2^{-j_{x,h}+1}k|h|)}(f)(z)\right|\,dz\notag\\
&\quad\sim
\inf_{P\in \mathcal{P}_{k-1} }
\fint_{B(x+ih,2^{-j_{x,h}+1}k|h|)}
\left|f(z)-P(z)\right|\,dz\notag\\
&\quad \lesssim
\inf_{P\in \mathcal{P}_{k-1} }
\fint_{Q_{x,h}}
\left|f(z)-P(z)\right|\,dz
\lesssim
|Q_{x,h}|^{\frac{k}{n}}\fint_{Q_{x,h}}
\left|\nabla^k f(z)\right|\,dz.
\end{align}
On the other hand,
applying \eqref{eq-BQB} again
and \cite[Lemma 4.1]{Lu95},
we find that
\begin{align}\label{eq-pig3}
\fint_{B(x+ih,2^{-j_{x,h}+1}k|h|)}
\left|f(z)-P^{(k-1)}_{B(x+ih,2^{-j_{x,h}+1}k|h|)}(f)(z)\right|\,dz
\lesssim
\fint_{Q_{x,h}}
\left|f(z)\right|\,dz.
\end{align}
Furthermore, from \eqref{eq-BQB}, we deduce that
\begin{align*}
c_1\left|2^{j_{x,h}}Q_{x,h}\right|^{\frac1n}
:=\left[2C_{(n)}k\omega_n^{\frac1n}\right]^{-1}\left|2^{j_{x,h}}Q_{x,h}\right|^{\frac1n}
\le|h| \le\left(2k\omega_n^{\frac1n}\right)^{-1}
\left|2^{j_{x,h}}Q_{x,h}\right|^{\frac1n},
\end{align*}
where $\omega_n$ denotes the volume of the unit ball in $\mathbb{R}^n$.
This implies that
$c_1|2^{j_{x,h}}Q_{x,h}|^{\frac1n}
\ge C_{(n)}^{-1}|h|\ge c_0$.
Thus, combining this, \eqref{eq-claim-02},
\eqref{eq-pig2}, \eqref{eq-pig3}, and
the assumption that $\rho$ is decreasing on $(c_0,\infty)$,
further implies that
\begin{align}\label{eq-dom-p}
&\frac{|f(x+ih)-P^{(k-1)}_{B_{x,h,k}}(f)(x+ih)|^q}{|h|^{kq}}\rho(|h|)\notag\\
&\quad=\frac{|f(x+ih)-P^{(k-1)}_{B_{x,h,k}}(f)(x+ih)|
^{q\theta+q(1-\theta)}}{|h|^{kq}}\rho(|h|)\notag\\
&\quad\lesssim
2^{j_{x,h}q(\epsilon-k)}
\left[\left|Q_{x,h}\right|^{-\frac{k}{n}}\fint_{Q_{x,h}}
\left| f(z)\right|\,dz\right]^{q\theta}
\left[\fint_{Q_{x,h}}
\left|\nabla^k f(z)\right|\,dz\right]^{q(1-\theta)}
\rho\left(c_1 \left|2^{j_{x,h}} Q_{x,h}\right|^{\frac{1}{n}}\right).
\end{align}
Using \eqref{eq-BQB} again, we find that
\begin{align*}
(x,h) \in
\begin{cases}
\displaystyle{Q_{{x,h}}\times \left[\left(2^{j_{x,h}}Q_{{x,h}}\right)-\{x\}\right]}&
\text{if } i=0,\\
\displaystyle{\left(2^{j_{x,h}}Q_{{x,h}}\right)\times \frac{[Q_{{x,h}}-\{x\}]}{i}}
&\text{if }i\in \{1,\dots,k\}.
\end{cases}
\end{align*}
Applying this, \eqref{eq-dom-p}, and $Q_{x,h}\in \mathcal{D}^{\alpha_{x,h}}$,
we obtain
\begin{align*}
&\frac{|f(x)-P^{(k-1)}_{B_{x,h,k}}(f)(x)|^q}{|h|^{kq}}\rho(|h|)\\
&\quad \lesssim
\sum_{j\in\mathbb{Z_{+}}}2^{jq(\epsilon-k)}
\sum_{\alpha\in\{0,\frac{1}{3},\frac{2}{3}\}^n}\sum_{Q\in
\mathcal{D}^{\alpha}}
\left[|Q|^{-\frac{k}{n}}\fint_{Q}
\left|f(z)\right|\,dz\right]^{q\theta}
\left[\fint_{Q}
\left|\nabla^k f(z)\right|\,dz\right]^{q(1-\theta)}\notag\\
&\qquad \times
\rho\left(c_1 \left|2^j Q\right|^{\frac{1}{n}}\right)
{\bf 1}_{E_0}(x,x+h)\notag
\end{align*}
and, for any $i\in\{1,\ldots,k\}$,
\begin{align}\label{eq-point}
&\frac{|f(x+ih)-P^{(k-1)}_{B_{x,h,k}}(f)(x+ih)|^q}{|h|^{kq}}\rho(|h|)\notag\\
&\quad \lesssim
\sum_{j\in\mathbb{Z_{+}}}2^{jq(\epsilon-k)}
\sum_{\alpha\in\{0,\frac{1}{3},\frac{2}{3}\}^n}\sum_{Q\in
\mathcal{D}^{\alpha}}
\left[|Q|^{-\frac{k}{n}}\fint_{Q}
\left|f(z)\right|\,dz\right]^{q\theta}
\left[\fint_{Q}
\left|\nabla^k f(z)\right|\,dz\right]^{q(1-\theta)}\notag\\
&\qquad \times
\rho\left(c_1 \left|2^j Q\right|^{\frac{1}{n}}\right)
{\bf 1}_{E_i}(x,x+ih),
\end{align}
where
\begin{align}\label{eq-Eii}
E_{i}:=\begin{cases} Q\times 2^{j}Q & \text{if } i=0,\\
\left(2^{j}Q\right)\times Q & \text{if } i\in \{1,\dots,k\}.
\end{cases}
\end{align}
We then consider the following two cases for $i$.

\emph{Case 1: $i=0$.}
In this case,
from Lemma \ref{lem-dy-cub}(i), the assumption that $\rho$
is decreasing on $(c_0,\infty)$, and \eqref{eq-conRho}, it follows that,
for any $\alpha\in\{0,\frac{1}{3},\frac{2}{3}\}^n$,
$j\in \mathbb{Z}_+$, and $x\in \mathbb{R}^n$,
\begin{align*}
&\sum_{Q\in
\mathcal{D}^{\alpha}}
\left[\rho\left(c_1\left|2^j Q\right|
^{\frac{1}{n}}\right)\right]^{\frac{p}{q}}
\left| 2^j Q\right|^{\frac{p}{q}(1-\frac{kq\theta}{n})}{\bf 1}_{Q}(x)\\
&\quad
=\sum_{\nu\in \mathbb{Z},\,c_1 2^{j+\nu}\ge c_0}
\left[\rho\left(c_1 2^{j+\nu}\right)\right]
^{\frac{p}{q}}2^{\frac{p(n-kq\theta)(j+\nu)}{q}}\\
&\quad\lesssim
\sum_{\nu\in \mathbb{Z},\,c_1 2^{j+\nu}\ge c_0}
\int_{B({\bf 0},c_1 2^{j+\nu})\setminus B({\bf 0},c_1 2^{j+\nu-1})}
\left[\rho(|y|)|y|^{n-kq\theta}\right]^\frac{p}{q}\,\frac{dy}{|y|^n}\\
&\quad\sim
\int_{c_0}^{\infty}
\left[\rho(r)r^{n-kq\theta}\right]
^{\frac{p}{q}}\,\frac{dr}{r}=C_{(\rho)}.
\end{align*}
By this, the assumption $\omega\in A_1 \subset A_p$, Lemma \ref{lem-apwight}(iv),
and
Tonelli's theorem,
we find that, for any
$\alpha\in\{0,\frac{1}{3},\frac{2}{3}\}^n$
and
$j\in \mathbb{Z}_{+}$,
\begin{align}\label{eq-key}
{\rm A}_{\alpha,j}:&=
\sum_{Q\in
\mathcal{D}^{\alpha}}
\left[\fint_{Q}
\left| f(z)\right|\,dz\right]^{p}
\left[\rho\left(c_1 \left|2^j Q\right|^{\frac{1}{n}}\right)
\right]^{\frac{p}{q}}\left|2^j Q\right|^{\frac{p}{q}(1-\frac{kq\theta}{n})}
\omega(Q)\notag\\
&\le
[\omega]_{A_p}\sum_{Q\in
\mathcal{D}^{\alpha}}
\left[\int_{Q}
\left| f(z)\right|^p\omega(z)\,dz\right]
\left[\rho\left(c_1\left|2^j Q\right|^{\frac{1}{n}}\right)
\right]^{\frac{p}{q}}\left|2^j Q\right|^{\frac{p}{q}(1-\frac{kq\theta}{n})}
\notag\\
&\le[\omega]_{A_1}\int_{\mathbb{R}^n}\left\{\sum_{Q\in
\mathcal{D}^{\alpha}}
\left[\rho\left(c_1\left|2^j Q\right|^{\frac{1}{n}}\right)
\right]^{\frac{p}{q}}
\left| 2^j Q\right|^{\frac{p}{q}(1-\frac{kq\theta}{n})}
{\bf 1}_{Q}(z)\right\}
\left|f(z)\right|^p\omega(z)\,dz
\notag\\
&\lesssim
C_{(\rho)} [\omega]_{A_1}
\left\| f\right\|^{p}_{L^{p}_{\omega} }
\end{align}
and, similarly,
\begin{align}\label{eq-key2}
{\rm B}_{\alpha,j}
:&=	\sum_{Q\in
\mathcal{D}^{\alpha}}
\left[\fint_{Q}
\left| \nabla^k f(z)\right|\,dz\right]^{p}
\left[\rho\left(c_1 \left|2^j Q\right|^{\frac{1}{n}}\right)
\right]^{\frac{p}{q}}
\left|2^j Q\right|
^{\frac{p}{q}(1-\frac{kq\theta}{n})}	\omega(Q)\notag\\
&\lesssim
C_{(\rho)} [\omega]_{A_1}
\left\| \nabla^k f\right\|
^{p}_{L^{p}_{\omega} }.
\end{align}
From these, \eqref{eq-point}, \eqref{eq-Eii},
\eqref{eq-rrr},
H\"older's inequality,
and \eqref{eq-ep},
we infer that
\begin{align*}
{\rm I}_0
& \lesssim
\sum_{j\in\mathbb{Z_{+}}}2^{jp(\epsilon-k+k\theta)}
\sum_{\alpha\in\{0,\frac{1}{3},\frac{2}{3}\}^n}\sum_{Q\in
\mathcal{D}^{\alpha}}
\left[\fint_{Q}
\left| f(z)\right|\,dz\right]^{p\theta}
\left[\fint_{Q}
\left|\nabla^k f(z)\right|\,dz\right]^{p(1-\theta)}\\
&\quad \times
\left[\rho\left(c_1 \left|2^j Q\right|^{\frac{1}{n}}\right)
\right]^{\frac{p}{q}}
\left|2^j Q\right|^{\frac{p}{q}(1-\frac{kq\theta}{n})}\omega(Q)\\
&\le
\sum_{j\in\mathbb{Z_{+}}}2^{jp(\epsilon-k+k\theta)}
\sum_{\alpha\in\{0,\frac{1}{3},\frac{2}{3}\}^n}
\left({\rm A}_{\alpha,j}\right)^{\theta}
\left({\rm B}_{\alpha,j}\right)^{1-\theta}\lesssim
C_{(\rho)}[\omega]_{A_1}
\left\|f\right\|^{p\theta}_{L^{p}_{\omega}}
\left\|\nabla^k f\right\|^{p(1-\theta)}_{L^{p}_{\omega}}.\notag
\end{align*}
Thus, we obtain \eqref{eq-wup2} in this case.

\emph{Case 2: $ i\in \{1,\dots,k\}$.}
In this case, using
\eqref{eq-point}, \eqref{eq-Eii}, \eqref{eq-rrr},
$\omega\in A_{1}$,
H\"older's inequality, \eqref{eq-ep},
\eqref{eq-key},
and \eqref{eq-key2},
we conclude that, for any $i\in \{1,\dots,k\}$,
\begin{align*}
{\rm I}_{i}
&\lesssim
[\omega]_{A_{1}}
\sum_{j\in\mathbb{Z_{+}}}
2^{jp[\epsilon+\frac{n}{p}-\frac{n}{q}-k+k\theta]}
\sum_{\alpha\in\{0,\frac{1}{3},\frac{2}{3}\}^n}\sum_{Q\in
\mathcal{D}^{\alpha}}
\left[\fint_{Q}
\left| f(z)\right|\,dz\right]^{p\theta}
\\
&\quad\times
\left[\fint_{Q}
\left|\nabla^k f(z)\right|\,dz\right]^{p(1-\theta)}\left[\rho\left(\left|2^j Q\right|^{\frac{1}{n}}\right)\right]^{\frac{p}{q}}
\left|2^j Q\right|^{\frac{p}{q}(1-\frac{kq\theta}{n})}\omega\left( Q\right)\\
&\le
[\omega]_{A_{1}}
\sum_{j\in\mathbb{Z_{+}}}
2^{jp[\epsilon+\frac{n}{p}-\frac{n}{q}-k+k\theta]}
\sum_{\alpha\in\{0,\frac{1}{3},\frac{2}{3}\}^n}
\left({\rm A}_{\alpha,j}\right)^{\theta}
\left({\rm B}_{\alpha,j}\right)^{1-\theta}\\
&\lesssim
C_{(\rho)}[\omega]_{A_{1}}^2
\left\|
f\right\|^{p\theta}_{L^{p}_{\omega} }
\left\|\nabla^k f\right\|^{p(1-\theta)}_{L^{p}_{\omega} }.
\end{align*}
This further implies that \eqref{eq-wup2} also holds in this case,
which then completes the proof of
Proposition \ref{lem-upw}.
\end{proof}

Now, we turn to show Theorem \ref{thm-in}.

\begin{proof}[Proof of Theorem \ref{thm-in}]
We first prove (i).
To do this, let $q\in [1,p]$.
We claim that, for any $s\in (0,1)$, $\omega\in A_1$, and $f\in
W^{k,p}_{\omega}$,
\begin{align}\label{eq-ss001}
s^{\frac1q}(1-s)^{\frac1q}\left\{
\int_{\mathbb{R}^n}
\left[\int_{\mathbb{R}^n}
\frac{|\Delta^k_h f(x)|^q}{|h|^{n+skq}}\,dh
\right]^\frac{p}{q}\omega(x)\,dx\right\}^{\frac{1}{p}}
\lesssim [\omega]_{A_1}^{\frac2p}\left\|f\right\|^{1-s}_{L^p_{\omega}}
\left\|\nabla^k f\right\|^s_{L^p_{\omega}}.
\end{align}
If this claim holds, then,
by an argument similar to that used in the proof of Lemma \ref{lem-po},
we immediately conclude (i).
Thus, it remains to show the above claim.
Let $\omega\in A_1$ and $f\in W^{k,p}_{\omega}$
with $\|\nabla^k f\|_{L^p_\omega}\in (0,\infty)$.
Using \eqref{eq-sss00} and Lemma \ref{eq-s1},
we find that, for any $s\in (0,1)$ and $r\in (0,\infty)$,
\begin{align}\label{eq-ss003}
&\left\{
\int_{\mathbb{R}^n}
\left[\int_{\mathbb{R}^n}
\frac{|\Delta^k_h f(x)|^q}{|h|^{n+skq}}\,dh
\right]^\frac{p}{q}\omega(x)\,dx\right\}^{\frac{1}{p}}\notag\\
&\quad\le
\left\{
\int_{\mathbb{R}^n}
\left[\int_{B({\bf 0},r)}
\frac{|\Delta^k_h f(x)|^q}{|h|^{n+skq}}\,dh
\right]^\frac{p}{q}\omega(x)\,dx\right\}^{\frac{1}{p}}\notag\\
&\qquad+
\left\{
\int_{\mathbb{R}^n}
\left[\int_{B({\bf 0},r)^{\complement}}
\frac{|\Delta^k_h f(x)|^q}{|h|^{n+skq}}\,dh
\right]^\frac{p}{q}\omega(x)\,dx\right\}^{\frac{1}{p}}\notag\\
&\quad\lesssim
(1-s)^{-\frac{1}{q}}[\omega]_{A_{1}}
^{\frac{1}{p}}
r^{(1-s)k}
\left\|\nabla^k f\right\|_{L^p_{\omega}}
+
s^{-\frac{1}{q}}[\omega]_{A_{1}}^{\frac{2}{p}}
r^{-sk}
\left\| f\right\|_{L^{p}_{\omega}}.
\end{align}
Choose $r\in (0,\infty)$ satisfying
$(1-s)^{-\frac{1}{q}}[\omega]_{A_{1}}
^{\frac{1}{p}}
r^{(1-s)k}\|\nabla^k f\|_{L^p_{\omega}}=s^{-\frac{1}{q}}[\omega]_{A_{1}}^{\frac{2}{p}}
r^{-sk}
\| f\|_{L^{p}_{\omega}},$
which implies that
$$
r=\left[\frac{(1-s)}{s}\right]^{\frac{1}{kq}}
[\omega]_{A_1}^{\frac{1}{kp}}\left\| f\right\|_{L^{p}_{\omega}}^{\frac{1}{k}}
\left\|\nabla^k f\right\|_{L^p_{\omega}}^{-\frac{1}{k}}.
$$
Therefore, combining this, \eqref{eq-ss003},
and the fact that $[\omega]_{A_1}\ge 1$ (see, for instance,
\cite[Proposition 7.1.5(5)]{Gra14}),
we obtain,
for any $s\in (0,1)$,
\begin{align*}
&s^{\frac1q}(1-s)^{\frac1q}\left\{
\int_{\mathbb{R}^n}
\left[\int_{\mathbb{R}^n}
\frac{|\Delta^k_h f(x)|^q}{|h|^{n+skq}}\,dh
\right]^\frac{p}{q}\omega(x)\,dx\right\}^{\frac{1}{p}}\\
&\quad\lesssim
s^{\frac{s}{q}}(1-s)^{\frac{1-s}{q}}
[\omega]_{A_1}^{\frac{2-s}{p}}\left\|f\right\|^{1-s}_{L^p_{\omega}}
\left\|\nabla^k f\right\|^s_{L^p_{\omega}}
\le s^{\frac{s}{q}}(1-s)^{\frac{1-s}{q}}
[\omega]_{A_1}^{\frac{2}{p}}\left\|f\right\|^{1-s}_{L^p_{\omega}}
\left\|\nabla^k f\right\|^s_{L^p_{\omega}}.
\end{align*}
From this and the observation
$\sup_{s\in (0,1)}s^{\frac{s}{q}}(1-s)^{\frac{1-s}{q}}\lesssim 1$,
we deduce that
the above claim \eqref{eq-ss001} holds.
This then finishes the proof of (i).

Next, we prove (ii). To this end, let $q\in (p,\infty)$ satisfy
$n(\frac{1}{p}-\frac{1}{q})<k$
and let
$\theta \in
(0, 1-\frac{n}{k}[\frac{1}{p}-\frac{1}{q}])$.
Fix $\omega\in A_1$ and $r\in (0,\infty)$.
For any given
$s\in (\max\{1-\theta,1-\frac{n}{kq}\},1)$
and any
$t\in (0,\infty)$, define
\begin{align*}
\rho_{s}(t):=\begin{cases}
0\quad &\text{if }t\in (0,r],\\
t^{(1-s)kq-n}\quad &\text{if }t\in (r,\infty).
\end{cases}
\end{align*}
Then, for any $s\in (\max\{1-\theta,1-\frac{n}{kq}\},1)$,
$\rho_s$ is a nonnegative and decreasing function on $(r,\infty)$
and
\begin{align*}
C_{(\rho_s)}:&=	
\int_{r}^{\infty}
\left[\rho_s (t)t^{n-kq\theta}
\right]^{\frac{p}{q}}\,\frac{dt}{t}
=\int_{r}^{\infty}
t^{kp(1-s-\theta)}
\,\frac{dt}{t}=
\frac{r^{kp(1-s-\theta)}}{kp(\theta +s -1)}.\notag
\end{align*}
From this, Proposition \ref{lem-upw} with
$\rho:= \rho_s$, and the fact that $[\omega]_{A_1}\ge 1$ again,
it follows that there exists a positive constant
$C_{(n)}$,
depending only on $n$, such that,
for any $s\in (\max\{1-\theta,1-\frac{n}{kq}\},1)$
and $f\in W^{k,p}_{\omega}$,
\begin{align}\label{eq-rr}
&\left\{\int_{\mathbb{R}^n}\left[\int_{B({\bf 0},C_{(n)}r)^{\complement}}
\frac{|\Delta^k_h f(x)|^q}{|h|^{n+skq}}\,dh
\right]^\frac{p}{q}\omega(x)\,dx\right\}^{\frac{1}{p}}\notag\\
&\quad\lesssim
(\theta+s-1)^{-\frac{1}{p}}
[\omega]^{\frac{2}{p}}_{A_1}
r^{k(1-s-\theta)}
\left\| f\right\|^{\theta}_{L^{p}_{\omega}
}
\left\|\nabla^k f\right\|^{1-\theta}_{L^{p}_{\omega}
}.
\end{align}
Repeating the proof of (i) with Lemma \ref{eq-s1} replaced by
\eqref{eq-rr}, we find that,
for any $s\in (\max\{1-\theta,1-\frac{n}{kq}\},1)$,
\begin{align*}
&(1-s)^{\gamma_{p,q}}\left\{
\int_{\mathbb{R}^n}
\left[\int_{\mathbb{R}^n}
\frac{|\Delta^k_h f(x)|^q}{|h|^{n+skq}}\,dh
\right]^\frac{p}{q}\omega(x)\,dx\right\}^{\frac{1}{p}}\\
&\quad\lesssim
{(1-s)^{\frac{(1-s)\gamma_{p,q}}{\theta}}(\theta+s-1)^{\frac{s-1}{p\theta}}}
[\omega]_{A_1}^{\frac{1-s}{p\theta}}
\left\|f\right\|^{1-s}_{L^p_{\omega}}
\left\|\nabla^k f\right\|^s_{L^p_{\omega}}.\notag
\end{align*}
From this and the facts that $\sup_{s\in (\max\{1-\frac{\theta}{2},1-\frac{n}{kq}\},1)}{(1-s)^{\frac{(1-s)\gamma_{p,q}}
{\theta}}(\theta+s-1)^{\frac{s-1}{p\theta}}}\lesssim 1$
and $[\omega]_{A_1}\ge 1$ again,
we infer that, for any $s\in (\max\{1-\frac{\theta}{2},1-\frac{n}{kq}\},1)$,
\begin{align*}
(1-s)^{\gamma_{p,q}}\left\{
\int_{\mathbb{R}^n}
\left[\int_{\mathbb{R}^n}
\frac{|\Delta^k_h f(x)|^q}{|h|^{n+skq}}\,dh
\right]^\frac{p}{q}\omega(x)\,dx\right\}^{\frac{1}{p}}\lesssim
[\omega]_{A_1}^{\frac{1}{p}}
\left\|f\right\|^{1-s}_{L^p_{\omega}}
\left\|\nabla^k f\right\|^s_{L^p_{\omega}}.
\end{align*}
By this and Lemma \ref{lem-po},
we obtain (ii), which then completes the proof of
Theorem \ref{thm-in}.
\end{proof}

Now, we can show Theorem \ref{thm-ex}.

\begin{proof}[Proof of Theorem \ref{thm-ex}]
By the proof of Theorem \ref{thm-Q}, we find that
$X:=L^p_\omega$ with $p\in [1,\infty)$ and $\omega\in A_p$
satisfies all the assumptions of Lemma \ref{lem-po}.
This, combined with Theorem \ref{thm-in}, further
implies Theorem \ref{thm-ex}.
\end{proof}

\section{A BBM Formula and A New Characterization
of \\ Ball Banach Sobolev Spaces}\label{sec-pf2}

This section is divided into two subsections.
In Subsection \ref{s51},
we establish the sharp higher-order BBM formula
for BBF spaces on extension domains
(see Theorem \ref{thm-main}).
Later, we prove
the related characterization of ball Banach Sobolev spaces
in Subsection \ref{s52} (see Theorem \ref{thm-cwkx}).

\subsection{A BBM Formula on Ball Banach Function Spaces}\label{s51}

In this subsection, we show a
BBM formula on BBF spaces.
To state this result, we
first recall some definitions.
The following concept is precisely
\cite[Definition 3.2]{WYY20} (see also \cite[Definition 3.1]{BS88}).

\begin{definition}
Let $\Omega\subset \mathbb{R}^n$ be
an open set, $X $ a {\rm BBF} space,
and $X(\Omega)$ its
restricted space on $\Omega$.
Then $X(\Omega)$ is said to have an
\emph{absolutely continuous norm}
if, for any $f\in X(\Omega)$ and any sequence $\{E_{j}\}_{j\in{\mathbb{N}}}$
of measurable sets in $\Omega$ satisfying that ${\bf 1}_{E_j}\to 0$ almost
everywhere as $j\to \infty$,
$\|f{\bf 1}_{E_j}\|_{X(\Omega)} \to 0$ as $j\to \infty$.
\end{definition}

\begin{remark}\label{r-dual}
Let $\Omega\subset \mathbb{R}^n$ be
an open set, $X $ a {\rm BBF} space
having an absolutely continuous norm,
and $X(\Omega)$ its
restricted space.
\begin{enumerate}[{\rm (i)}]
\item From \cite[Proposition 2.10]{ZYY24jga},
we deduce that $X(\Omega)$ also has
an absolutely continuous norm.

\item
By \cite[Propositions 3.12 and 3.13]{LN24im}
and \cite[Remark 3.2(vi)]{ZLYYZ24},
we conclude that $X(\Omega)$ is separable.
Moreover, from \cite[Remark 2.13]{ZYY24jga}, it follows that
$[X(\Omega)]'$ coincides with $[X(\Omega)]^*$;
see also \cite[p.\,23, Corollary 4.3]{BS88}. Here, and thereafter,
we denote the \emph{dual space} of $X(\Omega)$ by $[X(\Omega)]^*$.
\end{enumerate}
\end{remark}

We next present the following concept of
ball Banach Sobolev extension domains
(see, for instance, \cite[Definition 2.17]{ZYY24jga}).

\begin{definition}\label{df-EXD}
Let $X $ be a {\rm BBF} space and $k\in \mathbb{N}$. An open set
$\Omega\subset \mathbb{R}^n$ is called a $W^{k,X}$-extension domain if
there exists a linear extension
operator $\Lambda: W^{k,X}(\Omega)\to W^{k,X} $
such that, for any
$f\in W^{k,X}(\Omega)$,
$(\Lambda f)|_{\Omega}=f$
and
$
\|\Lambda f\|_{W^{k,X} }
\le C\|f\|_{W^{k,X}(\Omega)}
$ with the positive constant $C$ independent
of $f$.
\end{definition}

\begin{remark}
Let $k\in \mathbb{N}$ and $X $ be a {\rm BBF} space satisfying
the same assumptions as in Lemma \ref{lem-po}.
As pointed out in \cite[Theorem 5.10]{ZYY24jga},
an $(\epsilon,\delta)$-domain
(see \cite[p.\,73]{Jon81} for its definition),
with $\epsilon\in (0,1]$ and $\delta\in (0,\infty]$,
is a $W^{k,X}$-extension domain.
In particular, the Lipschitz domain
(see, for instance, \cite[Definition 4.4]{EG15}),
the classical snowflake domain
(see, for instance, \cite[pp.\,104--105]{LV73}),
and the half-space are all $W^{k,X}$-extension domains.
\end{remark}

Based on these, we state the following BBM formula
on ball Banach function spaces.

\begin{theorem}\label{thm-main}
Let $k\in \mathbb{N}$, $X $ be a {\rm BBF} space having an absolutely continuous norm,
and $\Omega\subset \mathbb{R}^n$ a $W^{k,X}$-extension domain.
Assume that $X$ satisfies the same assumptions as in Lemma \ref{lem-po}
with some $p\in [1,\infty)$.
If $q\in [1,\infty)$ satisfies that $q=p=1$ or
$n(\frac{1}{p}-\frac{1}{q})<k$ when $p>1$,
then, for any $f\in W^{k,X}(\Omega)$,
\begin{align}\label{eq-main1}
&\lim_{s\to 1^{-}}
(1-s)^{\frac{1}{q}}
\left\|\left[\int_{\Omega(\cdot,k)}
\frac{|\Delta^k_h f(\cdot)|^q}{|h|^{n+skq}}\,dh
\right]^\frac{1}{q}\right\|_{X(\Omega)}\notag\\
&\quad=\frac{1}{(kq)^\frac{1}{q}}\left\|\left[ \int_{\mathbb{S}^{n-1}}
\left|\sum_{\alpha\in\mathbb{Z}_+^n,\,|\alpha|=k}
\partial^{\alpha}f(\cdot)\xi^{\alpha}\right|^q\,d\mathcal{H}^{n-1}(\xi)
\right]^{\frac{1}{q}}\right\|_{X(\Omega)}.
\end{align}
\end{theorem}

\begin{remark}\label{rem-main}
\begin{enumerate}[{\rm (i)}]
\item  It is worth pointing out that, for any $f\in W^{k,X}(\Omega)$,
\begin{align}\label{eq-sim}
\left\|\left[ \int_{\mathbb{S}^{n-1}}
\left|\sum_{\alpha\in\mathbb{Z}_+^n,\,|\alpha|=k}
\partial^{\alpha}f(\cdot)\xi^{\alpha}\right|^q\,d\mathcal{H}^{n-1}(\xi)
\right]^{\frac{1}{q}}\right\|_{X(\Omega)}
\sim \left\|\nabla^k f\right\|_{X(\Omega)}
\end{align}
with the positive equivalence constants depending only on $n$, $k$, and $q$;
in particular, when $k=1$, the left-hand side of \eqref{eq-sim}
equals
$[\int_{\mathbb{S}^{n-1}}|\xi \cdot e|^q\,
d\mathcal{H}^{n-1}(\xi)]^{\frac{1}{q}}\left\|\nabla f\right\|_{L^p (\Omega)}
$
with any fixed $e\in \mathbb{S}^{n-1}$
(see, for instance, \cite[Lemma 4.12]{HLYYZ-bsvy}
or \cite{FKR15}).

\item
Theorem \ref{thm-main} improves
\cite[Theorem 3.36]{DGPYYZ24}.
Precisely, in the case where $k=1$, $\Omega=\mathbb{R}^n$,
and $X:=L^{p} $ with $p\in [1,\infty)$,
Theorem \ref{thm-main}
extends the range $q\in [1,p]$
in \cite[Theorem 3.36]{DGPYYZ24}
into the range $q\in [1,\infty)$ satisfying $n(\frac{1}{p}-\frac{1}{q})<1$.
Moreover,
when $k\in\mathbb{N}\cap[2,\infty)$,
Theorem \ref{thm-main}
is new in the setting of ball Banach function spaces.

\item The assumption
$n(\frac{1}{p}-\frac{1}{q})<k$ in
Theorem \ref{thm-main} is \emph{sharp}; see Proposition \ref{pro-sp3}.
\end{enumerate}
\end{remark}

To prove Theorem \ref{thm-main},
we first give some auxiliary lemmas.
Indeed, repeating the proof of \cite[Theorem 3.5]{ZYY24jga} with
the first-order difference, $p$, and $\beta$ therein replaced, respectively, by
the higher-order difference, $q$, and $\frac{n}{p}$ here,
we can obtain the following
asymptotic formula for good functions,
which is important in the proof of Theorem \ref{thm-main};
we omit the details.

\begin{lemma}\label{lem-good}
Let $X $ be a {\rm BBF} space and
$\Omega\subset \mathbb{R}^n$ an open set.
Assume $\epsilon_0\in(0,\infty)$
and a family $\{\rho_{\epsilon}\}_{\epsilon\in (0,\epsilon_0)}$
of locally integrable functions on $(0,\infty)$ satisfies that,
for any $\epsilon\in (0,\epsilon_0)$, $\rho_{\epsilon}$
is nonnegative and decreasing,
\begin{align}\label{eq-ati1}
\int_{0}^{\infty}\rho_{\epsilon}(r)r^{n-1}\,dr=1,
\end{align}
and, for any $\delta\in (0,\infty)$,
\begin{align}\label{eq-ati2}
\lim_{\epsilon\to 0^+}\int_{\delta}^{\infty}
\rho_{\epsilon}(r)r^{n-1}\,dr=0.
\end{align}
Assume that  $p\in [1,\infty)$ such that,
for any $\lambda\in [1,\infty)$,
\begin{align*}
\left\|{\bf 1}_{B({\bf 0},\lambda)}\right\|_{X }
\lesssim
\lambda^{\frac{n}{p}}\left\|{\bf 1}_{B({\bf 0},1)}\right\|_{X }
\end{align*}
with the implicit positive constant independent of
$\lambda$.
If $k\in \mathbb{N}$
and $q\in [1,\infty)$ satisfies
$n(\frac{1}{p}-\frac{1}{q})<k$,
then, for any $f\in C_{\rm c}^{\infty} $,
\begin{align*}
\lim_{\epsilon\to 0^+}
\left\|\left[\int_{\Omega(\cdot,k)}
\frac{|\Delta^k_h f(\cdot)|^q}{|h|^{kq}}\rho_{\epsilon}(|h|)\,dh
\right]^\frac{1}{q}\right\|_{X(\Omega)}
=
\left\|\left[ \int_{\mathbb{S}^{n-1}}
\left|\sum_{\alpha\in\mathbb{Z}_+^n,\,|\alpha|=k}
\partial^{\alpha}f(\cdot)\xi^{\alpha}\right|^q\,d\mathcal{H}^{n-1}(\xi)
\right]^{\frac{1}{q}}\right\|_{X(\Omega)}.
\end{align*}
\end{lemma}

Moreover, we need the other asymptotic formula as follows.

\begin{lemma}\label{lem-outBall}
Let $X $ be a {\rm BBF} space
satisfying the same assumptions as in Lemma \ref{lem-po}
with some $p\in [1,\infty)$.
Assume that $k\in \mathbb{N}$ and $q\in [1,\infty)$ satisfies
$n(\frac{1}{p}-\frac{1}{q})<k$.
If $\Omega$ is a $W^{k,X}$-extension domain,
then, for any $f\in W^{k,X}(\Omega)$,
\begin{align*}
\lim_{s\to 1^{-}}
(1-s)^{\frac{1}{q}}
\left\|\left[\int_{\{h\in \Omega(\cdot,k):|h|\ge 1\}}
\frac{|\Delta^k_h f(\cdot)|^q}{|h|^{n+skq}}\,dh
\right]^\frac{1}{q}\right\|_{X(\Omega)}=0.
\end{align*}
\end{lemma}

\begin{proof}
Let $\Lambda$ be an extension operator
from $W^{k,X}(\Omega)$ to $W^{k,X}$,
$f\in W^{k,X}(\Omega)$,
and $\theta \in
[0, \min\{1,1-\frac{n}{k}(\frac{1}{p}-\frac{1}{q})\})$.
From this, Lemma \ref{eq-s1}, and \eqref{eq-rr},
we infer that,
for any $s\in (\max\{1-\frac{\theta}{2},1-\frac{n}{kq}\},1)$
and $\Lambda f\in W^{k,p}_{\omega}$,
\begin{align*}
&\left\{\int_{\mathbb{R}^n}\left[\int_{B({\bf 0},1)^{\complement}}
\frac{|\Delta^k_h \Lambda f(x)|^q}{|h|^{n+skq}}\,dh
\right]^\frac{p}{q}\omega(x)\,dx\right\}^{\frac{1}{p}}\\
&\quad\lesssim
[\omega]_{A_1}^{\frac{2}{p}}
\max\left\{\left\| \Lambda f\right\|^{\theta}_{L^p_\omega}
\left\|\nabla^k \Lambda f\right\|^{1-\theta}_{L^p_\omega},
\left\| \Lambda f\right\|_{L^p_\omega}\right\}
\le[\omega]_{A_1}^{\frac{2}{p}}
\left\|\Lambda f\right\|_{W^{k,p}_\omega}.\notag
\end{align*}
By this, $\Lambda f|_{\Omega}=f$,
and an argument
similar to that used in the proof of Lemma \ref{lem-po},
we conclude that
\begin{align*}
(1-s)^{\frac{1}{q}}\left\|\left[\int_{\{h\in \Omega(\cdot,k):|h|\ge 1\}}
\frac{|\Delta^k_h f(\cdot)|^q}{|h|^{n+skq}}\,dh
\right]^\frac{1}{q}\right\|_{X(\Omega)}
&\le
(1-s)^{\frac{1}{q}}
\left\|\left[\int_{B({\bf 0},1)^{\complement}}
\frac{|\Delta^k_h \Lambda f(\cdot)|^q}{|h|^{n+skq}}\,dh
\right]^\frac{1}{q}\right\|_{X }\\
&\lesssim (1-s)^{\frac{1}{q}}
\left\|\Lambda f\right\|_{W^{k,X}}
\lesssim(1-s)^{\frac{1}{q}}
\left\|f\right\|_{W^{k,X}(\Omega)}
\to 0
\end{align*}
as $s\to 1^{-}$. This finishes the proof of
Lemma \ref{lem-outBall}.
\end{proof}

Finally, to show Theorem \ref{thm-main},
we also need a density result
on ball Banach Sobolev spaces on extension domains.
Recall that, for any open set $\Omega$,
the restriction to $\Omega$ of the space
$C_{\rm c}^{\infty} $ is defined by
setting
$
C_{\rm c}^{\infty} |_{\Omega}
: =\left\{f|_{\Omega}:f\in C_{\rm c}^{\infty} \right\}.
$
Moreover, we
use the \emph{symbol}
$W^{k,X}_{\rm c} $
to denote the set of all functions in
$W^{k,X} $ with compact support. Then we have the following conclusion.

\begin{lemma}\label{lem-dens}
Let $k\in \mathbb{N}$ and $X $
be a {\rm BBF} space having an absolutely
continuous norm.
If $X $
satisfies the same assumptions as in Lemma \ref{lem-po}
with some $p\in [1,\infty)$,
then the following assertions hold.
\begin{enumerate}[{\rm (i)}]
\item  $W^{k,X}_{\rm c} $ and $C_{\rm c}^{\infty} $
are both dense in $W^{k,X} $.

\item If $\Omega\subset \mathbb{R}^n$
is a $W^{k,X}$-extension domain,
then $C_{\rm c}^{\infty} |_{\Omega}$
is dense in $W^{k,X}(\Omega)$.
\end{enumerate}
\end{lemma}

\begin{proof}
We first prove (i). To this end,
let $f\in W^{k,X} $ and
$\phi\in C_{\rm c}^{\infty} $
be such that ${\bf 1}_{B({\bf0},1)}\le\phi\le{\bf 1}_{B({\bf0},2)}$.
For any $l\in \mathbb{N}$, let
$
f_l (\cdot):= f(\cdot)\phi({\cdot}/{l}).
$
Then, applying an argument similar to that used in
the proof of \cite[Proposition 3.6]{DGPYYZ24}
with the first-order derivatives therein replaced
by the higher-order derivatives,
we find that, for any $l\in\mathbb{N}$,
$f_{l}\in W^{k,X}_{\rm c} $ and
$
\lim_{l\to \infty}
\|f_l -f\|_{W^{k,X} }=0.
$
Thus, $W^{k,X}_{\rm c} $
is dense in $W^{k,X} $.
Now, assume $g\in W^{k,X}_{\rm c} $
and let
$\eta\in C_{\rm c}^{\infty} $ be
such that ${\rm supp}\,(\eta)\subset B({\bf 0},1)$
and $\int_{\mathbb{R}^n}\eta(x)\,dx =1$
and, for any $j\in \mathbb{N}$, let
$\eta_j(\cdot):= j^n\eta(j\cdot)$.
Then
$g\ast \eta_j \in C_{\rm c}^{\infty} $.
Repeating the proof of
\cite[Proposition 3.8]{DGPYYZ24}, we conclude that
$
\lim_{j\to \infty}
\|g\ast \eta_j-g\|_{W^{k,X} }=0
$
and hence
$C_{\rm c}^{\infty} $ is dense
in $W^{k,X}_{\rm c} $. From this
and the above just showed fact that $W^{k,X}_{\rm c} $
is dense in $W^{k,X} $,
we deduce that (i) holds.

We next prove (ii). To do this, let
$f\in W^{k,X} (\Omega)$. Since $\Omega$ is
a $W^{k,X}$-extension domain,
it follows that there
exists $g\in W^{k,X} $ such that $g|_{\Omega}=f$.
By (i), we find that there
exists a sequence $\{g_j\}_{j\in \mathbb{N}}$ in $C_{\rm c}^{\infty} $
satisfying $\|g_j -g\|_{W^{k,X} }\to 0$ as $j\to \infty$.
For any $j\in \mathbb{N}$,
let $f_j:= g_j |_{\Omega} \in C_{\rm c}^{\infty} |_{\Omega}$.
From this and the definition of
$W^{k,X}(\Omega)$, we infer that
$
\|f_j -f\|_{W^{k,X}(\Omega)}
\le
\|g_j -g\|_{W^{k,X} }\to 0
$
as $j\to \infty$.
This finishes the proof of (ii)
and hence Lemma \ref{lem-dens}.
\end{proof}

Now, we turn to show Theorem \ref{thm-main}.

\begin{proof}[Proof of Theorem \ref{thm-main}]
Let
$\Omega$ be a $W^{k,X}$-extension domain with $\Lambda$ being the
corresponding extension operator
and let
$f\in W^{k,X}(\Omega)$.
By this, Theorem \ref{thm-upi}(i) with $r:=1$, and Lemma \ref{lem-po},
we find that, for any
$s\in (0,1)$,
\begin{align*}
(1-s)^{\frac{1}{q}}\left\|
\left[\int_{B({\bf 0},1)}
\frac{|\Delta^k_h \Lambda f(\cdot)|^q}{|h|^{n+skq}}\,dh
\right]^\frac{1}{q}\right\|_{X }
\lesssim
\left\|\nabla^k \Lambda f\right\|_{X},
\end{align*}
which, together with the definition of $W^{k,X}$-extension
domains, further implies that
\begin{align*}
(1-s)^{\frac{1}{q}}	\left\|\left[\int_{\Omega(\cdot,k)\cap
B({\bf 0},1)}
\frac{|\Delta^k_h f(\cdot)|^q}{|h|^{n+skq}}\,dh
\right]^\frac{1}{q}\right\|_{X(\Omega)}
\lesssim\|f\|_{ W^{k,X}(\Omega)}
\end{align*}
with the implicit positive constants independent of
$s$ and $f$.
On the other hand, from Lemma \ref{lem-dens}, we deduce that,
for any $\delta\in (0,\infty)$,
there exists $\eta\in C_{\rm c}^{\infty} $
such that
$
\|f-\eta |_{\Omega}\|_{W^{k,X}(\Omega)}<\delta.
$
Using Lemma \ref{lem-rxp},
the assumptions that $X^{\frac{1}{p}} $
is a {\rm BBF} space and
that $\mathcal{M}$ is
bounded on $Y:= (X^{\frac{1}{p}})'$, and
Lemma \ref{lem-apwight}(i),
we conclude that, for any $\lambda\in [1,\infty)$,
\begin{align*}
\left\|{\bf 1}_{B({\bf 0},\lambda)}\right\|_{X }
&\sim
\sup_{\|g\|_{Y}=1}\left[R_{Y}g
\left(B({\bf 0},\lambda)\right)\right]^{\frac{1}{p}}\lesssim
\lambda^{\frac{n}{p}}	\sup_{\|g\|_{Y}=1}
\left[R_{Y}g\left(B({\bf 0},1)\right)\right]^{\frac{1}{p}}
\sim
\lambda^{\frac{n}{p}}\left\|{\bf 1}_{B({\bf 0},1)}\right\|_{X }.
\end{align*}
From this and Lemma \ref{lem-good} with
\begin{align}\label{eq-sRHO}
\rho_{\epsilon}(r):=\begin{cases}
kq\epsilon r^{\epsilon k q-n} \quad &\text{if }r\in (0,1),\\
0 &\text{if }r\in [1,\infty)
\end{cases}
\end{align}
for any $r\in (0,\infty)$ and any
given $\epsilon\in (0,\frac{n}{kq})$,
it follows that
\begin{align}\label{thm-maine1}
&\lim_{s\to 1^{-}}
(1-s)^{\frac{1}{q}}
\left\|\left[\int_{\Omega(\cdot,k)\cap B({\bf 0},1)}
\frac{|\Delta^k_h \eta(\cdot)|^q}{|h|^{n+skq}}\,dh
\right]^\frac{1}{q}\right\|_{X(\Omega)}\notag\\
&\quad=
\frac{1}{(kq)^\frac{1}{q}}\left\|\left[ \int_{\mathbb{S}^{n-1}}
\left|\sum_{\alpha\in\mathbb{Z}_+^n,\,|\alpha|=k}
\partial^{\alpha}\eta(\cdot)\xi^{\alpha}\right|^q\,d\mathcal{H}^{n-1}(\xi)
\right]^{\frac{1}{q}}\right\|_{X(\Omega)}.
\end{align}
This, combined with a standard density argument (see, for instance,
\cite[p.\,1732]{DGPYYZ24}), Remark \ref{rem-main}(i),
the triangle inequality of the norm $\|\cdot\|_{X(\Omega)}$,
and the arbitrariness of $\delta$,
further implies that \eqref{thm-maine1} also holds
with $\eta$ replaced by $f$.
By this and Lemma \ref{lem-outBall},
we obtain
\eqref{eq-main1}.
This then finishes the proof of Theorem \ref{thm-main}.
\end{proof}

\subsection{A New Characterization of Ball Banach Sobolev Spaces}\label{s52}

In this subsection, we use the BBM formula
(namely Theorem \ref{thm-main}) to establish the following
new characterization of ball Banach Sobolev spaces.

\begin{theorem}\label{thm-cwkx}
Let $k\in \mathbb{N}$, $X $ be a {\rm BBF} space,
and $\Omega\subset \mathbb{R}^n$ a $W^{k,X}$-extension domain.
Assume that both $X$ and $X' $ have absolutely continuous norms
and
$X$ satisfies the same assumptions as in Lemma \ref{lem-po}
with some $p\in (1,\infty)$.
Let $q\in [1,\infty)$ satisfy
$n(\frac{1}{p}-\frac{1}{q})<k$.
Then $f\in W^{k,X}(\Omega)$ if and only if $f\in W^{k-1,X}(\Omega)$
[where $W^{0,X}(\Omega):=X(\Omega)$] and
\begin{align}\label{eq-cha01}
I(f):=\liminf_{s\to 1^{-}}
(1-s)^{\frac{1}{q}}
\left\|\left[\int_{\Omega(\cdot,k)}
\frac{|\Delta^k_h f(\cdot)|^q}{|h|^{n+skq}}\,dh
\right]^\frac{1}{q}\right\|_{X(\Omega)}<\infty;
\end{align}
moreover, for any $f\in W^{k,X}(\Omega)$,
$\|f\|_{W^{k,X}(\Omega)}\sim \|f\|_{W^{k-1,X}(\Omega)}+I(f)$
with the positive equivalence constants independent of $f$.
\end{theorem}

\begin{remark}
\begin{enumerate}[{\rm (i)}]
\item Theorem \ref{thm-cwkx} improves
\cite[Corollary 4.10]{ZYY24jga}.
Indeed, in the case where $k=1$
and $X:=L^{p} $ with $p\in [1,\infty)$,
Theorem \ref{thm-cwkx}
extends the range $q\in [1,p]\cup (1,\frac{n}{n-1}]$
in \cite[Theorem 3.36]{DGPYYZ24}
into the range $q\in [1,\infty)$ satisfying $n(\frac{1}{p}-\frac{1}{q})<1$.
Furthermore,
when $k\in\mathbb{N}\cap[2,\infty)$,
Theorem \ref{thm-cwkx}
is new in the setting of ball Banach function spaces.
\item By Proposition \ref{pro-sp3} below, we find that
the assumption
$n(\frac{1}{p}-\frac{1}{q})<k$ in
Theorem \ref{thm-cwkx} is \emph{sharp}.
\end{enumerate}
\end{remark}

To prove Theorem \ref{thm-cwkx},
we need the following well-known inequality
(see, for instance \cite[p.\,699]{Bre02}).

\begin{lemma}\label{lem-thet}
Let $q\in [1,\infty)$. If $\theta\in(0,1)$, then there exists a
positive constant $C_{(\theta)}$,
depending only on $\theta$, such that, for any $a,b \in (0,\infty)$,
$
(a+b)^q \leq (1+\theta)a^q+C_{(\theta)}b^q.
$
\end{lemma}

Furthermore, the following
technical lemma is also needed.

\begin{lemma}\label{lem-point}
Let $k\in \mathbb{N}$, $q\in [1,\infty)$,
$\Omega \subset \mathbb{R}^n$ be an open set,
$\epsilon_0\in(0,\infty)$,
and $\{\rho_{\epsilon}\}_{\epsilon\in (0,\epsilon_0)}$
be as in Lemma \ref{lem-good}.
If $f\in C^{\infty}(\Omega)$
and $U \subset\Omega $ is a bounded open set
satisfying $\overline{U}\subset \Omega$, then,
for any $x\in U$,
\begin{align}\label{eq-pointwise}
\lim_{\epsilon\to 0^+}
\int_{B({\bf 0},R_x )}
\frac{|\Delta^k_h f(x)|^q}{|h|^{kq}}\rho_{\epsilon}(|h|)\,dh
=\int_{\mathbb{S}^{n-1}}
\left|\sum_{\alpha\in\mathbb{Z}_+^n,\,|\alpha|=k}
\partial^{\alpha}f(x)\xi^{\alpha}\right|^q\,d\mathcal{H}^{n-1}(\xi),
\end{align}
where $R_x:= k^{-1}{\rm dist}(x, U^{\complement})$.
\end{lemma}

\begin{proof}
Let $\epsilon\in (0,\epsilon_0)$, $f\in C^{\infty}(\Omega)$, and $x\in U$.
Since $U$ is bounded, we infer that,
for any $x\in U$,
\begin{align}\label{eq-Rx}
kR_x \le
{\rm diam}\, (U):=
\sup \{|y-z|:y,z\in U\}
<\infty.
\end{align}
From this, the assumption
$f\in C^{\infty}(\Omega)$, and
Taylor's formula, we deduce that
there exists a positive constant $C_{(k,f,U)}$,
depending only on $k$, $f$,
and $U$, such that,
for any $h\in B({\bf 0},R_x )$,
\begin{align*}
\left|\Delta^k_h
f(x)-\sum_{\alpha\in\mathbb{Z}_+^n,\,|\alpha|=k}
\partial^{\alpha}f(x)h^{\alpha}\right|
\le C_{(k,f,U)}|h|^{k+1}.
\end{align*}
This, together with Lemma \ref{lem-thet},
implies that, for any $\theta \in (0,1)$,
there exists a positive
constant $C_{(\theta)}$,
depending only on $\theta$, such that,
for any $h\in B({\bf 0},R_x )$,
\begin{align*}
\left|\Delta^k_h f(x)\right|^q
\le (1+\theta)
\left|\sum_{\alpha\in\mathbb{Z}_+^n,\,|\alpha|=k}
\partial^{\alpha}f(x)h^{\alpha}\right|^q
+C_{(\theta)}[C_{(k,f,U)}]^q|h|^{(k+1)q}.
\end{align*}
By this, the polar coordinate,
and \eqref{eq-Rx}, we obtain,
for any $\theta \in (0,1)$,
\begin{align}\label{eq-I1}
&\int_{B({\bf 0},R_x )}
\frac{|\Delta^k_h f(x)|^q}{|h|^{kq}}\rho_{\epsilon}(|h|)\,dh\notag\\
&\quad\le
(1+\theta)\int_{ B({\bf 0},R_x )}
\left|\sum_{\alpha\in\mathbb{Z}_+^n,\,|\alpha|=k}
\partial^{\alpha}f(x)h^{\alpha}\right|^q{|h|^{-kq}}
\rho_{\epsilon}(|h|)\,dh\notag\\
&\qquad +
C_{(\theta)}[C_{(k,f,U)}]^q
\int_{B({\bf 0},R_x )}
|h|^{q}\rho_{\epsilon}(|h|)\,dh\notag\\
&\quad\le
(1+\theta)\int_{\mathbb{S}^{n-1}}
\left|\sum_{\alpha\in\mathbb{Z}_+^n,\,|\alpha|=k}
\partial^{\alpha}f(x)\xi^{\alpha}\right|^q\,d\mathcal{H}^{n-1}(\xi)
\int_{0}^{{\rm diam}\,(U)}\rho_{\epsilon}(r)r^{n-1}\,dr\notag\\
&\qquad+
C_{(\theta)}[C_{(k,f,U)}]^q\mathcal{H}^{n-1}
\left(\mathbb{S}^{n-1}\right)\int_{0}^{{\rm diam}\,(U)}
\rho_{\epsilon}(r)r^{q+n-1}\,dr.
\end{align}
From \eqref{eq-ati1} and \eqref{eq-ati2},
it follows that
\begin{align}\label{eq-pn1}
\lim_{\epsilon\to 0^+}\int_{0}^{{\rm diam}\,(U)}
\rho_{\epsilon}(r)r^{n-1}\,dr=1
\end{align}
and, for any $\delta \in (0, {\rm diam}\,(U))$,
\begin{align*}
\int_{0}^{{\rm diam}\,(U)}
\rho_{\epsilon}(r)r^{q+n-1}\,dr
&=
\int_{0}^{\delta}
\rho_{\epsilon}(r)r^{q+n-1}\,dr
+\int_{\delta}^{{\rm diam}\,(U)}
\rho_{\epsilon}(r)r^{q+n-1}\,dr\\
&\le \delta^q \int_{0}^{\infty}
\rho_{\epsilon}(r)r^{n-1}\,dr
+\left[{\rm diam}\,(U)\right]^q\int_{\delta}^{\infty}
\rho_{\epsilon}(r)r^{n-1}\,dr\to \delta^q
\end{align*}
as $\epsilon\to 0^+$, which, together with
the arbitrariness of $\delta$, further implies that
\begin{align*}
\lim_{\epsilon\to 0^+}\int_{0}^{{\rm diam}\,(U)}
\rho_{\epsilon}(r)r^{q+n-1}\,dr=0.
\end{align*}	
Applying this, \eqref{eq-I1},
and \eqref{eq-pn1},  letting $\epsilon\to 0^+$,
and then letting $\theta\to 0^+$, we obtain
\begin{align*}
\limsup_{\epsilon\to 0^+}
\int_{B({\bf 0},R_x )}
\frac{|\Delta^k_h f(x)|^q}{|h|^{kq}}\rho_{\epsilon}(|h|)\,dh
\le\int_{\mathbb{S}^{n-1}}
\left|\sum_{\alpha\in\mathbb{Z}_+^n,\,|\alpha|=k}
\partial^{\alpha}f(x)\xi^{\alpha}\right|^q\,d\mathcal{H}^{n-1}(\xi).
\end{align*}
Similarly, we have
\begin{align*}
\liminf_{\epsilon\to 0^+}
\int_{B({\bf 0},R_x )}
\frac{|\Delta^k_h f(x)|^q}{|h|^{kq}}\rho_{\epsilon}(|h|)\,dh
\ge\int_{\mathbb{S}^{n-1}}
\left|\sum_{\alpha\in\mathbb{Z}_+^n,\,|\alpha|=k}
\partial^{\alpha}f(x)\xi^{\alpha}\right|^q\,d\mathcal{H}^{n-1}(\xi).
\end{align*}
Therefore, we obtain \eqref{eq-pointwise},
which then completes the proof of Lemma \ref{lem-point}.
\end{proof}

Based on this lemma,
we next obtain the following conclusion, which
also plays a key role in the proof of
Theorem \ref{thm-cwkx}.

\begin{lemma}\label{lem-cha}
Let $X $ be a {\rm BBF} space
satisfying the same assumptions as in Lemma \ref{lem-po}
with some $p\in (1,\infty)$.
Assume that both $X $ and
$X' $ have absolutely continuous norms.
Let $k\in \mathbb{N}$ and $q\in [1,\infty)$ satisfy
$n(\frac{1}{p}-\frac{1}{q})<k$, $\Omega$ be an open set,
$\epsilon_0\in(0,\infty)$,
and
$\{\rho_{\epsilon}\}_{\epsilon\in (0,\epsilon_0)}$ be
as in Lemma \ref{lem-good}.
If
$f\in W^{k-1,X}(\Omega)$ [where $W^{0,X}(\Omega):=X(\Omega)$] and
\begin{align}\label{eq-cha1}
{\rm I}:=\liminf_{\epsilon\to 0^+}
\left\|\left[\int_{\Omega(\cdot,k)}
\frac{|\Delta^k_h f(\cdot)|^q}{|h|^{kq}}\rho_{\epsilon}(|h|)\,dh
\right]^\frac{1}{q}\right\|_{X(\Omega)}<\infty,
\end{align}
then $f\in W^{k,X}(\Omega)$ and $\|\nabla^k f\|_{X(\Omega)}\lesssim {\rm I}$.
\end{lemma}

\begin{proof}
Let
$f\in W^{k-1,X}(\Omega)$ satisfy \eqref{eq-cha1} and
$\varphi\in C_{\rm c}^{\infty} $ be a nonnegative
function satisfying both
$\int_{\mathbb{R}^n}\varphi(x)\,dx=1$
and ${\rm supp}\,(\varphi)\subset B({\bf 0},1)$.
For any $t\in (0,\infty)$, let
$\varphi_t := \frac{1}{t^{n}}\varphi(\frac{\cdot}{t})$
and
$$\Omega_t:=
\left\{x\in \Omega:{\rm dist}\,
\left(x,\Omega^{\complement}\right)>t\right\}.$$
For any $x\in \mathbb{R}^n$ and $t\in (0,\infty)$,
let $r_{x,t}:=k^{-1}{\rm dist}\,(x, \Omega_t^{\complement})$
and
$r_{x}:=k^{-1}{\rm dist}\,(x,\Omega^{\complement})$.
We claim that,
for any $t\in (0,\infty)$ and $\epsilon\in(0,\epsilon_0)$,
\begin{align}\label{eq-con}
&\left\|\left[\int_{B({\bf 0},r_{\cdot,t})}
\frac{|\Delta^k_h (\varphi_t \ast
f)(\cdot)|^q}{|h|^{kq}}\rho_{\epsilon}(|h|)\,dh
\right]^\frac{1}{q}{\bf 1}_{\Omega_t}\right\|_{X }\lesssim
\left\|\left[\int_{B({\bf 0},r_{\cdot})}
\frac{|\Delta^k_h f(\cdot)|^q}{|h|^{kq}}\rho_{\epsilon}(|h|)\,dh
\right]^\frac{1}{q}{\bf 1}_{\Omega}\right\|_{X }.
\end{align}
Indeed, observe that, for any $t\in (0,\infty)$,
$x\in \Omega_t $, and $y\in B({\bf 0},t)$,
\begin{align*}
{\rm dist}\,\left(x-y,\Omega^{\complement}\right)
\ge{\rm dist}\,\left(x,\Omega^{\complement}\right)
-|y|>0,
\end{align*}
which further implies that $x-y\in \Omega$
and hence
\begin{align}\label{eq-xy}
{\bf 1}_{\Omega_t}(\cdot)
\le {\bf 1}_{\Omega}(\cdot-y).
\end{align}
On the other hand,
for any $t\in (0,\infty)$,
$x\in \Omega_t $, $y\in B({\bf 0},t)$, and $z\in \Omega^{\complement}$,
we have $y+z\in \Omega_t^{\complement}$ and
\begin{align*}
|x-y-z|=|x-(y+z)|\ge {\rm dist}\, \left(x,\Omega_t^{\complement}\right)
\end{align*}
and hence
$
{\rm dist}\,(x-y,\Omega^{\complement})
\ge {\rm dist}\,(x,\Omega_t^{\complement}),
$
which further implies $r_{x,t}\le r_{x-y}$.
This, together with
the fact that ${\rm supp}\,(\varphi_t)\subset B({\bf 0},t)$, \eqref{eq-xy},
Minkowski's inequality,
\cite[Proposition 2.7]{Duo01},
and the assumption $\int_{\mathbb{R}^n}\varphi(x)\,dx=1$,
further implies that
\begin{align}\label{eq-con2}
&\left\|\left[\int_{B({\bf 0},r_{\cdot,t})}
\frac{|\Delta^k_h (\varphi_t
\ast f)(\cdot)|^q}{|h|^{kq}}\rho_{\epsilon}(|h|)\,dh
\right]^\frac{1}{q}{\bf 1}_{\Omega_t}(\cdot)\right\|_{X } \notag\\
&\quad\le
\left\|\left[\int_{B({\bf 0},r_{\cdot,t})}
\left\{
\int_{B({\bf 0},t)}\varphi_t (y)
\frac{|\Delta^k_h (f)(\cdot-y)|}{|h|^{k}}
[\rho_{\epsilon}(|h|)]^{\frac{1}{q}}
\,dy\right\}^q\,dh
\right]^\frac{1}{q}{\bf 1}_{\Omega_t}(\cdot)
\right\|_{X } \notag\\
&\quad\le
\left\|\int_{B({\bf 0},t)}
\varphi_t (y)
\left\{
\int_{B({\bf 0},r_{\cdot,t})}
\frac{|\Delta^k_h (f)(\cdot-y)|^q}{|h|^{kq}}
\rho_{\epsilon}(|h|){\bf 1}_{\Omega_t}(\cdot)\,dh
\right\}^{\frac{1}{q}}\,dy
\right\|_{X }
\notag\\
&\quad\le
\left\|\int_{B({\bf 0},t)}
\varphi_t (y)
\left\{
\int_{\mathbb{R}^n}
\frac{|\Delta^k_h (f)(\cdot-y)|^q}{|h|^{kq}}
\rho_{\epsilon}(|h|)
{\bf 1}_{B({\bf 0},r_{\cdot-y})}(h)
\,dh
\right\}^{\frac{1}{q}}	{\bf 1}_{\Omega}(\cdot-y)\,dy
\right\|_{X }
\notag\\
&\quad
= \left\|\varphi_t \ast F\right\|_{X }
\le\left\|\mathcal{M}\left(F\right)
\right\|_{X},
\end{align}
where, for any $x\in \mathbb{R}^n$,
\begin{align}\label{eq-FFF}
F(x):=\left[\int_{B({\bf 0},r_{x})}
\frac{|\Delta^k_h f(x)|^q}{|h|^{kq}}\rho_{\epsilon}(|h|)\,dh
\right]^\frac{1}{q}{\bf 1}_{\Omega}(x).
\end{align}
From \cite[Lemma 3.7]{ZYY23ccm} and
the assumptions that
$p\in (1,\infty)$,
$X^{\frac{1}{p}} $ is a {\rm BBF} space,
and $\mathcal{M}$ is bounded on $ (X^{\frac{1}{p}})'$,
we infer that $\mathcal{M}$ is bounded on $X $.
By this, \eqref{eq-con2}, and \eqref{eq-FFF},
we find that the above claim \eqref{eq-con}
holds.

Notice that, for any $x\in \Omega$,
$B({\bf 0},r_x)\subset \Omega(x,k)$.
Thus, using this, \eqref{eq-con},
and \eqref{eq-cha1},
we obtain, for any $t\in (0,\infty)$,
\begin{align}\label{eq-cha3}
\liminf_{\epsilon\to 0^+}
\left\|\left[\int_{B({\bf 0},r_{\cdot,t})}
\frac{|\Delta^k_h (\varphi_t
\ast f)(\cdot)|^q}{|h|^{kq}}\rho_{\epsilon}(|h|)\,dh
\right]^\frac{1}{q}\right\|_{X(\Omega_t)} \lesssim {\rm I}.
\end{align}
For any $t\in (0,\infty)$,
let $U_t \subset \Omega_t$
be a bounded open set satisfying
$\overline{U_t}\subset \Omega_t$ and,
for any $x\in \Omega$, let
$r'_{x,t}:=k^{-1}{\rm dist}(x,U_t^\complement)$.
From this and Lemma \ref{lem-point},
we deduce that, for any $t\in (0,\infty)$ and
$x\in U_t$,
\begin{align*}
\lim_{\epsilon\to 0^+}
\int_{B({\bf 0},r'_{x,t})}
\frac{|\Delta^k_h f(x)|^q}{|h|^{kq}}\rho_{\epsilon}(|h|)\,dh
=\int_{\mathbb{S}^{n-1}}
\left|\sum_{\alpha\in\mathbb{Z}_+^n,\,|\alpha|=k}
\partial^{\alpha}f(x)\xi^{\alpha}\right|^q\,d\mathcal{H}^{n-1}(\xi).
\end{align*}
By this and Fatou's lemma on $X(\Omega)$
(see, for instance, \cite[Lemma 3.7]{ZYY24jga}),
we obtain, for any $t\in (0,\infty)$,
\begin{align*}
&\left\|\left[\int_{\mathbb{S}^{n-1}}
\left|\sum_{\alpha\in\mathbb{Z}_+^n,\,|\alpha|=k}
\partial^{\alpha}\left(\varphi_t \ast f\right)(\cdot)\xi^{\alpha}\right|^q
\,d\mathcal{H}^{n-1}(\xi)
\right]^\frac{1}{q}{\bf 1}_{U_t}\right\|_{X(\Omega_t)}\\
&\quad\le
\liminf_{\epsilon\to 0^+}
\left\|\left[\int_{B({\bf 0}, r'_{\cdot,t})}
\frac{|\Delta^k_h (\varphi_t \ast f)(\cdot)|^q}{|h|^{kq}}
\rho_{\epsilon}(|h|)\,dh
\right]^\frac{1}{q}{\bf 1}_{U_t}\right\|_{X(\Omega_t)},
\end{align*}
which, combined with the arbitrariness of
$U_t$ and \cite[Proposition 2.8]{ZYY24jga},
further implies that
\begin{align}\label{eq-cha11}
&\left\|\left[\int_{\mathbb{S}^{n-1}}
\left|\sum_{\alpha\in\mathbb{Z}_+^n,\,|\alpha|=k}
\partial^{\alpha}\left(\varphi_t \ast f\right)(\cdot)\xi^{\alpha}\right|^q
\,d\mathcal{H}^{n-1}(\xi)
\right]^\frac{1}{q}\right\|_{X(\Omega_t)}\\
&\quad\le
\liminf_{\epsilon\to 0^+}
\left\|\left[\int_{B({\bf 0},r_{\cdot,t})}
\frac{|\Delta^k_h (\varphi_t \ast f)(\cdot)|^q}{|h|^{kq}}
\rho_{\epsilon}(|h|)\,dh
\right]^\frac{1}{q}\right\|_{X(\Omega_t)}.
\notag
\end{align}
Let $U\subset \Omega$ be an open set satisfying $\overline{U} \subset \Omega$
and choose $t_0 \in (0,\infty)$ sufficiently small such that
$U\subset \Omega_{t_0} $.
From this, the monotonicity of
$\Omega_t$ with respect to $t$,
Remark \ref{rem-main}(i), \eqref{eq-cha11},
and \eqref{eq-cha3}, it follows that,
for any $t\in (0,t_0)$,
\begin{align}\label{e-vtf}
\left\|\nabla^k(\varphi_t \ast f)\right\|_{X(U)}
\lesssim
{\rm I}.
\end{align}
By the assumptions that $X $
and $X' $ have absolutely continuous
norms and Remarks \ref{rem-x-a} and \ref{r-dual}, we find that
both $X(\Omega)$ and $[X(\Omega)]'$ have absolutely continuous norms,
$X(\Omega)$ is reflexive and separable,
$X(\Omega)=[X(\Omega)]^{**}$, and $[X(\Omega)]^{*}=[X(\Omega)]'$.
These, together with \eqref{e-vtf} and the Banach--Alaoglu theorem
(see, for instance, \cite[Theorem 3.17]{Rud91}), further imply that,
for any given $\alpha\in \mathbb{Z}_+^n$ with $|\alpha|=k$, there
exist a sequence
$\{t_j\}_{j\in \mathbb{N}}
$ in $ (0,t_0)$ and $g_{\alpha}\in X(U)$
such that $t_j\to 0^{+}$ as $j\to \infty$ and, for any $\Phi\in [X(U)]'$,
\begin{align}\label{e-p-integral}
\int_{U}\partial^{\alpha}(\varphi_{t_j} \ast
f)(x)\Phi(x)\,dx
\to
\int_{U}g_{\alpha}(x)\Phi(x)\,dx
\end{align}
as $j\to \infty$.
By this, Definition \ref{def-X'}, and Remark \ref{rem-x-a}(iii),
we conclude that
\begin{align}\label{e-piX}
\left\|g_{\alpha}\right\|_{X(U)}
&=
\sup_{\|\Phi\|_{[X(U)]'}=1}
\left|\int_{U}g_{\alpha}(x)\Phi(x)\,dx \right|=
\sup_{\|\Phi\|_{[X(U)]'}=1}
\lim_{j\to \infty}
\left|\int_{U}\partial^{\alpha}(\varphi_{t_j} \ast
f)(x)\Phi(x)\,dx \right|\notag\\
&\le
\sup_{\|\Phi\|_{[X(U)]'}=1}
\sup_{j\in \mathbb{N}}
\left|\int_{U}\partial^{\alpha}(\varphi_{t_j} \ast
f)(x)\Phi(x)\,dx \right|\le
\sup_{j\in \mathbb{N}}	\left\|\nabla^k(\varphi_{t_j} \ast f)\right\|_{X(U)}
\lesssim
{\rm I}.
\end{align}
Moreover, from Remark \ref{rem-x-a}(ii) and
\cite[Proposition 2.8]{ZYY24jga} again, we infer that $C_{\rm
c}^{\infty}(U)\subset X'(U)=[X(U)]'$.
By \cite[p.\,27, Theorem 1.2.19]{Gra14},
we find that $\varphi_t \ast f\to f$ in $L^{1}_{\rm
loc}(U)$ as $t\to 0^{+}$.
Applying this and
\eqref{e-p-integral}, we conclude that, for any $\phi\in C_{\rm
c}^{\infty} (U)$,
\begin{align*}
\int_{U}f(x)\partial^{\alpha}\phi(x)\,dx
&=\lim_{j\to \infty}
\int_{U}(\varphi_{t_j} \ast
f)(x)\partial^{\alpha}\phi(x)\,dx \\
&=(-1)^k
\lim_{j\to \infty}
\int_{U}\partial^{\alpha}(\varphi_{t_j} \ast
f)(x)\phi(x)\,dx
=(-1)^k
\int_{U}g_{\alpha}(x)\phi(x)\,dx.
\end{align*}
This, combined with the definition and the uniqueness of weak
derivatives (see, for instance, \cite[pp.\,143--144]{EG15}), further
implies that $\partial^{\alpha}f$ exists on $U$ and,
for almost every $x\in U$, $\partial^{\alpha}f(x)=g_{\alpha}(x)$.
From this
and \eqref{e-piX}, we deduce that
$\|\nabla^k f\|_{X(U)}
\sim\sum_{\alpha\in
\mathbb{Z}_+^n,\,|\alpha|=k}\|g_{\alpha}\|_{X(U)}
\lesssim{\rm I}.$
By this, the arbitrariness of $U$, and Fatou's lemma on $X(\Omega)$
again
(see \cite[Lemma 3.7]{ZYY24jga}),
we find that, for any $\alpha \in \mathbb{Z}_+^n$
with $|\alpha|=k$, $\partial^{\alpha}f$ exists on $\Omega$
and
$
\|\nabla^k f\|_{X(\Omega)}
\lesssim
{\rm I},
$
which further implies that $f\in {W}^{k,X}(\Omega)$
and hence completes the proof of Lemma \ref{lem-cha}.
\end{proof}

Applying Lemma \ref{lem-cha} and Theorem \ref{thm-main},
we now show Theorem \ref{thm-cwkx}.

\begin{proof}[Proof of Theorem \ref{thm-cwkx}]
We first prove the necessity.
To do this, let $f\in W^{k,X}(\Omega)$.
Applying Theorem \ref{thm-main}, the definition
of $W^{k,X}(\Omega)$, and Remark \ref{rem-main}(i),
we immediately obtain $$I(f)\lesssim \left\|\nabla^k f\right\|_{X(\Omega)}<\infty$$ and
$f\in W^{k-1,X}(\Omega)$.
This finishes the proof of
the necessity.

Next, we show the sufficiency. For this purpose, let
$f\in W^{k-1,X}(\Omega)$ satisfy \eqref{eq-cha01}
and $\{\rho_{\epsilon}\}_{\epsilon\in (0,\frac{n}{kq})}$
be the same as in \eqref{eq-sRHO}.
By this,
we conclude that
\begin{align*}
& \liminf_{\epsilon\to 0^+}
\left\|\left[\int_{\Omega(\cdot,k)}
\frac{|\Delta^k_h f(\cdot)|^q}{|h|^{kq}}\rho_{\epsilon}(|h|)\,dh
\right]^\frac{1}{q}\right\|_{X(\Omega)}\\
&\quad \lesssim\liminf_{s\to 1^{-}}
(1-s)^{\frac{1}{q}}
\left\|\left[\int_{\{h\in \Omega(\cdot,k):|h|<1\}}
\frac{|\Delta^k_h f(\cdot)|^q}{|h|^{n+skq}}\,dh
\right]^\frac{1}{q}\right\|_{X(\Omega)}\le I(f)<\infty.
\end{align*}
From this and
Lemma \ref{lem-cha}, it follows that
$f\in W^{k,X}(\Omega)$ and $\|f\|_{W^{k,X}(\Omega)}\lesssim
\|f\|_{W^{k-1,X}(\Omega)}+I(f) $.
Thus, we obtain the sufficiency. This
finishes the proof of Theorem \ref{thm-cwkx}.
\end{proof}

Finally, applying Theorems \ref{thm-main} and \ref{thm-cwkx}
to weighted
Lebesgue spaces, we give the proof of Theorem \ref{thm-wls}.

\begin{proof}[Proof of Theorem \ref{thm-wls}]
We first prove (i).
From \cite[Theorem 1.34]{Rud87},
we infer that $L^p_{\omega}$ has an absolutely continuous norm.
By this and an argument similar to that used in the proof of Theorem \ref{thm-Q},
we find that $L^p_{\omega}$ satisfies all the assumptions of
Theorem \ref{thm-main}, which further implies (i).

Now, we show (ii). Let $p\in (1,\infty)$.
From this, the fact that $[L^p_{\omega} ]'=L^{p'}_{\omega^{1-p'}}$,
and \cite[Theorem 1.34]{Rud87} again, we deduce that
$[L^p_{\omega} ]'$
has an absolutely continuous norm.
This, together with the above proved fact that $L^p_{\omega}$
satisfies all the assumptions of
Theorem \ref{thm-main}, further implies that
$L^p_{\omega}$ satisfies all the assumptions of
Theorem \ref{thm-cwkx}.
By this, we obtain (ii), which completes
the proof of Theorem \ref{thm-wls}.
\end{proof}

\section{The Sharpness of The Triebel--Lizorkin Exponent $q$}\label{sec-q}

In this section, we prove the sharpness of the Triebel--Lizorkin
exponent $q$
in the main theorems of this article.
First, the following proposition shows that
the Triebel--Lizorkin exponent $q$ in Theorem \ref{thm-upi} is nearly sharp.

\begin{proposition}\label{pro-q}
Let $k\in \mathbb{N}$, $s\in (0,1)$, and $p,q\in [1,\infty)$.
\begin{enumerate}[{\rm (i)}]
\item
If there exists a positive constant $C$ such that,
for any $r\in (0,\infty)$
and $f\in C_{\rm c}^\infty$,
\begin{align}\label{eq-sp1}
\left\{
\int_{\mathbb{R}^n}
\left[\int_{B({\bf 0},r)}
\frac{|\Delta^k_h f(x)|^q}{|h|^{n+skq}}\,dh
\right]^\frac{p}{q}\,dx\right\}^{\frac{1}{p}}
\le C r^{(1-s)k}
\left\|\nabla^k f\right\|_{L^p},
\end{align}
then $n(\frac{1}{p}-\frac{1}{q})\le k$.

\item Assume $p\in (1,\infty)$. Then there exists a weight
$\omega\in A_p$ such that, if inequality
\eqref{eq-sss}
holds for any $r\in (0,\infty)$
and $f\in C_{\rm c}^\infty$,
then $n(\frac{p_\omega}{p}-\frac{1}{q})\le k$.
\end{enumerate}
\end{proposition}

\begin{proof}
We first prove (i). To this end,
choose $f\in C_{\rm c}^\infty$ satisfying ${\bf 1}_{B({\bf 0},1)}
\le f\le {\bf 1}_{B({\bf 0},2)}$
and fix $r\in (12k,\infty)$. Then, for any $i\in\{0,\ldots,k-1\}$,
$y\in B({\bf 0},1)$, and
$x\in \mathbb{R}^n$
with $\frac{r}{4}\le |x|<\frac{r}{2}$,
$|x-y|<r$ and $|\frac{(k-i)x+iy}{k}|\ge \frac{|x|-(k-1)|y|}{k} >2$.
Similar to \eqref{eq-s2}, by these and a change of variables,
we find that
\begin{align}\label{eq-sharp02}
&\left\{
\int_{\mathbb{R}^n}
\left[\int_{B({\bf 0},r)}
\frac{|\Delta^k_h f(x)|^q}{|h|^{n+skq}}\,dh
\right]^\frac{p}{q}\,dx\right\}^{\frac{1}{p}}\notag\\
&\quad\gtrsim	\left\{
\int_{\frac{r}{4}\le |x|<\frac{r}{2}}
\left[\int_{B({\bf 0},1)}
\frac{|\Delta^k_{\frac{y-x}{k}} f(x)|^q}{|y-x|^{n+skq}}\,dy
\right]^\frac{p}{q}\,dx\right\}^{\frac{1}{p}}\notag\\
&\quad\gtrsim
r^{-\frac{n}{q}-sk}
\left\{
\int_{\frac{r}{4}\le |x|<\frac{r}{2}}\left[\int_{B({\bf0},1)}
|f(y)|^q\,dy\right]^{\frac pq}
\,dx\right\}^{\frac{1}{p}}\sim r^{n(\frac{1}{p}-\frac{1}{q})-sk}.
\end{align}
From this and \eqref{eq-sp1}, it follows that
$r^{n(\frac{1}{p}-\frac{1}{q})-sk}\lesssim r^{(1-s)k}$.
Combining this and the arbitrariness of $r\in(12k,\infty)$,
we further obtain $n(\frac{1}{p}-\frac{1}{q})\le k$.
This finishes the proof of (i).

Next, we show (ii). To do this,
consider the weight $\omega(x):=|x|^{\delta-n}$
for any $x\in\mathbb{R}^n$ and any given $\delta \in (n,np)$.
Then, using \cite[Example 7.1.7]{Gra14}, we find that
$\omega\in A_p$ and $p_\omega =\frac{\delta}{n}$
(see also \cite[p.\,24]{Cla25}).
Let $f$ be as in (i).
Similar to \eqref{eq-sharp02}, we obtain, for any
$r\in (12k,\infty)$,
\begin{align*}
\left\{
\int_{\mathbb{R}^n}
\left[\int_{B({\bf 0},r)}
\frac{|\Delta^k_h f(x)|^q}{|h|^{n+skq}}\,dh
\right]^\frac{p}{q}|x|^{\delta-n}\,dx\right\}^{\frac{1}{p}}
&\gtrsim	\left\{
\int_{\frac{r}{4}\le |x|<\frac{r}{2}}
\left[\int_{B({\bf 0},1)}
\frac{|\Delta^k_{\frac{y-x}{k}} f(x)|^q}{|y-x|^{n+skq}}\,dy
\right]^\frac{p}{q}|x|^{\delta-n}\,dx\right\}^{\frac{1}{p}}\\
&\sim
r^{-\frac{n}{q}-sk}
\left\{
\int_{\frac{r}{4}\le |x|<\frac{r}{2}}
|x|^{\delta-n}\,dx\right\}^{\frac{1}{p}}\sim r^{\frac{\delta}{p}-\frac{n}{q}-sk}.
\end{align*}
From this and \eqref{eq-sss},
we infer that $\frac{\delta}{p}-\frac{n}{q}\le k$;
that is, $n(\frac{p_\omega}{p}-\frac{1}{q})\le k$.
Thus, (ii) holds, which then completes the proof of
Proposition \ref{pro-q}.
\end{proof}

Similarly, the following result gives the near sharpness
of the range of $q$ in Theorems \ref{thm-Q} and \ref{thm-QX}.

\begin{proposition}\label{pro-q2}
Let $k\in \mathbb{N}$, $s\in (0,1)$, and $p,q\in [1,\infty)$.
\begin{enumerate}[{\rm (i)}]
\item If \eqref{eq-sss3} with $\omega\equiv1$ holds
for any cube $Q\subset \mathbb{R}^n$
and any $f\in C_{\rm c}^\infty(Q)$,
then $n(\frac{1}{p}-\frac{1}{q})\le k$.

\item Assume $p\in (1,\infty)$. Then there exists
a weight
$\omega\in A_p$ such that, if inequality
\eqref{eq-sss3}
holds for any cube $Q\subset \mathbb{R}^n$
and any $f\in C_{\rm c}^\infty(Q)$,
then $n(\frac{p_\omega}{p}-\frac{1}{q})\le k$.
\end{enumerate}
\end{proposition}

\begin{proof}
Let $Q$ be a cube centered at ${\bf 0}$ with $\ell (Q)\in (24k,\infty)$.
Choose $f\in C_{\rm c}^\infty(Q)$ satisfying
${\bf 1}_{B({\bf 0},1)}\le f\le {\bf 1}_{B({\bf 0},2)}$.
Repeating the proof of Proposition \ref{pro-q} with $r$ and \eqref{eq-sp1} therein replaced,
respectively, by $\frac{\ell (Q)}{2}$
and \eqref{eq-sss3} here,
we obtain the desired results.
\end{proof}

Finally, the sharpness of the range of $q$
in the
Gagliardo--Nirenberg interpolation
inequality
(Theorem \ref{thm-in}),
the
BBM formula (Theorem \ref{thm-main}),
and the related new characterization of Sobolev type spaces
(Theorem \ref{thm-cwkx}) follows from the following proposition.

\begin{proposition}\label{pro-sp3}
Let $k\in \mathbb{N}$, $s\in (0,1)$, and $p,q\in [1,\infty)$.
\begin{enumerate}[{\rm (i)}]
\item If $n(\frac{1}{p}-\frac{1}{q})\ge k$,
then there exists $f\in C_{\rm c}^\infty$
such that
\begin{align*}
\left\{
\int_{\mathbb{R}^n}
\left[\int_{\mathbb{R}^n}
\frac{|\Delta^k_h f(x)|^q}{|h|^{n+skq}}\,dh
\right]^\frac{p}{q}\,dx\right\}^{\frac{1}{p}}
=\infty.
\end{align*}

\item Assume $p\in (1,\infty)$.
Then there exist
$\omega\in A_p$
and $f\in C_{\rm c}^\infty$ such that, if
$n(\frac{p_{\omega}}{p}-\frac{1}{q})\ge k$, then
\begin{align*}
\left\{
\int_{\mathbb{R}^n}
\left[\int_{\mathbb{R}^n}
\frac{|\Delta^k_h f(x)|^q}{|h|^{n+skq}}\,dh
\right]^\frac{p}{q}\omega(x)\,dx\right\}^{\frac{1}{p}}
=\infty.
\end{align*}
\end{enumerate}
\end{proposition}

\begin{proof}
We first prove (i).
For this purpose,
choose $f\in C^{\infty}_{\rm c} $ satisfying
${\bf 1}_{B({\bf 0},\frac{1}{2})}\le f\le {\bf 1}_{B({\bf
0},\frac{3}{2})}$.
Notice that, for any $i\in\{0,\ldots,k-1\}$,
$x\in B({\bf 0},2k)^{\complement}$,
and $y\in B({\bf 0},\frac{1}{2})$,
$|\frac{(k-i)x+iy}{k}|\ge \frac{|x|-(k-1)|y|}{k}>\frac{3}{2}$,
which, together with an argument similar to that used in
the proof of \eqref{eq-s2},
further implies that $|\Delta_{\frac{y-x}{k}}f(x)|=|f(y)|=1$.
From this and the change of variable $y:=x+kh$, we deduce that
\begin{align*}
\int_{\mathbb{R}^n}\left[\int_{\mathbb{R}^n}
\frac{|\Delta^k_h f(x)|^q}{|h|^{n+skq}}\,dh
\right]^\frac{p}{q}\,dx
&\ge
\int_{B({\bf 0},2k)^{\complement}}\left[\int_{B({\bf 0},\frac{1}{2})}
|y-x|^{-n-skq}\,dy
\right]^\frac{p}{q}\,dx\\
&\gtrsim
\int_{B({\bf 0},2k)^{\complement}}
|x|^{-\frac{pn}{q}-skp}
\,dx\sim
\int_{2k}^{\infty}
r^{p[\frac{n}{p}-\frac{n}{q}-k]+(1-s)kp-1}
\,dr\\
&\ge \int_{2k}^{\infty}
r^{(1-s)kp-1}
\,dr=\infty.
\end{align*}
Thus, (i) holds.

Now, we show (ii). To this end,
fix $\delta \in (n,np)$ and, for any $x\in\mathbb{R}^n$,
let $\omega(x):=|x|^{\delta-n}$.
Then, from \cite[Example 7.1.7]{Gra14}, it follows that
$\omega\in A_p$ and $p_\omega =\frac{\delta}{n}$
(see also \cite[p.\,24]{Cla25}).
Applying this and an argument similar to that used in the proof of (i),
we conclude that, if $\frac{\delta}{p}-\frac nq=n(\frac{p_\omega}{p}-\frac1q)\ge k$,
then, for any $f\in C^{\infty}_{\rm c} $ satisfying
${\bf 1}_{B({\bf 0},\frac{1}{2})}\le f\le {\bf 1}_{B({\bf
0},\frac{3}{2})}$,
\begin{align*}
\int_{\mathbb{R}^n}\left[\int_{\mathbb{R}^n}
\frac{|\Delta^k_h f(x)|^q}{|h|^{n+skq}}\,dh
\right]^\frac{p}{q}|x|^{\delta-n}\,dx
&\gtrsim
\int_{B({\bf 0},2k)^{\complement}}
|x|^{-\frac{pn}{q}-skp+\delta-n}\,dx
\sim\int_{2k}^{\infty}
r^{p[\frac{\delta}{p}-\frac{n}{q}-k]+(1-s)kp-1}\,dr\\
&\ge \int_{2k}^{\infty}
r^{(1-s)kp-1}
\,dr=\infty.
\end{align*}
This finishes the proof (ii)
and hence Proposition \ref{pro-sp3}.
\end{proof}

\section{Applications to Specific Function Spaces}\label{sec-app}
In this section, we apply Theorems \ref{thm-QX},
\ref{thm-in}, \ref{thm-main},
and \ref{thm-cwkx}
to various examples of ball Banach function spaces,
including
Bourgain--Morrey-type spaces (see Subsection \ref{ss-BM}),
mixed-norm Lebesgue spaces (see Subsection \ref{ss-mix}),
variable Lebesgue spaces (see Subsection \ref{ss-va}),
Lorentz spaces (see Subsection \ref{ss-lor}),
Orlicz(-slice) spaces (see Subsection \ref{ss-os}), and
Herz spaces (see Subsection \ref{ss-Herz}).
In Subsection \ref{ss-m}, we present two more
specific function spaces to which the aforementioned
main theorems of this article can be applied,
without giving the details.
More applications to newfound function spaces are highly anticipated.

\subsection{Bourgain--Morrey-Type Spaces}\label{ss-BM}

Recall that Morrey spaces,
introduced by Morrey \cite{Mor38},
have important
applications in the theory of partial differential equations
and harmonic analysis (see, for instance, the monographs
\cite{a15} of Adams,
\cite{SDGD20I,SDGD20II} of Sawano et al.,
and \cite{ysy10} of Yuan et al.).
In 1991, Bourgain \cite{Bou91} introduced a new function space
(a special case of Bourgain--Morrey spaces)
to study the Bochner--Riesz multiplier in $\mathbb{R}^3$.
Later, Masaki \cite{Mas-arXiv} extended it to the full range of
exponents (that is, what is now known as the Bourgain--Morrey space)
for nonlinear Schr\"{o}dinger equations,
generalizing Morrey spaces.
These spaces have since become crucial in the research
of partial differential equations
(see, for example, \cite{BV07,Bou98,KPV00,ms18,ms18-2,MVV99}),
while their fundamental real-variable properties were
recently established by Hatano et al. \cite{HNSH22}.
Very recently,
via mixing Besov--Triebel--Lizorkin and Bourgain--Morrey structures,
Zhao et al. \cite{ZSTYY23} and Hu et al. \cite{HLY23} respectively
introduced the Besov--Bourgain--Morrey space and
the Triebel--Lizorkin--Bourgain--Morrey space.

We next recall the definitions of
the aforementioned Bourgain--Morrey-type spaces.
Let $p\in(0,\infty]$.
We use
$L^p_{\rm loc}(\Omega)$ to denote the set of all locally
$p$-integrable
functions on an open set $\Omega$.
For any $j\in{\mathbb Z}$ and
$m:=(m_1,\ldots,m_n) \in {\mathbb Z}^n$,
recall that the \emph{dyadic cube} $Q_{j, m}$ in
$\mathbb{R}^n$ is defined by
setting
$
Q_{j, m}:=2^{-j}\left(m+[0,1)^n\right).
$
The following definitions of Bourgain--Morrey-type spaces
can be found in, for example,
\cite{HNSH22,HLY23,ZSTYY23}.

\begin{definition}
Let $0<p\le u\le r\le\infty$ and $\tau\in(0,\infty]$.
\begin{enumerate}[{\rm(i)}]
\item The \emph{Morrey space}
$M^u_p$
is defined to be the
set of all $f\in L^p_{\rm loc}$ such that
\begin{equation*}
\|f\|_{M^u_p}:=
\sup_{j\in\mathbb{Z},\,m\in\mathbb{Z}^n}
\left|Q_{j,m}\right|^{\frac 1u-\frac 1p}
\left\|f{\mathbf{1}}_{Q_{j,m}}
\right\|_{L^{p}}<\infty.
\end{equation*}
\item The \emph{Bourgain--Morrey space}
$M^u_{p,r}$ is defined to be the
set of all $f\in L^p_{\rm loc}$ such that
\begin{equation*}
\|f\|_{M^u_{p,r}}
:=\left\{\sum_{j\in{\mathbb Z},\,m\in{\mathbb Z}^n}
\left[\left|Q_{j,m}\right|^{\frac{1}{u}-\frac{1}{p}}
\left\|f{\mathbf{1}}_{Q_{j,m}}
\right\|_{L^{p}}\right]^r\right\}^{\frac{1}{r}},
\end{equation*}
with the usual modification made when $r=\infty$, is finite.
\item[{\rm(iii)}] The \emph{Besov--Bourgain--Morrey space
$M\dot{B}_{p,r}^{u,\tau}$} is defined to be the set
of all $f\in L^p_{\rm loc}$ such that
\begin{align*}
\|f\|_{M\dot{B}_{p,r}^{u,\tau}}:=
\left\{\sum_{j\in{\mathbb Z}}
\left[\sum_{m\in{\mathbb Z}^n}
\left\{\left|Q_{j, m}\right|^{\frac{1}{u}-\frac{1}{p}}
\left\|f{\mathbf{1}}_{Q_{j,m}}\right\|_{L^{p}}
\right\}^r \right]^\frac{\tau}{r}\right\}^\frac{1}{\tau},
\end{align*}
with the usual modifications made when $r=\infty$ or $\tau=\infty$,
is finite.
\item[{\rm(iv)}] The \emph{Triebel--Lizorkin--Bourgain--Morrey space
$M\dot{F}_{p,r}^{u,\tau}$} is defined to be the set
of all $f\in L^p_{\rm loc}$ such that
\begin{align*}
\|f\|_{M\dot{F}_{p,r}^{u,\tau}}
:=\left(\int_{\mathbb{R}^n}
\left\{\int_{0}^{\infty}\left[t^{n(\frac1u
-\frac1p-\frac1r)}\left\|f{\bf 1}_{B(y,t)}
\right\|_{L^p}\right]^\tau\,
\frac{dt}{t}\right\}^{\frac r\tau}\,dy\right)^{\frac1r},
\end{align*}
with the usual modifications made when $r=\infty$ or $\tau=\infty$,
is finite.
\end{enumerate}
\end{definition}

\begin{remark}
It is obvious that
$M^{u}_{p,\infty}=M^u_p$
and $M\dot{B}^{u,r}_{p,r}
=M^u_{p,r}$.
Moreover, from \cite[Proposition 3.6(iii)]{HLY23},
we infer that $M\dot{F}^{u,r}_{p,r}
=M^u_{p,r}$.
\end{remark}

\begin{theorem}\label{thm-bm}
Let $k\in \mathbb{N}$, $1\le  p<u< r\le \infty$,
$\tau\in (0,\infty]$, $A\in \{B,F\}$, and $q\in[1,\infty)$
satisfy $n(\frac{1}{p}-\frac{1}{q})<k$.
Then Theorem \ref{thm-QX} holds with $X:=M\dot{B}_{p,r}^{u,\tau}$
and Theorem \ref{thm-in} holds with $X:=M\dot{A}_{p,r}^{u,\tau}$.
Moreover, if $p\in (1,\infty)$, then
Theorem \ref{thm-cwkx} holds with $X:=M\dot{A}_{p,r}^{u,\tau}$
and $\Omega:=\mathbb{R}^n$;
if $p,r\in (1,\infty)$, then
Theorem \ref{thm-main} holds with $X:=M\dot{B}_{p,r}^{u,\tau}$.
\end{theorem}

\begin{proof}
By \cite[Proposition 7.7]{LYYZZ-arXiv}, an argument similar
to that used
in the proof of \cite[Theorem 7.5]{LYYZZ-arXiv}, and
Theorems \ref{thm-Q}, \ref{thm-ex}, and
\ref{thm-wls}(ii) with $\omega\in A_1$,
we can easily conclude that
Theorems \ref{thm-QX} and
\ref{thm-in} hold with $X:=M\dot{A}_{p,r}^{u,\tau}$
and Theorem \ref{thm-cwkx}
holds with $X:=M\dot{A}_{p,r}^{u,\tau}$ and $p\in (1,\infty)$.
Furthermore,
from the proof of \cite[Theorem 4.12]{ZYY23ccm},
we deduce that $M\dot{B}_{p,r}^{u,\tau}$
under consideration satisfies
all the assumptions of Theorem \ref{thm-QX}
and $M\dot{B}_{p,r}^{u,\tau}$ with $p,r\in (1,\infty)$
satisfies all the assumptions of Theorem \ref{thm-main}. Applying
these, we obtain the desired conclusions.
\end{proof}

\begin{remark}
Let all the symbols be the same as in Theorem \ref{thm-bm}.
Then Theorem \ref{thm-QX} with $X:=M\dot{B}_{p,r}^{u,\tau}$
and Theorem \ref{thm-in} with $X:=M\dot{A}_{p,r}^{u,\tau}$ are new.
When $p\in (1,\infty)$,
Theorem \ref{thm-cwkx} with $X:=M\dot{A}_{p,\infty}^{u,\infty}$, $k:=1$, $q\in (1,p)$, and
$\Omega:=\mathbb{R}^n$
coincides with \cite[Proposition 5.3]{DGPYYZ24}, and with $X:=M\dot{A}_{p,r}^{u,\tau}$, $k:=1$, and
$q\in [1,\frac{n}{n-1}]\cup [1,\min\{p,\tau\})$
reduces to \cite[Theorem 6.16]{ZYY24jga};
other cases of Theorem \ref{thm-cwkx} with $X:=M\dot{A}_{p,r}^{u,\tau}$ are new.
When $p,r\in (1,\infty)$,
Theorem \ref{thm-main} with $X:=M\dot{B}_{p,r}^{u,\tau}$, $k:=1$,
and $q\in [1,\frac{n}{n-1}]\cup [1,\min\{p,\tau\})$
reduces to \cite[Theorem 6.16]{ZYY24jga}
and other cases are new.
\end{remark}

\subsection{Mixed-Norm Lebesgue Spaces}\label{ss-mix}

For a given vector $\vec{r}:=(r_1,\ldots,r_n)
\in[1,\infty]^n$, the \emph{mixed-norm Lebesgue
space $L^{\vec{r}} $} is defined to be the
set of all $f\in\mathscr{M} $ with
the following finite \emph{norm}
\begin{equation*}
\|f\|_{L^{\vec{r}} }:=\left\{\int_{\mathbb{R}}
\cdots\left[\int_{\mathbb{R}}\left|f(x_1,\ldots,
x_n)\right|^{r_1}\,dx_1\right]^{\frac{r_2}{r_1}}
\cdots\,dx_n\right\}^{\frac{1}{r_n}},
\end{equation*}
where the usual modifications are made when $r_i=
\infty$ for some $i\in\{1,\ldots,n\}$.
The study of mixed-norm Lebesgue spaces
can be traced back to H\"ormander \cite{Hor60}
and Benedek and Panzone \cite{BP61}.
For more studies on mixed-norm Lebesgue spaces,
we refer to Cleanthous et al. \cite{CG20,CGN17,CGN19}
and the survey article \cite{HY21}.

\begin{theorem}\label{thm-mix}
Let $k\in{\mathbb{N}}$, $\vec{r}:=(r_1,\ldots,r_n)\in(1,\infty)^n$,
$r_-:=\min\{r_1, \ldots, r_n\}$,
and
$q\in [1,\infty)$ satisfy
$n(\frac{1}{r_{-}}-\frac{1}{q})<k$.
Then Theorems \ref{thm-QX},
\ref{thm-in}, \ref{thm-main},
and \ref{thm-cwkx} hold with $X:=L^{\vec{r}} $.
\end{theorem}

\begin{proof}
From the proof of \cite[Theorem 4.19]{ZYY23ccm},
it follows that $L^{\vec{r}}$
under consideration satisfies
all the assumptions of Theorem \ref{thm-cwkx}. Applying
this, we obtain the desired conclusions.
\end{proof}

\begin{remark}
Let all the symbols be the same as in Theorem \ref{thm-mix}.
Then Theorems \ref{thm-QX} and \ref{thm-in} with $X:=L^{\vec{r}} $
are new. Theorem \ref{thm-main} with $X:=L^{\vec{r}}$, $k:=1$, $q\in [1,r_{-})$, and
$\Omega:=\mathbb{R}^n$
coincides with \cite[Theorem 5.6]{DGPYYZ24}
and other cases are new.
Theorem \ref{thm-cwkx}  with $X:=L^{\vec{r}}$, $k:=1$, and $q\in [1,\frac{n}{n-1}]\cup [1,r_{-})$
reduces to \cite[Theorem 6.23]{ZYY24jga} and other cases
are new.
\end{remark}

\subsection{Variable Lebesgue Spaces}\label{ss-va}
Recall that the \emph{variable Lebesgue space
$L^{r(\cdot)} $} associated with the function
$r:\mathbb{R}^n\to(0,\infty)$ is defined to be the set
of all $f\in\mathscr{M} $ with the following finite
\emph{quasi-norm}
\begin{equation*}
\|f\|_{L^{r(\cdot)} }:=\inf\left\{\lambda
\in(0,\infty):\int_{\mathbb{R}^n}\left[\frac{|f(x)|}
{\lambda}\right]^{r(x)}\,dx\le1\right\}.
\end{equation*}
For more related results on variable Lebesgue spaces,
we refer to
\cite{CF13,DHR09,KR91,NS12}.

Let $r:\mathbb{R}^n\to(0,\infty)$ be a
measurable function,
$$
\widetilde{r}_-:=\underset{x\in\mathbb{R}^n}{
\mathop\mathrm{\,ess\,inf\,}}\,r(x),\
\text{and}\
\widetilde{r}_+:=\underset{x\in\mathbb{R}^n}{
\mathop\mathrm{\,ess\,sup\,}}\,r(x).$$
A function $r:\mathbb{R}^n\to(0,\infty)$ is said to be
\emph{globally
{\rm log}-H\"older continuous} if there exist
$r_{\infty}\in\mathbb{R}$ and a positive constant $C$ such that, for
any $x,y\in\mathbb{R}^n$,
\begin{equation*}
|r(x)-r(y)|\le \frac{C}{\log(e+\frac{1}{|x-y|})}
\ \ \text{and}\ \
|r(x)-r_\infty|\le \frac{C}{\log(e+|x|)}.
\end{equation*}

\begin{theorem}\label{thm-v}
Let $k\in{\mathbb{N}}$, $r:\mathbb{R}^n\to(0,\infty)$ be globally
log-H\"older continuous, $1< \widetilde{r}_-
\leq\widetilde{r}_+<\infty$,
and
$q\in [1,\infty)$ satisfy $n(\frac{1}{\widetilde{r}_-}-\frac{1}{q})<k$.
Then Theorems \ref{thm-QX},
\ref{thm-in}, \ref{thm-main},
and \ref{thm-cwkx} hold with $X:=L^{r(\cdot)}$.
\end{theorem}

\begin{proof}
Applying the proof of \cite[Theorem 5.10]{ZYY23}, we conclude that all
the assumptions of Theorem \ref{thm-cwkx}
hold for the variable Lebesgue space
$L^{{r}(\cdot)} $ under
consideration. Thus,
we obtain the desired results.
\end{proof}

\begin{remark}
Let all the symbols be the same as in Theorem \ref{thm-v}.
Then Theorems \ref{thm-QX} and
\ref{thm-in} with $X:=L^{r(\cdot)}$ are new.
Theorem
\ref{thm-main} with $X:=L^{r(\cdot)}$, $k:=1$, $q\in [1,\widetilde{r}_-)$,
and $\Omega:=\mathbb{R}^n$
coincides with \cite[Theorem 5.13]{DGPYYZ24}
and other cases are new.
Theorem \ref{thm-cwkx} with $X:=L^{r(\cdot)}$, $k:=1$, and $q\in [1,\frac{n}{n-1}]\cup[1,\widetilde{r}_-)$
reduces to \cite[Theorem 6.23]{ZYY24jga} and
other cases are new.
\end{remark}

\subsection{Lorentz Spaces}\label{ss-lor}

Recall that, for any $r,\tau\in(0,\infty)$,
the \emph{Lorentz space $L^{r,\tau} $}, originally studied by Lorentz \cite{Lor50,Lor51},
is defined to be the set of all
$f\in\mathscr{M} $ such that
\begin{equation*}
\|f\|_{L^{r,\tau} }
:=\left\{\int_0^{\infty}
\left[t^{\frac{1}{r}}f^*(t)\right]^\tau
\frac{\,dt}{t}\right\}^{\frac{1}{\tau}}
<\infty,
\end{equation*}
where $f^*$ denotes the \emph{decreasing rearrangement of $f$},
that is, for any $t\in[0,\infty)$,
\begin{equation*}
f^*(t):=\inf\left\{s\in(0,\infty):\left|
\left\{x\in\mathbb{R}^n:|f(x)|>s\right\}\right|\leq t\right\}
\end{equation*}
with the convention $\inf \emptyset =\infty$.
\begin{theorem}\label{thm-lorentzs}
Let $k\in {\mathbb{N}}$,
$r,\tau\in(1,\infty)$,
and $q\in [1,\infty)$ satisfy
$n(\frac{1}{\min\{r,\tau\}}-\frac{1}{q})<k$.
Then Theorems \ref{thm-QX},
\ref{thm-in}, \ref{thm-main},
and \ref{thm-cwkx} hold with $X:=L^{r,\tau}$.
\end{theorem}

\begin{proof}
Using the proof of \cite[Theorem 5.18]{ZYY23}, we easily find that all
the assumptions with $X :=L^{r,\tau} $ of Theorem \ref{thm-cwkx}
and hence Theorems \ref{thm-QX},
\ref{thm-in}, and \ref{thm-main}
are satisfied, which further implies the desired conclusions and
hence completes the proof of Theorem \ref{thm-lorentzs}.
\end{proof}

\begin{remark}
Let all the symbols be the same as in Theorem \ref{thm-lorentzs}.
Then Theorems \ref{thm-QX} and
\ref{thm-in} with $X:=L^{r,\tau}$ are new.
Theorem \ref{thm-main} with $X:=L^{r,\tau}$, $k:=1$, $q\in [1,\min\{r,\tau\})$,
and $\Omega:=\mathbb{R}^n$ coincides with \cite[Theorem 5.24]{DGPYYZ24}
and other cases are new.
Theorem \ref{thm-cwkx} with $X:=L^{r,\tau}$, $k:=1$, and
$q\in [1,\frac{n}{n-1}]\cup[1,\min\{r,\tau\})$ reduces to
\cite[Theorem 6.27]{ZYY24jga} and other cases
are new.
\end{remark}

\subsection{Orlicz and Orlicz-Slice Spaces}\label{ss-os}
Recall that a non-decreasing function $\Phi:[0,\infty)
\ \to\ [0,\infty)$ is called an \emph{Orlicz function}
if $\Phi$ satisfies that
$\Phi(0)= 0$,
$\Phi(t)\in(0,\infty)$ for any $t\in(0,\infty)$,
and
$\lim_{t\to\infty}\Phi(t)=\infty$.
An Orlicz function $\Phi$ is said to be of \emph{lower}
(resp. \emph{upper}) \emph{type} $r$ for some
$r\in\mathbb{R}$ if there exists a positive constant
$C_{(r)}$ such that,
for any $t\in[0,\infty)$ and
$s\in(0,1)$ [resp. $s\in[1,\infty)$],
$
\Phi(st)\le C_{(r)} s^r\Phi(t).
$
In this subsection, we always assume that
$\Phi$
is an Orlicz function with both positive lower
type $r_{\Phi}^-$ and positive upper type $r_{\Phi}^+$.
The \emph{Orlicz space $L^\Phi $}
is defined to be the set of all $f\in\mathscr{M} $
with the following finite \emph{quasi-norm}
\begin{equation*}
\|f\|_{L^\Phi }:=\inf\left\{\lambda\in
(0,\infty):\int_{\mathbb{R}^n}\Phi\left(\frac{|f(x)|}
{\lambda}\right)\,dx\le1\right\}.
\end{equation*}
For more related results on Orlicz spaces,
we refer to \cite{DFMN21,NS14,RR02}.
For any given $t,r\in(0,\infty)$,
the \emph{Orlicz-slice space}
$(E_\Phi^r)_t $ is defined to be
the set of all $f\in\mathscr{M} $ with the
following finite \emph{quasi-norm}
\begin{equation*}
\|f\|_{(E_\Phi^r)_t }:=\left\{\int_{\mathbb{R}^n}
\left[\frac{\|f\mathbf{1}_{B(x,t)}\|_{L^\Phi }}
{\|\mathbf{1}_{B(x,t)}\|_{L^\Phi }}\right]
^r\,dx\right\}^{\frac{1}{r}},
\end{equation*}
where $\Phi$
is an Orlicz function with both positive
lower type $r_{\Phi}^-$ and positive upper
type $r_{\Phi}^+$.
The Orlicz-slice space was introduced in
\cite{ZYYW19} as a generalization of both
the slice space
(see, for example, \cite{AM19,AP17}) and the Wiener amalgam space
(see, for example, \cite{Ho19,Hol75,KNTYY07}). For more recent
studies on Orlicz-slice spaces, we refer to Ho \cite{Ho21,Ho22,Ho23}.

\begin{theorem}\label{thm-orlicz}
Let $k\in{\mathbb{N}}$ and
$\Phi$ be an Orlicz function with both positive lower
type $r_{\Phi}^-$ and positive upper type $r_{\Phi}^+$
satisfying $1< r^-_{\Phi}\leq
r^+_{\Phi}<\infty$.
\begin{enumerate}[{\rm (i)}]
\item If $q\in [1,\infty)$
satisfies $n(\frac{1}{r^-_{\Phi}}-\frac{1}{q})<k$,
then Theorems \ref{thm-QX},
\ref{thm-in}, \ref{thm-main},
and \ref{thm-cwkx} hold with $X:=L^{\Phi}$.

\item If $t\in(0,\infty)$,
$r\in(1,\infty)$, and $q\in [1,\infty)$
satisfy $n(\frac{1}{\min\{r^-_{\Phi},r\}}-\frac{1}{q})<k$,
then Theorems \ref{thm-QX},
\ref{thm-in}, \ref{thm-main},
and \ref{thm-cwkx} hold with $X:=(E_\Phi^r)_t$.
\end{enumerate}
\end{theorem}

\begin{proof}
From the proof of
\cite[Theorems 4.25 and 4.27]{ZYY23ccm}, we infer that the
Orlicz space $L^{\Phi} $ and
the Orlicz-slice space
$(E_\Phi^r)_t $ under consideration
satisfy all the assumptions of
Theorem \ref{thm-cwkx} and hence Theorems \ref{thm-QX},
\ref{thm-in}, and \ref{thm-main}.
By this, we obtain the
desired conclusions.
\end{proof}

\begin{remark}
Let all the symbols be the same as in Theorem \ref{thm-orlicz}.
\begin{enumerate}[{\rm (i)}]
\item Theorems \ref{thm-QX} and
\ref{thm-in} with $X:=L^{\Phi}$ are new.
Theorems \ref{thm-main} and \ref{thm-cwkx} with $X:=L^{\Phi}$, $k:=1$,
and $q\in [1,\frac{n}{n-1}]\cup[1, r^-_{\Phi})$
reduces to \cite[Theorem 6.29]{ZYY24jga} and other cases are new.

\item Theorems \ref{thm-QX} and
\ref{thm-in} with $X:=(E_\Phi^r)_t$ are new.
Theorems \ref{thm-main} and \ref{thm-cwkx} with $X:=(E_\Phi^r)_t$, $k:=1$,
and $q\in [1,\frac{n}{n-1}]\cup[1, \min\{r^-_{\Phi},r\})$
reduces to \cite[Theorem 6.31]{ZYY24jga} and other cases are new.
\end{enumerate}
\end{remark}

\subsection{Local and Global Generalized Herz Spaces}\label{ss-Herz}

Recall that the classical Herz space was originally
introduced by Herz \cite{herz} to
study the Bernstein theorem on absolutely
convergent Fourier transforms.
Recently, Rafeiro and Samko \cite{RS20}
introduced the local and the global generalized Herz spaces
(see Definition \ref{def-herz})
which respectively generalize the classical
Herz spaces and generalized Morrey type spaces.
For more studies on Herz spaces,
we refer to \cite{GLY98,HY99,HS25,HWYY23,LY96,
LYH22,RS20,ZYZ22}.

To state the definitions of the aforementioned
generalized Herz spaces,
we first present the Matuszewska--Orlicz indices
as follows, which were introduced in
\cite{MO,Mo65} and characterize the growth
properties of functions at
origin and infinity (see also \cite{LYH22}).

\begin{definition}
Let $\omega$ be a positive function on $(0,\infty)$. Then
the \emph{Matuszewska--Orlicz indices}
$m_0(\omega)$, $M_0(\omega)$,
$m_\infty(\omega)$, and $M_\infty(\omega)$ of
$\omega$ are defined, respectively, by setting, for
any $h\in(0,\infty)$,
$$m_0(\omega):=\sup_{t\in(0,1)}
\frac{\log\left[\limsup\limits_{h\to0^+}\frac{\omega(ht)}
{\omega(h)}\right]}{\log t},\ M_0(\omega):=\inf_{t\in(0,1)}\frac{\log
\left[
\liminf\limits_{h\to0^+}
\frac{\omega(ht)}{\omega(h)}\right]}{\log t},$$
$$
m_{\infty}(\omega):=\sup_{t\in(1,\infty)}
\frac{\log\left[\liminf\limits_{h\to\infty}
\frac{\omega(ht)}{\omega(h)}\right]}{\log t},\ \mathrm{and}\ M_\infty(\omega)
:=\inf_{t\in(1,\infty)}\frac{\log\left[\limsup\limits_
{h\to\infty}\frac{\omega(ht)}{\omega(h)}\right]}{\log t}.
$$
\end{definition}

Let $\mathbb{R}_+:=(0,\infty)$
and $\omega$ be a nonnegative function on $\mathbb{R}_+$.
Then the function $\omega$ is said to be
\emph{almost increasing}
(resp.\ \emph{almost decreasing})
on $\mathbb{R}_+$ if there exists
a constant $C\in[1,\infty)$ such that,
for any $t,\tau\in\mathbb{R}_+$ satisfying
$t\leq\tau$ (resp.\ $t\geq\tau$),
$\omega(t)\leq C\omega(\tau).$
The \emph{function class} $M(\mathbb{R}_+)$
is defined to be the set of all positive functions
$\omega$ on $\mathbb{R}_+$ such that, for any $0<\delta<N<\infty$,
$0<\inf_{t\in(\delta,N)}\omega(t)\leq\sup_
{t\in(\delta,N)}\omega(t)<\infty$ and
there exist four constants $\alpha_{0},
\beta_{0},\alpha_{\infty},\beta_{\infty}\in\mathbb{R}$ such that
\begin{enumerate}[{\rm(i)}]
\item for any $t\in(0,1]$,
$\omega(t)t^{-\alpha_{0}}$ is almost increasing and
$\omega(t)t^{-\beta_{0}}$ is almost decreasing;
\item for any $t\in[1,\infty)$,
$\omega(t)t^{-\alpha_{\infty}}$ is
almost increasing and $\omega(t)t^{-\beta_{\infty}}$ is
almost decreasing.
\end{enumerate}

The following concept of generalized Herz spaces
was originally introduced by
Rafeiro and Samko in \cite[Definition 2.2]{RS20}
(see also \cite{LYH22}).

\begin{definition}\label{def-herz}
Let $p,r\in(0,\infty]$ and $\omega\in M(\mathbb{R}_+)$.
\begin{enumerate}[{\rm(i)}]
\item Let $\xi\in\mathbb{R}^n$.
The \emph{local generalized Herz space}
$\dot{\mathcal{K}}^{p,r}_{\omega,\xi} $
is defined to be the set of all
$f\in L^p_{\rm loc}(\mathbb{R}^n\setminus\{\xi\})$ such that
$$
\|f\|_{\dot{\mathcal{K}}^{p,r}_{\omega,\xi}
}:=\left\{\sum_{k\in\mathbb{Z}}
\left[\omega\left(2^{k}\right)\right]^{r}
\left\|f{\bf 1}_{B({\bf 0},2^{k})\setminus B({\bf 0},2^{k-1})}
\right\|_{L^{p} }^{r}\right\}^{\frac{1}{r}}<\infty.
$$
\item The
\emph{global generalized
Herz space} $\dot{\mathcal{K}}^{p,r}_{\omega} $
is defined to be the set of all
$f\in L^p_{\rm loc} $ such that
$\|f\|_{\dot{\mathcal{K}}^{p,r}_{\omega} }
:=\sup_{\xi\in\mathbb{R}^n}
\|f\|_{\dot{\mathcal{K}}^{p,r}_{\omega,\xi} }<\infty$.
\end{enumerate}
\end{definition}

\begin{theorem}\label{thm-herz}
Let $k\in \mathbb{N}$, $p,r\in[1,\infty)$,
$\mu \in [1,\min\{p,r\}]$,
$q\in [1,\infty)$
satisfy
$n(\frac{1}{\mu}-\frac1q)<k$,
and $\omega\in M(\mathbb{R}_+)$
satisfy
\begin{align*}
-\frac np<m_0(\omega)\le M_0(\omega)< n\left(\frac{1}{\mu}-\frac{1}{p}\right)
\text{ and }
-\frac np<m_\infty(\omega)\le M_\infty(\omega)<n\left(\frac{1}{\mu}-\frac{1}{p}\right).
\end{align*}
\begin{enumerate}[{\rm(i)}]
\item
Let $\xi\in \mathbb{R}^n$.
Then Theorems \ref{thm-QX} and
\ref{thm-in}
hold with $X:=\dot{\mathcal{K}}^{p,r}_{\omega,\xi}$.
If $\mu\in (1,\infty)$ or $\mu=q=1$, then Theorem \ref{thm-main} holds with $X:=\dot{\mathcal{K}}^{p,r}_{\omega,\xi}$.
If $\mu\in (1,\infty)$, then Theorem \ref{thm-cwkx}
holds with $X:=\dot{\mathcal{K}}^{p,r}_{\omega,\xi}$.
\item Theorems \ref{thm-QX} and
\ref{thm-in}
hold with $X:=\dot{\mathcal{K}}^{p,r}_{\omega}$.
If $\mu\in (1,\infty)$, then Theorem \ref{thm-cwkx}
holds with $X:=\dot{\mathcal{K}}^{p,r}_{\omega}$.
\end{enumerate}
\end{theorem}

\begin{proof}
From an argument similar to that used in the proof of \cite[Theorem 4.15]{ZYY23ccm},
we deduce (i). On the other hand,
using (i) and the definition of $\dot{\mathcal{K}}^{p,r}_{\omega}$,
we can easily obtain (ii).
This finishes the proof of Theorem \ref{thm-herz}.
\end{proof}

\begin{remark}
Let all the symbols be the same as in Theorem \ref{thm-herz}.
\begin{enumerate}[{\rm (i)}]
\item Theorems \ref{thm-QX} and
\ref{thm-in} with $X:=\dot{\mathcal{K}}^{p,r}_{\omega,\xi}$ are new.
Theorem \ref{thm-main} with $X:=\dot{\mathcal{K}}^{p,r}_{\omega,\xi}$,
$k:=1$, $p,r\in (1,\infty)$, and
$q\in [1,\frac{n}{n-1}]\cup [1,\mu]$ reduces to \cite[Theorem 6.19]{ZYY24jga}
and other cases
are new.
When $p,r\in (1,\infty)$,
Theorem \ref{thm-cwkx} with $X:=\dot{\mathcal{K}}^{p,r}_{\omega,\xi}$, $k:=1$, and
$q\in [1,\frac{n}{n-1}]\cup [1,\mu]$ reduces to \cite[Theorem 6.19]{ZYY24jga}
and other cases
are new.

\item Theorems \ref{thm-QX} and
\ref{thm-in} with $X:=\dot{\mathcal{K}}^{p,r}_{\omega}$ are new.
When $p,r\in (1,\infty)$, Theorem \ref{thm-cwkx}
with $X:=\dot{\mathcal{K}}^{p,r}_{\omega}$, $k:=1$, and $q\in [1,\frac{n}{n-1}]\cup [1,\mu)$ reduces
to \cite[Theorem 6.21]{ZYY24jga} and other cases are new.
\end{enumerate}
\end{remark}

\subsection{More Function Spaces}\label{ss-m}

In this subsection, we indicate that the main results of this article can be applied to more
specific function spaces.
To limit the length of this article, we omit the details, but provide
some related references.

{\bf Weighted Orlicz spaces.} Let $A_{\infty}:=\bigcup_{p\in [1,\infty)} A_p$ and $\omega$ be a $A_{\infty}$-weight.
Recall that the \emph{weighted Orlicz space $L^\Phi_{\omega} $}
is defined to be the set of all $f\in\mathscr{M} $
having the following finite \emph{quasi-norm}
\begin{equation*}
\|f\|_{L^\Phi_{\omega} }:=\inf\left\{\lambda\in
(0,\infty):\int_{\mathbb{R}^n}\Phi\left(\frac{|f(x)|}
{\lambda}\right)\omega(x)\,dx\le1\right\}.
\end{equation*}
From \cite[Remark 5.22 (ii)]{WYYZ21}, we infer that the
weighted Orlicz space with positive lower
type $r_{\Phi}^-$ and positive upper type $r_{\Phi}^+$
satisfying $1< r^-_{\Phi}\le
r^+_{\Phi}<\infty$ is a
{\rm BBF} space.
By \cite[p.\,20, Theorem 13]{RR02}, \cite[Theorem 2.1.1]{KK91}, and the definition of absolutely continuous norms,
we find that the weighted Orlicz space $L^\Phi_{\omega} $ satisfies all the assumptions of Theorems \ref{thm-QX},
\ref{thm-in}, \ref{thm-main}, and \ref{thm-cwkx}.
Thus, Theorems \ref{thm-QX},
\ref{thm-in}, \ref{thm-main}, and \ref{thm-cwkx}
can to applied to $X:=L^\Phi_{\omega}$,
and the obtained results are new; we omit the details.

{\bf Mixed-norm Herz spaces.}
Let $\vec{p}:=(p_1,\ldots,p_n),\vec{q}:=(q_1,\ldots,q_n)\in [1,\infty)^n$
and $\vec{\alpha}:=(\alpha_1,\ldots, \alpha_n )\in \mathbb{R}^n$.
The \emph{mixed-norm Herz space} $\dot{E}^{\vec{\alpha},\vec{p}}_{\vec{q}}$
is defined to be the set
of all $f\in \mathscr{M}$ having the following finite \emph{quasi-norm}
\begin{align*}
\|f\|_{\dot{E}^{\vec{\alpha},\vec{p}}_{\vec{q}}}
:	&=\left\{\sum_{k_n \in \mathbb{Z}}2^{k_n p_n \alpha_n }
\left[\int_{R_{k_n}}\cdots\left\{\sum_{k_1 \in \mathbb{Z}}2^{k_1 p_1 \alpha_1 }\right.
\right. \right.\\
&\quad\left.\left.\left.\times
\left[\int_{R_{k_1}}|f(x_1,\ldots,x_n)|^{q_1}\,dx_1\right]^{\frac{p_1}{q_1}}\right\}^{\frac{q_2}{p_1}}\cdots
\,dx_n\right]^{\frac{p_n}{q_n}}\right\}^{\frac{1}{p_n}},
\end{align*}
where, for any $i\in \{1,\ldots,n\}$ and $k_i\in \mathbb{Z}$,
$R_{k_i}:= (-2^{k_i}, 2^{k_i})\setminus (-2^{k_-1}, 2^{k_i -1})$.
Applying \cite[Proposition 2.22]{ZYZ22}, we find that the mixed-norm Herz space
$\dot{E}^{\vec{\alpha},\vec{p}}_{\vec{q}}$ with $\alpha_i \in (-\frac{1}{q_i}, 1-\frac{1}{q_i})$
is a {\rm BBF} space.
From \cite[Corollaries 2.19 and 4.7]{ZYZ22} and the definition of absolutely continuous norms,
we deduce that the mixed-norm Herz space $\dot{E}^{\vec{\alpha},\vec{p}}_{\vec{q}}$
satisfies all the assumptions of Theorems \ref{thm-QX},
\ref{thm-in}, \ref{thm-main}, and \ref{thm-cwkx}.
Therefore, Theorems \ref{thm-QX},
\ref{thm-in}, \ref{thm-main}, and \ref{thm-cwkx}
can be applied to $X:=\dot{E}^{\vec{\alpha},\vec{p}}_{\vec{q}}$,
and the obtained results are new; we omit the details.

\begin{remark}
It is quite interesting
to see whether all these results on specific function spaces
in this section can be applied to the study of partial differential equations.
\end{remark}

\section*{Acknowledgment}

The authors would like to sincerely thank all
the anonymous referees for their
carefully reading and several comments which improve the
presentation of the article.
The authors also wish
to express their deep gratitude to Emiel Lorist
for pointing out the aforementioned error 
(see Remark \ref{e}) appeared in Theorem \ref{thm-upi}
in the case where $p=1$ of the original published version.

\bigskip

\noindent Pingxu Hu, Yinqin Li, Dachun Yang (Corresponding author) and
Wen Yuan.

\medskip

\noindent Laboratory of Mathematics and Complex Systems
(Ministry of Education of China),
School of Mathematical Sciences, Beijing Normal University,
Beijing 100875, The People's Republic of China

\smallskip

\noindent {\it E-mails}: \texttt{pingxuhu@mail.bnu.edu.cn} (P. Hu)

\noindent\phantom{\it E-mails }
\texttt{yinqli@mail.bnu.edu.cn} (Y. Li)

\noindent\phantom{\it E-mails }
\texttt{dcyang@bnu.edu.cn} (D. Yang)

\noindent\phantom {\it E-mails }
\texttt{wenyuan@bnu.edu.cn} (W. Yuan)

\end{document}